\documentclass[noinfoline]{imsart}

\RequirePackage[OT1]{fontenc}
\RequirePackage{amsthm,amsmath}
\RequirePackage[numbers]{natbib}
\RequirePackage[colorlinks,citecolor=blue,urlcolor=blue]{hyperref}

\usepackage{a4wide} 

\usepackage{amsfonts, amssymb}
\usepackage{graphicx, color, latexsym}

\theoremstyle{plain} 
\newtheorem{thm}{Theorem}
\newtheorem{prop}{Proposition}
\newtheorem{lem}{Lemma}
\theoremstyle{remark}

{}

\def\1{1\!{\rm l}}

\newcommand{\leqa}{\lesssim}
\newcommand{\geqa}{\gtrsim}

\newcommand{\mf}{m}
\newcommand{\gf}{g}
\newcommand{\bef}{B}
\newcommand{\tmf}{\tilde{m}}

\newcommand{\EM}{\ensuremath}

\newcommand{\al}{\alpha}
\newcommand{\be}{\beta}

\newcommand{\ga}{\gamma}

\newcommand{\La}{\Lambda}

\newcommand{\te}{\theta}
\newcommand{\ta}{\tau}
\newcommand{\veps}{\varepsilon}

\newcommand{\cA}{\EM{\mathcal{A}}}

\newcommand{\cC}{\EM{\mathcal{C}}} 
\newcommand{\cD}{\EM{\mathcal{D}}}

\newcommand{\cL}{\EM{\mathcal{L}}}

\newcommand{\cN}{\EM{\mathcal{N}}}

\newcommand{\cT}{\EM{\mathcal{T}}}

\definecolor{blendedblue}{rgb}{0.2,0.2,0.7}

\newcommand{\sbl}[1]{{\color{blendedblue}{#1}}}

\DeclareMathAlphabet{\mathpzc}{OT1}{pzc}{m}{it}

\newcommand{\RR}{\mathbb{R}}

\newcommand{\given}{\,|\,}

\newtheorem{remark}{Remark}

\newcommand{\mockalph}[1]{}

\newcommand{\eps}{\varepsilon}
\newcommand{\hal}{\hat\al}
\newcommand{\hatha}{\hat\theta_\al}

\begin{document}

\begin{frontmatter}
\title{Spike and slab empirical Bayes sparse credible sets} 
\runtitle{Spike and Slab empirical Bayes credible sets}

\begin{aug}
\author{\fnms{Isma\"el} \snm{Castillo}\thanksref{t3}
\ead[label=e1]{ismael.castillo@upmc.fr}}
\and
\author{\fnms{Botond} \snm{Szab\'o}\thanksref{t1,t2}
\ead[label=e2]{b.t.szabo@math.leidenuniv.nl}}

\address{Sorbonne Universit\'e\\
Laboratoire Probabilit\'es, Statistique et Mod\'elisation\\
4, place Jussieu, 75005 Paris, France\\
\printead{e1}}

\address{Leiden University\\
Mathematical Institute, 
Niels Bohrweg 1\\
2333 CA Leiden, The Netherlands\\
\printead{e2}}

\thankstext{t3}{Work partly supported by the grant ANR-17-CE40-0001-01
of the French National Research Agency ANR (project BASICS).}
\thankstext{t1}{Research supported by the Netherlands Organization for Scientific Research.}
\thankstext{t2}{The research leading to these results has received funding from the European
  Research Council under ERC Grant Agreement 320637.}

\runauthor{I. Castillo and B. Szab\'o}

\affiliation{}

\end{aug}

\begin{abstract}
In the sparse normal means model, coverage of adaptive Bayesian posterior credible sets associated to spike and slab prior distributions is considered. The key sparsity hyperparameter is calibrated via marginal maximum likelihood empirical Bayes. First, adaptive posterior contraction rates are derived with respect to $d_q$--type--distances for $q\le 2$. Next, under a type of so-called excessive-bias conditions, credible sets are constructed that have coverage of the true parameter at prescribed $1-\al$ confidence level and at the same time are of optimal diameter. We also prove that the previous conditions cannot be significantly weakened from the minimax perspective.
\end{abstract}



\end{frontmatter}

\section{Introduction}

\subsection{Setting}

In the sparse normal means model, one observes a sequence $X=(X_1,\ldots,X_n)$
\begin{equation} \label{model}
 X_i = \theta_i + \veps_i,\quad i=1,\ldots,n, 
\end{equation} 
with $\theta=(\theta_1,\ldots,\theta_n)\in\RR^n$ and $\veps_1,\ldots,\veps_n$ 
i.i.d. $\cN(0,1)$. Given $\te$, the distribution of $X$ is a product of Gaussians and is  denoted by $P_\te$. Further, one assumes that the `true' vector $\theta_0$ belongs to 
\begin{equation*}  
 \ell_0[s] = \left\{\te\in\RR^n,\ \ |\{i:\ \te_i\neq0\}|\le s \right\}, 
\end{equation*}
the set of vectors that have at most $s$ nonzero coordinates, where $s$ is a sequence such that $s/n=o(1)$ and $s\to\infty$ as $n\to\infty$. A natural problem is that of reconstructing  $\te$ with respect to the $\ell^q$--type--metric for $0<q\le 2$ (it is a true metric only for $q\le 1$) defined by
\[ d_q(\te,\te')=\sum_{i=1}^n |\te_i-\te_i'|^q.\] 
A benchmark is given by the minimax rate for this loss over the class of sparse vectors $\ell_0[s]$.  The minimax rate over $\ell_0[s]$ for the loss $d_q$ is of the order, as $n\to\infty$, see \cite{djhs92},
\[ r_q:=r_{n,q}=s[\log(n/s)]^{q/2}. \]

The sparse sequence model has become very central in statistics as one of the simplest and natural models to describe sparsity, in a similar way as the Gaussian white noise model in the setting of nonparametrics. Many authors have contributed to its study both from Bayesian and non-Bayesian perspectives, in particular in terms of convergence rates. Some seminal contributions include \cite{BirgeMas01}, \cite{Golubev02}, \cite{fdr06}. Methods using an empirical Bayes approach to study aspects of the posterior distribution  include works by George and Foster \cite{georgefoster}, Johnstone and Silverman \cite{js04} (whose approach we describe in more detail in Sections \ref{subsec-eb}-\ref{subsec-motiv}) and Jiang and Zhang \cite{jiangzhang09}. Works studying the full posterior have started more recently, and we include a brief overview below.      
Here our interest is in a popular class of Bayesian procedures associated to spike and slab prior distributions.  We undertake a so-called frequentist analysis of the posterior distribution. That is, we first construct a prior distribution on the unknown sparse $\te$ and then use the Bayesian framework  to produce a posterior distribution, which is then studied under the frequentist assumption that the data has actually been generated from a `true' unknown sparse parameter $\te_0$. 
 
Our interest is in precise understanding of how posterior distributions for spike and slab priors work for {\em inference} in terms of convergence and confidence sets. Such priors play a central role in statistics, in sparse and non-sparse settings (such as nonparametric function estimation, see e.g. \cite{JohnSilvWave05}), and also as tools for lower bounds. In sparsity contexts, especially for $\ell_0[s]$ classes, they are one of the most natural choice of priors.  
Despite recent advances, there are many open questions regarding mathematical properties of such fundamental priors for inference. A brief overview of the literature on sparse priors is given below. We note also that the present work is a natural continuation of \cite{cm17}, where rates of convergence in the case $q=2$ were investigated. Here we handle the fundamentally different issue of building confidence regions, as well as posterior convergence rates, with respect to $\ell_q$--type--metrics for all $q$ in $(0,2]$. 

The construction of confidence sets is of key importance in statistics, but is a delicate issue. For convenience let us formally denote by $\{\Theta_\beta:\,\beta\in B\}$ a collection of models indexed by some parameter $\beta\in B$ (e.g. sparsity, regularity, dimension, etc.). In practice it is typically unknown which model $\Theta_\beta$ the true $\theta$ belongs to, hence one wants to develop adaptive methods not relying on the knowledge of $\beta$. 
 Constructing adaptive confidence sets in high-dimensional and nonparametric problems is very challenging, in fact impossible in general, see for instance \cite{Low:97,RobinsvdV, gine:nickl:2016} in context of nonparametric models and \cite{nickl:geer:2013} in (sparse) high dimensional problems. Therefore it is sometimes necessary to introduce further assumptions on the models $\Theta_\beta$, $\beta\in B$ to derive positive results, see  \cite{nickl:geer:2013} for more detailed description of the problem in the high-dimensional setting as well as Sections \ref{sec-credible}--\ref{sec-credible-necessity} below.

In various fields of applications, for their flexibility and practical convenience, Bayesian credible sets are routinely used as a measure of uncertainty. However, it is not immediately clear what the frequentist interpretation of these sets is, i.e. whether such sets can be used as confidence sets or whether by doing so one provides a misleading haphazard uncertainty statement. The asymptotic properties of Bayesian credible sets have been investigated only in recent years, see for instance \cite{szabo:vdv:vzanten:13,rousseau:sz:2016,castillo:nickl:2013} and references therein. In the context of sparse high dimensional problems there are only a few results available. For the sparse normal means model, the frequentist coverage properties of a sparsity prior with empirically chosen Gaussian slabs  \cite{belnur15} and of the horseshoe prior  \cite{vsvuq17}  were investigated, while in the more general linear regression model credible sets for the modified Gaussian slab prior are studied in \cite{belitser:ghosal:16}, all for the quadratic risk. In the present paper the focus is on the standard and popular spike and slab prior, which requires 	a substantially different analysis compared to the previous examples, as explained in more details below.

\subsection{Spike and slab prior and associated posterior distribution} 
\label{subsec-post}

The spike and slab prior with sparsity parameter $\alpha$ is the prior $\Pi_\al$ on $\te$ given by 
\begin{equation} \label{prior}
\te \sim \otimes_{i=1}^n \big((1-\al)\delta_0 + \al G(\cdot)\big)=: \Pi_\alpha, 
\end{equation} 
where $\delta_0$ denotes the Dirac mass at $0$ and $G$ is a given
probability measure of density $\ga$. It is often assumed that $\ga$ is a symmetric unimodal density on $\RR$. We will make specific choices in the sequel. 
The posterior distribution under \eqref{model}-\eqref{prior} is 
\begin{equation} \label{postsp}
\Pi_\alpha[\cdot\given X] 
\sim \otimes_{i=1}^n \big((1-a(X_i))\delta_0 + a(X_i) \ga_{X_i}(\cdot)\big),
\end{equation}
where we have set, denoting $\phi$ the standard normal density and 
$\gf(x)=\phi*\ga(x)=\int \phi(x-u)dG(u)$ the convolution of $\phi$ and $G$,  
\begin{align*}
\gf(X_i) & = (\phi*\ga) (X_i),\\
\ga_{X_i}(\cdot) & = \frac{\phi(X_i-\cdot) \ga(\cdot)}{g(X_i)},\\
a(X_i) & = a_\al(X_i)=\frac{\al \gf(X_i)}{(1-\alpha)\phi(X_i) + \alpha \gf(X_i)}.
\end{align*}
If the choice of $\al$ is clear from the context, we denote $a(x)$ instead of $a(\al,x)$ for simplicity. 

{\em Introducing the posterior median threshold.}  For any symmetric $\ga$ density,  the vector $\hat\te_\al$ of medians  of the  coordinates of the posterior \eqref{postsp} (whose $i$th coordinate by  \eqref{postsp} only depends on $X_i$)  has been studied in \cite{js04}. The following property is used repeatedly in what follows, see Lemma 2 in \cite{js04}: the posterior coordinate-wise median  has a thresholding property:  there exists $t(\al)>0$ such that $\hat\te_\al(X)_i=0$ if and only if $|X_i|\le t(\al)$.

\subsection{Empirical Bayes estimation of $\al$ via marginal likelihood}   
\label{subsec-eb} 

In a seminal paper, Johnstone and Silverman \cite{js04} considered estimation of $\te$ using spike and slab priors combined with a very simple empirical Bayes method for choosing $\al$ that we also follow here and describe next. 
The marginal likelihood in $\alpha$ is the density of $X\given \al$ at the observation points in the Bayesian model. A simple calculation reveals that its logarithm equals 
\[ \ell(\al) = \ell_n(\al;X)= \sum_{i=1}^n \log( (1-\al)\phi(X_i) + \al \gf(X_i)). \]
The corresponding score function equals $S(\al):=\ell'(\al)=\sum_{i=1}^n \bef(X_i,\al)$, where 
\begin{equation} \label{eqbeta}
 \bef(x) = \frac{\gf}{\phi}(x)-1;\qquad \bef(x,\al)=\frac{\bef(x)}{1+\al\bef(x)}.
\end{equation} 
Then \cite{js04} define $\hal$ as the maximiser, henceforth abbreviated as MMLE,  of the log-likelihood 
\begin{equation} \label{defhal}
\hal = \underset{\al\in \cA_n}{\text{argmax}}\ \ell_n(\al;X),
\end{equation}
where $\cA_n=[\al_n,1]$, and $\al_n$ is defined by, with $t(\al)$ the posterior median  threshold as above,
\[ t(\al_n)=\sqrt{2\log{n}}. \]

\subsection{Motivating risk results}
\label{subsec-motiv}

 Let as above $\hat\te_\al$ denote the posterior coordinate-wise median associated to the posterior \eqref{postsp} with fixed hyper-parameter $\al$ and let $\hat\te=\hat\te_{\hat\al}$, with $\hat\al$ as in \eqref{defhal}. 
Suppose that   $s=o(n)$ as $n\to\infty$ and that for some  constant $\kappa_1>0$,
\begin{equation} \label{techsn} 
 \kappa_1 \log^2 n \le s.
 \end{equation} 

\noindent {\em Fact 1 (direct consequence of \cite{js04}, Theorem 1)}. 
Let $\ga$ be the Laplace or the Cauchy density. Suppose \eqref{techsn} holds. For any $0<q\le 2$, there exists a  constant $C=C(q,\ga)>0$ such that 
\[ \sup_{\te_0\in\ell_0[s]} E_{\te_0}d_q(\hat\te,\te_0) \le Cr_q, \]
thereby proving minimaxity (up to a constant multiplier) of the estimator $\hat\te=\hat\te_{\hat\al}$ over $\ell_0[s]$. The estimator is adaptive, as the knowledge of $s$ is not required in its construction. Condition \eqref{techsn} is quite mild. In case it is not satisfied, the upper bound on the rate above is $Cr_q+\log^3n$ instead of $Cr_q$, which means there may be a slight logarithmic penalty to use $\hat\te$ in the extremely sparse situation where $s\ll \log^2n$. Theorem 2 in \cite{js04} shows that the estimate $\hat\al$ can in fact be modified so that the minimax risk result holds even if the lower bound in \eqref{techsn} is not satisfied. For simplicity in the present paper we work under \eqref{techsn} but presumably modifying the estimator as in \cite{js04} leads to minimax optimality also in the extremely sparse range in the context of Theorem \ref{thm-risk-dq}.  In this respect, we note that \cite{cm17}, Theorem 5, shows that this is indeed the case when $q=2$ and a Cauchy prior is used.

Consider the plug-in empirical Bayes posterior, for $\hat \al$ the MMLE as defined above,
\[ \Pi_{\hat\al}[\cdot\given X] \sim \otimes_{i=1}^n \big((1-a_{\hat\al}(X_i))
\delta_0 + a_{\hat\al}(X_i) \ga_{X_i}(\cdot)\big).\]

\noindent {\em Fact 2 (\cite{cm17}, Theorems 1 and 3)}
Let $\hat \al$ be the MMLE given by \eqref{defhal}. Let $\ga$ be the Cauchy density.   Under \eqref{techsn} there exists $C>0$ such that, for $n$ large enough, 
\[ \sup_{\te_0\in \ell_0[s]} E_{\te_0} \int d_2(\te,\te_0) d\Pi_{\hat\al}(\te\given X) 
\le Cs\log(n/s).\]
For $\ga$ the Laplace density, the result does not hold: there exists $\te_0\in \ell_0[s]$  and $c>0$ such that 
\[  E_{\te_0} \int d_2(\te,\te_0) d\Pi_{\hat\al}(\te\given X) 
\ge c M_n s\log(n/s), \]
where $M_n=\exp\{\sqrt{\log(n/s)}\}$ goes to infinity with $n/s$. This shows that if tails of the slab prior are not heavy enough, the corresponding posterior does not reach the optimal minimax rate over sparse classes. In particular, typical credible sets such as  balls arising from this posterior will not have optimal diameter. 
These observations naturally lead to wonder if confidence sets in the  squared euclidean norm $d_2=\|\cdot\|^2$ could be obtained using a Cauchy slab, for which the optimal posterior contraction rate is guaranteed, and how the previous facts evolve if $d_q$--type--metrics for $q<2$ are considered.

\subsection{Brief overview of results on sparse priors}
 
Many popular sparse priors can be classified into two categories: first, priors that put some coefficients to the exact zero value, such as spike and slab priors and second, priors that instead draw coefficients using absolutely continuous distributions, and thus do not generate exact zero values. In the first category, one can generalise the spike and slab prior scheme \eqref{prior} by first selecting a random subset $S$ of indexes within $\{1,\ldots,n\}$ and then given $S$ setting $\te_i=0$ for $i\notin S$ and drawing $\te_i$ for $i\in S$ from some absolutely continuous prior distribution. This scheme has been considered e.g. in \cite{cv12}, where the case of an induced prior $\pi_n$ on the number of non-zero coefficients of the form $\pi_n(k)\propto \exp[-c_1 k \log(c_2n/k)]$, called complexity prior, is studied and the slab distribution has tails at least as heavy as Laplace. Belitser and coauthors \cite{belnur15}, \cite{belitser:ghosal:16}, and Martin and Walker \cite{martin:walker:14}, consider the case of Gaussian slabs that are recentered at the observation points. Other proposals for slab distributions include non-local priors as in \cite{johnson:rossell:12}. 
   
In the second category, one can replace the Dirac mass at zero of the spike by a density approaching it, as in the spike and slab LASSO introduced by Ro{\v c}kov\'a and George, see \cite{rockovageorge17}, \cite{rockova17}. One can also directly define a certain continuous density with a lot of mass at zero and heavy tails, as does the horseshoe prior introduced in \cite{carvalhoetal} and further studied in \cite{vkv14}, \cite{vsv17}, \cite{vsvuq17}, see also \cite{salometal} for other families of mixture priors. Other approaches to continuous shrinkage priors include the Dirichlet-Laplace priors of \cite{bpd15}. 
    
Most previously cited works are concerned with posterior convergence rates, with the exception of \cite{belnur15}--\cite{belitser:ghosal:16} (that considers also oracle results) and \cite{vsvuq17}, that derive properties of credible sets, all with respect to the squared $\ell^2$--loss. The prior and confidence sets introduced in \cite{belnur15} are quite different from those considered here, in that, for instance, the radius of credible set we  consider is determined directly from the posterior distribution, and the priors and confidence sets in \cite{belnur15} require some post-processing (e.g. specific separate estimation of the radius and recentering of the  posterior selected components). As noted above, the horseshoe prior belongs to a different category of priors, not setting any coefficient to $0$, and further, it is not clear if its Cauchy tails would be sufficiently heavy to handle $d_q$--losses for small $q$, at least via a MMLE--empirical Bayes choice of its tuning parameter $\ta$. An overview of current research can be found in the discussion paper \cite{vsvuq17}. 

\subsection{Outline and notation}

\noindent {\em Outline and summary of main results.} Section \ref{sec:main} contains our main results. First, adaptive convergence rates in $d_q$--type--distances, $0<q\le 2$, are derived for the full empirical Bayes posterior $\Pi_{\hat\al}[\cdot\given X]$, for a well--chosen slab distribution. Second, frequentist coverage results are obtained for credible balls centered at the posterior median estimator $\hat\te$ and whose radius is a constant $M$ times the posterior expected radius $\int d_q(\te,\hat\te)d\Pi(\te\given X)$, both for deterministic and data-driven choice of the hyper-parameter $\al$. In the later case, we prove that under an excessive--bias condition, the credible sets have optimal diameter and frequentist coverage, already for fixed large enough $M$ (so without the need of a `blow-up' $M=M_n\to\infty$). Focusing on the case $q=2$, we then discuss the obtained excessive--bias condition  and show that such a condition cannot be weakened from the minimax perspective. 
Section \ref{sec:discussion} briefly  discusses the main findings of the paper. Proofs are organised as follows: Section \ref{sec:proofs-generalities} regroups some useful preliminary bounds,  Section \ref{sec:proofs:cs} is devoted to proofs for credible sets. A separate supplementary material \cite{castillo:szabo:2018:supplement} gathers proofs of technical lemmas, as well as the proof of the rate Theorem \ref{thm-risk-dq}.  \\

\noindent {\em Notation.} 
For two sequences $a_n,b_n$ let us write $a_n\lesssim b_n$ if there exists a universal constant $C>0$ such that $a_n\leq C b_n$, and  $a_n\asymp b_n$ if $a_n\lesssim b_n$ and $b_n\lesssim a_n$ hold simultaneously. We write $a_n\sim b_n$ for $b_n\neq 0$ if $a_n/b_n=1+o(1)$. We denote throughout by $c$ and $C$ universal constants whose value may change from line to line.
Also, the $d_q$--diameter of a set $\cC$ is written $diam_q(\cC)$, i.e.
\begin{align*}
diam_q(\cC)=\sup_{\theta,\theta'\in \cC}d_q(\theta,\theta').
\end{align*}
For convenience in the  case $q=2$, we denote by $diam(\cC)$ the $d_2$--diameter of the set $\cC$. Finally, when $\cD$ is a finite set, $|\cD|$ denotes the cardinality of $\cD$.

\section{Main results} \label{sec:main}

\subsection{Slab prior distributions}

For a fixed $\delta\in(0,2)$, consider the unimodal symmetric density $\gamma=\gamma_\delta$ on $\RR$ given by 
\begin{equation} \label{dens-hv}
\ga(x)=\frac{c_\ga}{(1+|x|)^{1+\delta}} =  c_\ga\Delta(1+|x|), \quad \text{for }\Delta(u)=u^{-1-\delta}, 
\end{equation} 
where $c_\ga=c_\ga(\delta)$ is the normalising constant making $\ga$ a density. The purpose of this density is to have sufficiently heavy tails, possibly heavier than Cauchy. To fix ideas, we take the specific form \eqref{dens-hv}, but as is apparent from the proofs, similar results continue to hold for densities with same tails: for instance, the Cauchy density could be used instead of $\ga_2$. The possibility of having a broad range of heavy tails is essential to achieve optimal rates in terms of $d_q$ for the considered empirical Bayes procedure. If $\delta\ge 1$, the function $u\to (1+u^2)\ga(u)$ is bounded and the density $\ga$ falls in the framework of \cite{js04}. If $\delta\in(0,1)$, we show below how this changes estimates quantitatively. In all cases, $\ga$ still satisfies
\[ \sup_{u>0} \left| \frac{d}{du} \log \ga(u) \right|=: \La<\infty. \]
Recall $\gf=\phi*\ga$ is the convolution of the heavy-tailed $\gamma$ given by \eqref{dens-hv} and the noise density $\phi$. 
Basic properties of $\gf$ are gathered in Lemma \ref{lem-gf}, while Lemma \ref{lem-momgf} provides bounds on corresponding moments of the score function.

\subsection{Adaptive risk bounds for integrated posterior}

\begin{thm} \label{thm-risk-dq}
Fix $\delta\in(0,2)$. Let $\ga=\ga_\delta$ be the density defined by \eqref{dens-hv} and let $\hat \al$ be the corresponding MMLE given by \eqref{defhal} and  suppose \eqref{techsn} 
 holds. Then there exists a universal constant $C>0$, that in particular is independent of $\delta, q$,  such that, for large enough $n$, for any $q\in (2\delta,2]$,  
\[ \sup_{\te_0\in \ell_0[s]} E_{\te_0} \int d_q(\te,\te_0) d\Pi_{\hat\al}(\te\given X) 
\le Cs\log^{q/2}(n/s).\] 
\end{thm}

Theorem \ref{thm-risk-dq} shows that the posterior $q$th moment converges at the minimax rate for $d_q$--type--distances over $\ell_0[s]$. By contrast, note that the results in \cite{cv12} for $d_q$ covered complexity priors on the dimension, but  not   spike and slab priors (which induce a binomial prior on the dimension), and were results on the posterior convergence as a probability measure and as such did not  imply convergence at minimax rate of e.g. the posterior mean. The proof of Theorem \ref{thm-risk-dq} is given in the supplement, Section \ref{supp-th1}.

Let us now briefly comment on the behaviour of some point estimators and on simulations from the empirical Bayes posterior. Under the conditions of Theorem \ref{thm-risk-dq}, the posterior mean is rate-minimax  for any $1\le q\le 2$. This follows from Theorem  \ref{thm-risk-dq} using the convexity of $\te\to d_q(\te,\te_0)$ if  $1\le q\le 2$ and Jensen's inequality.  More details on the posterior mean, in particular its suboptimality when $q<1$, can be found in the supplement, Section \ref{supp-pointest}. Concerning the posterior median, one can check that it is rate-minimax for any $0<q\le 2$, see Section \ref{supp-pointest} in the supplement. 
We also note in passing that simulation from the considered empirical Bayes posterior distribution is fast: for Cauchy-type slab tails, one can  directly use the EbayesThresh package of Johnstone and Silverman, see \cite{js05}. To compute $\hat\al$ corresponding to the precise slab form $\ga$ in \eqref{dens-hv}, one can compute approximations of $g(x)=(\ga*\phi)(x)$ by a numerical integration method and next insert this in the EbayesThresh subroutine computing $\hat\al$.  \\
   
\begin{remark}   
One may consider a density $\ga$ `on the boundary' by setting, say,
\[ \Delta(u) = u^{-1}\log^{-2}u. \]
For this choice of $\ga$, it can be checked that the risk bound of Theorem \ref{thm-risk-dq} holds uniformly for $q\in(0,2]$. However, this prior density has some somewhat undesired properties for confidence sets: it can be checked that the variance term of the empirical Bayes posterior is, for $q=2$, of the order
$s\ta(\al)^2/\log\ta(\al)$, which turns out to be sub-optimally small and a blow-up factor of order at least $\log\log(n/s)$ would be needed to guarantee coverage of the corresponding credible set.
\end{remark}

\subsection{Credible sets for fixed $\al$}

For $q\in(0,2]$, and as before $\hat\te_\al$ the posterior coordinate-wise median for fixed $\al$, we set
\begin{equation} \label{def: credible_mmle_q_fixed}
 \cC_{q,\al} = \{\te\in \RR^n,\ d_q(\te , \hatha) \le M v_{q,\al}(X)\}, 
\end{equation} 
where $M$ is a constant to be chosen below, and where we denote  
\[ v_{q,\al}(X)  = \int  d_q(\te , \hatha) d\Pi_{\al}(\te\given X). \]
Note that by Markov's inequality, for $M\geq 1/\beta$ it holds
\[ \Pi_{\al}[ \cC_{q,\al}  \given X]\ge 1-\beta,\]
so that $\cC_{q,\al}$ is a $1-\beta$ credible set (actually it is sufficient to take $M=1+\eps$, for arbitrary $\eps>0$, to achieve $1-\beta$ posterior coverage asymptotically, see Remark \ref{rem: cred:alternative} below).  
The proposition below reveals the frequentist properties of the so--constructed credible sets for a fixed value of the tuning parameter $\al$. Taking $\alpha\lesssim s\log^{\delta/2}(n/s)/n$, the {\em size} of the credible set is (nearly) optimal, reaching the (nearly) minimax rate $s\log^{q/2}(n)$, and by taking $\alpha\asymp s\log^{\delta/2}(n/s)/n$ the exact minimax rate (up to a constant) $s\log^{q/2}(n/s)$ is achieved.
On the other hand, the frequentist {\em coverage} properties of $\cC_{q,\al}$ behave in an opposite way with respect to $\alpha$. Indeed, one can find elements of the class $\ell_0[s]$ for which too small choice of the hyperparameter $\alpha$, i.e. $\alpha=o\big(s\log^{\delta/2}(n/s) /n\big)$, results in misleading uncertainty statements. At the same time sufficiently large values of $\alpha$ (i.e. $\alpha\gtrsim s\log^{\delta}(n/s)/n$) provide high frequentist coverage. Let us  introduce the set
\begin{equation} \label{def:tildeT}
\tilde{\Theta}_{s,\alpha}=\Big\{ \, \theta\in\ell_0[s]:\  |\left\{\,i:\, t(\alpha)/8\leq|\theta_{0,i}|\leq   t(\alpha)/4\,\right\}| = s \, \Big\}.
\end{equation}
Note that this sets contains non-zero signals with large enough (but not too large) values.

\begin{prop}\label{thm: coverage:nonadapt}
Let $\delta\in(0,2)$ be arbitrary and $\Pi_\al$ be the spike and slab prior with $\ga=\ga_\delta$ the density defined by \eqref{dens-hv}. 
Then for any $q>\delta$ and $s\log^{\delta/2}(n/s)/n\lesssim \al \le \al_1$ for sufficiently small constant $\al_1$, the Bayes credible set \eqref{def: credible_mmle_q_fixed} has, with respect to $d_q$, frequentist coverage tending to one for some sufficiently large choice of $M$ 
\begin{align*}
\sup_{\theta_0\in \ell_0[s]} P_{\theta_0} \big(\te_0\in \cC_{q,\al}\big)\rightarrow1.
\end{align*}
However, for $ \alpha=o\big(s\log^{\delta/2}(n/s)/n\big)$ the credible set  has frequentist coverage tending to zero for true signals $\te_0$ in the set $\tilde{\Theta}_{s,\al}$ defined in \eqref{def:tildeT}, for arbitrary choice of $M>0$, i.e. 
$$\sup_{\theta_0\in\tilde{\Theta}_{s,\al}}P_{\theta_0} \big(\te_0\in \cC_{q,\al}\big)\rightarrow0.$$ 
\end{prop} 
\noindent   The next Proposition shows that the region of $\alpha$'s where the diameter of the fixed $\alpha$--credible set is optimal is in a sense `reversed'. Both results are proved in Section \ref{sec: coverage:nonadapt}.
\begin{prop}\label{thm: radius:nonadapt}
Let $\delta\in(0,2)$ be arbitrary and $\Pi_\al$ be the spike and slab prior with $\ga=\ga_\delta$ the density defined by \eqref{dens-hv}. 
Then for any $\delta<q\le 2$ and for any  $s\log^{\delta/2}(n/s)/n\ll  \al  \le \al_1$ for sufficiently small constant $\al_1$, the Bayes credible set \eqref{def: credible_mmle_q_fixed} has, with respect to $d_q$, suboptimal diameter 
\begin{align*} 
\inf_{\te_0\in\ell_0[s]} 
E_{\theta_0}\left[\text{diam}_q(\cC_{q,\al})\right] \gg s\log^{q/2}(n/s).
\end{align*}
However, if $(s/n)^{c_1}\lesssim \alpha\lesssim (s/n)\log^{\delta/2}(n/s)$ for some $c_1\ge 1$,  the credible set  has optimal diameter
$$\sup_{\theta_0\in\ell_0[s]}E_{\theta_0}\left[\text{diam}_q(\cC_{q,\al})\right] \leqa s\log^{q/2}(n/s).$$
\end{prop} 

\begin{remark}\label{rem: cred:alternative}
One can consider other types of credible sets as well, for instance balls centered around the posterior coordinate-wise median, i.e.
\begin{align}
\tilde{C}_{q,\alpha}=\left\{\theta\in\mathbb{R}^n,\, d_q(\theta,\hat{\theta}_\alpha)\leq r_\beta \right\}, \quad \text{with $r_\beta$ taken as}\quad \Pi_\alpha (\tilde{C}_{q,\alpha}|X )=1-\beta\label{def: credible:standard}
\end{align}
(if the equation has no solution, one takes the smallest $r_\be$ such that $\Pi_\alpha (\tilde{C}_{q,\alpha}|X )\ge 1-\beta$). 

One can show that these two types of credible sets are the same up to a $(1+o(1))$ blow-up factor for every fixed $0<\beta<1$, since 
\begin{align}
r_{\beta}=(1+o(1))v_{q,\alpha}(X),\label{eq: equiv:nonadapt}
\end{align}
for all $\alpha\in(M_n (\log_2 n)^{\delta/2}/n,\alpha_1)$, with $\alpha_1>0$ a small enough constant and $M_n\rightarrow \infty$ arbitrarily slowly. The proof of this statement  is given in Section \ref{sec: cred:alternative} of the supplement.

Therefore, by inflating the credible set \eqref{def: credible:standard} by a sufficiently large constant factor $L$, it has frequentist coverage tending to one for $s\log^{\delta/2}(n/s)/n\lesssim \alpha\leq\alpha_1$, that is, for $\tilde{C}_{q,\alpha}(L)=\{\theta\in\mathbb{R}^n,\, d_q(\theta,\hat{\theta}_\alpha)\leq Lr_\beta \}$, we have
\begin{align*}
\sup_{\theta_0\in\ell_0[s]}P_{\theta_0}(\theta_0\in \tilde{C}_{q,\alpha}(L))\rightarrow 1.
\end{align*}
However, for $ \alpha=o\big(s\log^{\delta/2}(n/s)/n\big)$ the credible set  has frequentist coverage tending to zero for true signals $\te_0$ in the set $\tilde{\Theta}_{s,\al}$.
\end{remark}

\subsection{Adaptive credible sets for $q=2$} \label{sec-credible}

In this section we investigate the adaptive version of the credible set $\mathcal{C}_\al$ introduced in \eqref{def: credible_mmle_q_fixed}, in the case $q=2$.
Define the random set, for $M\ge 1$ to be chosen,
\begin{align} 
 \cC_{\hat\al} = \cC_{2,\hat\al} = \Big\{\, \te\in \RR^n,\ \|\te - \hat\te_{\hat\al}\|^2 \le M v_{\hat\al}(X)\, \Big\}, \label{def: credible_mmle}
\end{align}
where   $\|\cdot\|^2=d_2$ is the square of the standard euclidian norm and
\[ v_{\hat\al}(X)  = \int \| \theta - \hat\theta_{\hat\al} \|^2 d\Pi_{{\hat\al}}(\te\given X). \]
By Markov's inequality the set $\cC_{\hat\al}$ has at least $1-\beta$ posterior coverage for $M\geq 1/\beta$. Also, it is a direct consequence of Theorem \ref{thm-risk-dq}, Markov's inequality and the rate optimality of the posterior median estimator that the size of this sets adapts to the minimax rate: 
for every $\epsilon>0$ there exists $M_{\epsilon}>0$ such that for any $\theta_0\in \ell_0[s]$, 
$$P_{\theta_0}( v_{\hat\al}(X)\geq M_{\epsilon}s\log(n/s))\leq \epsilon.$$
So the credible set has an optimal diameter uniformly. 
However, from similar arguments as in \cite{nickl:geer:2013}, this means that the present credible set cannot have honest coverage for every sparse $\theta_0$, since the construction of adaptive and honest confidence sets for the quadratic risk is impossible in the sparse normal means model, see the Supplement \cite{castillo:szabo:2018:supplement} for a precise statement and proof.

To achieve good frequentist coverage one has to introduce certain extra assumptions on the parameter set $\ell_0[s]$. We consider the excessive-bias restriction investigated in the context of the sparse normal means model in \cite{belnur15,vsv17}, i.e.
we say that $\te_0\in\ell_0[s]$ satisfies the \emph{excessive-bias restriction} 
 for constants $A>1$ and  $C_2,D_2>0$, if there exists an integer $s\ge \ell\ge \log^2 n$, with
\begin{align}\label{condition: EB}
\sum_{i: |\te_{0,i}|< A\sqrt{2\log (n/\ell)}}\te_{0,i}^2\le D_2 \ell\log (n/\ell),
\qquad \left| \big\{ i: |\te_{0,i}|\geq A\sqrt{2\log (n/\ell)}\big\}\right|  \geq \frac{\ell}{C_2}.
\end{align}
We denote the set of all such vectors $\te_0$ by $\Theta_0^2[s]=\Theta_0^2[s;A,C_2,D_2]$, and let $\tilde s=\tilde s (\te_0)$ be $\big| \bigl\{i: |\te_{0,i}|\geq A\sqrt{2\log (n/\ell)}\bigr\}\big|$, for the smallest possible $\ell$ such that \eqref{condition: EB} is satisfied. We note that the assumption $\ell\geq \log^2 n$ can be relaxed to $\ell\geq 1$ by considering a modified MLE estimator, as discussed below assumption \eqref{techsn}. However for the sake of simplicity and better readability we work under the assumption $\ell\geq\log^2 n$. The necessity of condition \eqref{condition: EB} is investigated in Section \ref{sec-credible-necessity}. 

By definition $\tilde{s}\leq s$ and possibly, if $\te_0$ has many small coefficients, one can have $\tilde{s}=o(s)$.
 We call the quantity $\tilde{s}$ the {\it effective sparsity} of $\te_0\in\ell_0[s]$. 
  It shows the number of large enough signal components which can be distinguished from the noise. The rest of the signals are too small to be detectable, but at the same time their energy (the sum of their squares) is not too large so the bias of standard estimators (which will shrink or truncate the observations corresponding to small signals) won't be dominant. Our goal is to adapt to the present effective sparsity value and at the same time have appropriate frequentist coverage for the credible sets.

\begin{thm}\label{thm: coverage_mmle}
Let $\delta\in(0,2)$ be arbitrary and let $\Pi_\al$ be the spike and slab prior with $\ga=\ga_\delta$ the density defined by \eqref{dens-hv}. Let $\hat \al$ be the corresponding MMLE given by \eqref{defhal} and  suppose that the excessive-bias condition \eqref{condition: EB} hold. Then the MMLE empirical Bayes credible set \eqref{def: credible_mmle} has, with respect to the $d_2$--type distance ($q=2$),  adaptive size to the effective sparsity $\tilde{s}$ and frequentist coverage tending to one, i.e. for any $s=o(n)$
\begin{align} 
\inf_{\theta_0\in \Theta_0^2[s;A,C_2,D_2]}P_{\theta_0}\big(diam (\cC_{\hat\al})\leq C  \tilde{s} \log (n/\tilde{s}) \big)\rightarrow 1,\\
\inf_{\theta_0\in \Theta_0^2[s;A,C_2,D_2]}P_{\theta_0} \big(\te_0\in \cC_{\hat\al}\big)\rightarrow1,
\end{align}
for sufficiently large constants $C,M>0$  (the latter in \eqref{def: credible_mmle}), depending on $A, C_2$ and $D_2$ in the excessive bias condition.
\end{thm}
This result is a particular case of Theorem \ref{thm: coverage_mmle_q} below, whose proof is given in Section  \ref{sec: coverage_mmle_q}. 
Similar to the risk results presented in Section \ref{subsec-motiv}, choosing a Laplace slab would lead to suboptimal diameter of the confidence sets. A heavy tail slab is crucial when following an empirical Bayes method for estimating $\alpha$ in the spike and slab prior.

\subsection{Adaptive credible sets: extension to the case $q<2$} \label{sec-credible-q}

Let $q\in(0,2]$ and let us start by defining an analogous condition to \eqref{condition: EB} in order to control the bias of the posterior.  We say that $\te_0\in\ell_0[s]$ satisfies the \emph{$d_q$-excessive-bias restriction}, in short $EB(q)$, for constants $A>1$ and  $C_q,D_q>0$, if there exists an integer $s\ge \ell\ge \log^2 n$ with
\begin{align}\label{condition: EB_q}
\sum_{i: |\te_{0,i}|< A\sqrt{2\log (n/\ell)}}|\te_{0,i}|^q\le D_q \ell\log^{q/2}(n/\ell),
\qquad \Bigl| \{ i: |\te_{0,i}|\geq A\sqrt{2\log (n/\ell)}\} \Bigr|  \geq \frac{\ell}{C_q}.
\end{align}
The set of all such vectors $\te_0$ is denoted $\Theta_0^q[s]=\Theta_0^q[s;A,C_q,D_q]$. For any $q'\in(q,2]$,
\begin{align}
 \Theta_0^q[s;A,C_q,D_q]
 \subset \Theta_0^{q'}[s;A,C_q,D_q(\sqrt{2}A)^{q'-q}]. \label{eq: subset:EB}
\end{align}
This means that up to a change in the constants, the $EB(q)$ condition becomes stronger when $q$ decreases. Let $\tilde s_{q}=\tilde s_q (\te_0)$ be $\big|\bigl\{i: |\te_{0,i}|\geq A\sqrt{2\log (n/\ell)}\bigr\}\big|$, for the smallest possible $\ell$ such that \eqref{condition: EB_q} is satisfied. 

Next we define the random set, for any $q\in(0,2]$ and $M\ge 1$ to be chosen,
\begin{align}
 \cC_{q,\hat\al}  = \Big\{\, \te\in \RR^n,\ d_q(\te,\hat\te_{\hat\al}) \le M v_{q,\hat\al}(X)
 \, \Big\},\label{def: credible_mmle_q}
\end{align}
where $v_{q,\hat\al}(X)  = \int d_q(\theta,\hat\theta_{\hat\al}) d\Pi_{{\hat\al}}(\te\given X)$. 
By Markov's inequality this set has at least $1-\beta$ posterior coverage for $M\geq 1/\beta$. Once again, 
from Theorem \ref{thm-risk-dq} and   the optimality of the posterior median estimator in $d_q$, the size of these sets adapts to the minimax rate:
 for every $\epsilon>0$, there exists $M_{\epsilon}>0$ such that for any $\theta_0\in \ell_0[s]$,
$$P_{\theta_0}( v_{q,\hat\al}(X)\geq M_{\epsilon}s\log^{q/2}(n/s))\leq \epsilon.$$

\begin{thm}\label{thm: coverage_mmle_q}
Fix $\delta\in(0,2)$ and let $q\in (2\delta,2]$. Let $\ga=\ga_\delta$ be the density defined by \eqref{dens-hv}. Let $\hat \al$ be the corresponding MMLE given by \eqref{defhal} and  suppose \eqref{condition: EB_q} holds. Then the MMLE empirical Bayes credible set \eqref{def: credible_mmle_q} for sufficiently large $M$ ($M>3c_0 (2^qD_qC_q+1)(2^{q-1}\vee1 )2^6 (q-\delta)/\delta$ is sufficiently large, where $c_0$ is given in Lemma \ref{lem: effective:sparsity}),
 has  adaptive size to the effective sparsity $\tilde{s}_{q}$ and frequentist coverage tending to one, i.e. for any $s=o(n)$
\begin{align}
\inf_{\theta_0\in \Theta_0^q[s;A,C_q,D_q]} P_{\theta_0}\big(diam_q(\cC_{q,\hat\al})\leq C  \tilde{s}_{q} \log^{q/2} (n/\tilde{s}_{q}) \big)\rightarrow 1,\\
\inf_{\theta_0\in \Theta_0^q[s;A,C_q,D_q]}  P_{\theta_0} \big(\te_0\in \cC_{q,\hat\al}\big)\rightarrow1,
\end{align}
for some sufficiently large constant $C>0$ (depending on $A, C_q$ and $D_q$) 
%in the excessive bias condition).
\end{thm}
The proof of this result is given in Section \ref{sec: coverage_mmle_q}.

\begin{remark}
The results of Theorem \ref{thm: coverage_mmle_q} are uniform over $q\in(2\delta,2)$ provided $M$ is larger than $\sup_{q\in(2\delta,2)}3c_0 (2^qD_qC_q+1)(2^{q-1}\vee1 )2^6 (q-\delta)/\delta$ as seen in the proof of Theorem \ref{thm: coverage_mmle_q}.
\end{remark}
\begin{remark}
In Lemma \ref{lem: effective:sparsity} we shall see that, under condition $EB(q)$, the parameters $\tilde{s}$ and $\tilde{s}_{q}$ are equivalent up to a constant multiplier,  hence the result above also holds with the effective sparsity $\tilde{s}$ corresponding to the parametrisation $\Theta_0^2(s;A,C_q,D_q(\sqrt{2}A)^{2-q})$.
\end{remark}
\begin{remark}\label{rem: equiv:adapt}
We note that the same results as in Theorem \ref{thm: coverage_mmle_q} hold for the inflated version of the credible set define in \eqref{def: credible:standard} as well. The proof is deferred to Section \ref{sec: cred:alternative} of the supplementary material.
\end{remark}
\subsection{Excessive bias conditions: comparison and minimax necessity}
\label{sec-credible-necessity}

In this section for simplicity we restrict the discussion to the case $q=2$. Let us briefly summarise the results obtained in \cite{nickl:geer:2013}, where the authors work in the random design sparse regression model. If $s$ is of smaller order than $\sqrt{n}$, the authors in \cite{nickl:geer:2013} show, see their Theorem 4 part (A), that construction of adaptive and honest confidence sets is impossible (strictly speaking this result is for the regression model, and for completeness we derive a sequence model version of it in the Supplement \cite{castillo:szabo:2018:supplement}). They also show, see their Theorem 4 part (B), that if one cuts out part of the parameter set, thus obtaining a certain {\em slicing} formulated in terms of a certain separation (or `testing') condition, adaptive confidence sets do again exist. 
If one knows beforehand that one deals with a moderately sparse vector, for which $s$ is of larger order than $\sqrt{n}$, then construction of adaptive confidence sets is possible as well, but requires a different procedure than in the highly sparse case under the testing condition, see e.g. Theorem 1 in \cite{nickl:geer:2013}.  
 
First we compare the excessive-bias condition with the testing condition introduced in \cite{nickl:geer:2013} adapted to the sparse sequence model (of course they work on a somewhat different model, but on the same parameter space $\ell_0[s]$). The testing condition was originally given for two sparseness classes $\ell_0[s_1]$ and $\ell_0[s_2]$ for some $s_1\leq s_2\wedge n^{1/2}$  and it was shown in Theorems 3 and 4 of \cite{nickl:geer:2013} that constructing adaptive and honest confidence sets is possible when restricting true signals to  the set
\begin{align*}
\ell_0[s_1]\cup  \cT[s_1,s_2;c]\quad\text{with}\quad   \cT[s_1,s_2;c] := \left\{\theta\in\ell_0[s_2]:\, \|\theta-\ell_0[s_1]\|^2_*
\geq c [n^{1/2}\wedge (s_2\log n)]\right\},
\end{align*}
for some large enough constant $c>0$, where in the setting of \cite{nickl:geer:2013} the loss is $\|\cdot\|_*^2=n\times d_2$, while here we take $\|\cdot\|_*^2=\|\cdot\|^2=d_2$.

This condition can be extended to cover every sparsity class up to a certain level $s$ (possibly $s=n$) for instance by introducing the dyadic partition $s_i=2^{i}$, $i=1,2,...,\lfloor \log_2 s \rfloor$ and formulating the testing condition on every consecutive sparsity class on this grid. A similar type of dyadic partitioning was introduced in \cite{bull:2013} in the nonparametric regression and density estimation for H\"older smoothness classes. Set, for given $c>0$ and $0\le s\le n$,
\begin{align*}
\cT_d[s;c] :=\bigcup_{i=1}^{\lfloor\log_2 s\rfloor -1} \cT[s_i,s_{i+1};c].
\end{align*}
Then one can construct adaptive and honest confidence sets on the set $\cT_d[s; c]$ provided $c$ is large enough, see for instance the closely related result in context of the nonparametric regression model in \cite{bull:2013}.  If a vector $\te\in\RR^n$ belongs to $\cT_d[s;c]$ for some $s,c$, we say that it satisfies the {\em testing condition}. The next Lemma, whose proof can be found in Section \ref{sec: proof:eb},  shows that for well-chosen constants $A, C_2, D_2>0$ the excessive-bias condition is a weaker condition than the testing condition for sparsities $s=o(\sqrt{n})$ (up to a log factor).\begin{lem} \label{lem:compa}
Let $c>0$ and $s\le s_{max}=
\lfloor \sqrt{n}/(c\log{n})\rfloor$. For $1\le s_{max}\le n/e$, we have
\[ \cT_d[s;c] \, \subset \, \Theta_0^2[s;\sqrt{c/2},1,c]. \]
Further if $c\log{(n/s_{max})}>1$, we have the strict inclusion $\cT_d[s;c] \, \subsetneq \, \Theta_0^2[s;\sqrt{c/2},1,c]$.
\end{lem}

Next we show that the slicing of the parameter space induced by the excessive-bias condition in some sense cannot be weakened to construct adaptive confidence sets even between two sparsity classes. To do so, we proceed in a similar way as for the testing condition in \cite{nickl:geer:2013}, and consider three different types of weakening of the excessive-bias assumption. Let $m_n$ denote a sequence tending to zero arbitrary slowly. First we relax the upper bound on the energy of the small signal component from $C q \log (n/\ell)$ to $m_n^{-1}q \log (n/\ell)$, second we relax the lower bound on the number of signal components above the detection boundary from $\ell/C_2$ to $m_n \ell$ and third we relax the detection threshold  $A\sqrt{2\log (n/\ell)}$ to $m_n\sqrt{2\log (n/\ell)}$. More formally the three different  relaxations are 
\begin{align}
\sum_{i: |\te_{0,i}|< A\sqrt{2\log (n/\ell)}}\te_{0,i}^2\le m_n^{-1} \ell \log (n/\ell),
\qquad \big| \bigl\{i: |\te_{0,i}|\geq A\sqrt{2\log (n/\ell)}\bigr\}\big|\geq \ell/C_2.\label{condition: EB_weaker1}\\
\sum_{i: |\te_{0,i}|< A\sqrt{2\log (n/\ell)}}\te_{0,i}^2\le D_2 \ell \log (n/\ell),
\qquad  \big|\big\{ i: |\te_{0,i}|\geq A\sqrt{2\log (n/\ell)}\bigr\}\big|\geq m_n \ell,\label{condition: EB_weaker2}\\
\sum_{i: |\te_{0,i}|< m_n\sqrt{2\log (n/\ell)}}\te_{0,i}^2\le D_2 \ell \log (n/\ell),
\qquad \big| \bigl\{i: |\te_{0,i}|\geq m_n\sqrt{2\log (n/\ell)}\bigr\}\big|\geq \ell/C_2.\label{condition: EB_weaker3}
\end{align}
Theorem \ref{thm: info:theoretic:bounds} below, whose proof is given in Section \ref{sec: proof:eb}, shows that under neither of these relaxations is it possible to construct adaptive confidence sets. 
\begin{thm}\label{thm: info:theoretic:bounds}
Take any $L>0$, $s_2=n^{1/2-\eps}$, for some $\eps>0$, and $s_1=m_n s_2$, for some $m_n=o(1)$. Then under neither of the weaker excessive-bias condition \eqref{condition: EB_weaker1} or  \eqref{condition: EB_weaker2} or \eqref{condition: EB_weaker3} (each of them denoted by $\Theta_0$ for simplicity) with $A>1
$ and $C_2, D_2>0$,  exists a confidence set $\cC_n(X)$ satisfying simultaneously for $i=1,2$ that
\begin{align} 
\varliminf_n\ \inf_{\theta_0\in \Theta_0\cap \ell_0[s_i]}\ P_{\theta_0}(\theta_0\in\cC_n(X) )
\, & \geq \, 1-\beta,\label{eq: eb:coverage}\\
\varliminf_n\ \inf_{\theta_0\in \Theta_0\cap \ell_0[s_i]}\ P_{\theta_0}\left(\text{diam}(\cC_n(X)) \leq L s_i\log (n/s_i)\right)\, & \geq \, 1-\beta',\label{eq: eb:adapt}
\end{align}
for some $\beta,\beta'\in(0,1/3)$.
\end{thm}

The above result shows that, if one slices the parameter space according to an excessive-bias condition, the slicing cannot be refined by making constants arbitrarily smaller: one cannot construct adaptive confidence sets for the resulting  larger parameter set. In \cite{nickl:geer:2013}, a similar result is shown for the testing condition above. The slicing induced by the excessive-bias condition is a bit more general as indicated by Lemma \ref{lem:compa}: it has  more flexibility given that it depends also on more parameters. While the impossibility of the weakening \eqref{condition: EB_weaker3} can proved appealing to as similar proof as for the testing condition in \cite{nickl:geer:2013}, the other two weakenings correspond to slicing the space in different directions and require a completely new proof. 

Theorem \ref{thm: info:theoretic:bounds} can be interpreted as showing the optimality of the slicing within the excessive-bias scale, in the highly sparse regime $s=o(\sqrt{n})$, where adaptive confidence sets do not exist without further assumptions on the space. It could be also interesting to consider the dense regime where one knows beforehand that $s$ is of larger order than $\sqrt{n}$, although it is a qualitatively different question which is not considered here from the optimality perspective. In that case, a different empirical Bayes choice of the sparsity parameter could presumably be used,  in a similar spirit as \cite{szabo2015} in context of the Gaussian white noise model, enabling the construction of adaptive confidence sets from corresponding posteriors, but this is beyond the scope of the present paper.

\section{Discussion} \label{sec:discussion}

In the paper we show that the empirical Bayes posterior distribution corresponding to the spike and slab prior, with heavy enough slab tails, results in optimal recovery and reliable uncertainty quantification in $\ell_q$-type-norm, $q\in(0,2]$, under the excessive-bias assumption. We have further shown that the excessive-bias assumption is optimal in a minimax sense for $s=o(\sqrt{n})$ and $\ell_2$-norm. A natural extension of the derived results could be to consider hierarchical Bayes methods. Relatively similar results are expected, but computationally simulating from the posterior can be more involved.

We note that the derived contraction and coverage results heavily depend on the choice of the slab prior. The empirical Bayes procedure with Laplace slabs results in sub-optimal contraction rate and therefore too conservative credible sets in $\ell_2$-norm, see \cite{cm17}. Therefore, to achieve optimal recovery of the truth one has to use slab priors with polynomial tails. Considering $\ell_q$-type-metrics, $q\in(0,2]$, one has to carefully choose the order of the polynomial, for instance for $q\in(0,1)$ sub-Cauchy tails have to be applied. 
In view of Propositions \ref{thm: coverage:nonadapt} and \ref{thm: radius:nonadapt} one can see that the optimal choice of mixing hyper parameter $\alpha$ in terms of rates and coverage is  $\alpha\asymp (s/n)\log^{\delta/2}(n/s)$, for $\delta< q$. Also note that we have looked at a special excessive-bias type slicing with specific effective sparsity definition; but the `effective sparsity' in particular could presumably be more general, thus leading to an even more general slicing (but presumably more difficult to study).

\section{Preliminaries to the proofs} \label{sec:proofs-generalities}

In Section \ref{sec: gen_prop_thres}, we introduce quantities used repeatedly in the proofs, and state their properties. Notably, the function $B$ appearing in the score function, see \eqref{eqbeta}, is shown to be increasing on $\RR^+$, and bounds for its moments are given. In Section \ref{sec: fixed_al_risk}, risk bounds for fixed $\al$ are derived, that will be useful both for the rate and confidence sets results. The proofs of these results is given in the Supplement \cite{castillo:szabo:2018:supplement}, Section \ref{supp-thresh}.

\subsection{General properties and useful thresholds} \label{sec: gen_prop_thres}

\begin{lem}\label{lem-gf}
For $\ga$ defined by \eqref{dens-hv}, $\delta\in(0,2)$
 and $\gf=\phi*\ga$, as $x\to\infty$,
\begin{align*}
& \gf(x) \asymp \ga(x), \\
& \gf(x)^{-1} \int_x^\infty \gf(u) du \asymp  x/\delta.
\end{align*}
Also, $\gf/\phi$ is strictly increasing from $(\gf/\phi)(0)<1$ to $+\infty$ as $x\to\infty$. 
\end{lem}

{\em Threshold $\zeta(\al)$.}  The monotonicity property in Lemma \ref{lem-gf} enables one to define a pseudo-threshold from the function $\bef=(\gf/\phi)-1$ as 
\[ \zeta(\al)=\bef^{-1}(\al^{-1}).\]
Using the bounds from \cite{js04} Sections 5.3 and 5.4, noting that these do not use any moment bound on $\ga$ (so their bounds also hold even if $\delta<1$ in the prior density \eqref{dens-hv}), one can link thresholds $t(\alpha)$ (the  threshold of the posterior median for given $\al$, see Section \ref{subsec-eb}) and $\zeta(\alpha)$ as follows: $t(\al)^2<\zeta(\al)^2$, and $\phi(t(\al))<C\phi(\zeta(\al))$, where $C$ is independent of $\delta$, and $\bef(\zeta(\al))\le 2+\bef(t(\al))$, arguing as in the proof of Lemma 3 of \cite{js04}.

{\em Threshold $\ta(\al)$.} As $\gf/\phi$ is continuous, one can define $\tau(\al)$ as the solution in $x$ of 
\[ \Omega(x,\al):=\frac{a(x)}{1-a(x)}=\frac{\al}{1-\al}\frac{\gf}{\phi}(x)=1. \]
Equivalently, $a(\ta(\al))=1/2$. Define $\al_0$ as $\ta(\al_0)=1$ and set 
\[ \tilde\ta(\al) = \ta(\al\wedge \al_0). \]
This is the definition from \cite{js04}, but note that for $\al$ small enough, $\tilde\ta(\al) = \ta(\al)$. Also, it follows from the definition of $\ta(\al)$ that $\bef(\ta(\al))=\bef(\zeta(\al))-2\le \bef(t(\al))$, by the inequality mentioned above, so that $\ta(\al)\le t(\al)$, as $\bef$ is increasing. This and the previous inequalities relating the thresholds $\zeta(\al), t(\al), \ta(\al)$ will be freely used in the sequel. 

\begin{lem} \label{lemzetaall}
For $\ga$ defined by \eqref{dens-hv}, $\delta\in(0,2)$, there exists $C_1>0$, $C_2\in\RR$ and $a_1>0$ such that for any $\alpha\le a_1$, 
\[ 2\log(1/\al)+C_1 \le  \zeta(\alpha)^2  \le 2\log(1/\al)+ (1+\delta)\log\log(1/\al) +C_2. \]
The same bounds hold, with possibly different constants $C_1, C_2$,  for $\tau(\al)^2$ and $t(\al)^2$. 
\end{lem} 

{\em Moments of the score function.} Recall that $\bef(x,\al) = \bef(x)/\big(1+\al\bef(x)\big)$ and  set 
\[ \tmf(\al) = - E_0 \bef(X,\al ), \quad \mf_1(\mu,\al) = E_\mu \bef(X,\al ),  \quad 
\mf_2(\mu,\al) = E_\mu \bef( X,\al )^2.\]

\begin{lem} \label{lem-momgf}
The function $\al\to \tmf(\al)$ is nonnegative and increasing in $\al$.  
For every fixed $\al\in(0,1)$, the function $\mu\mapsto m_1(\mu,\al)$  is symmetric and monotone increasing for $\mu\geq 0$. For every fixed $\mu>0$, the map $\al\to m_1(\mu,\al)$ is decreasing. As $\al\to 0$,
\[ \tmf(\al) \asymp \int_\zeta^\infty \gf(u)du\asymp \zeta \gf(\zeta)
\asymp \zeta^{-\delta}/\delta. \] 
We have $\mf_1(\mu,\al)\le (\al\wedge c)^{-1}$ and 
$\mf_2(\mu,\al) \le(\al\wedge c)^{-2}$ for all $\mu$, and
\begin{align*}
\mf_1(\mu,\al) & \le  
\begin{cases}
\vspace{.2cm}
-\tmf(\al) + C \zeta(\al)^{2-\delta} \mu^2,
\quad &\text{for } |\mu|< 1/\zeta(\al), \\
\vspace{.2cm}
C\frac{\phi(\zeta/2)}{ \al }, \quad &\text{for } |\mu|< \zeta(\al)/2 
%(\al\wedge c)^{-1}, \quad &\text{for all } \mu.
\end{cases} \\
\vspace{.5cm}
\mf_2(\mu,\al) & \le 
\begin{cases}
\vspace{.2cm}
\frac{C\delta}{\zeta(\al)^2}\frac{\tmf(\al)}{\al},
\quad &\text{for } |\mu|< 1/\zeta(\al), \\
\vspace{.2cm}
\frac{C}{\zeta}\frac{\phi(\zeta/2)}{ \al^2 },\quad & \text{for } |\mu|< \zeta(\al)/2
%(\al\wedge c)^{-2}, \quad & \text{for all } \mu,
\end{cases}
\end{align*}
for universal constants $c, C>0$. Finally, as $\al\to 0$, one has $\mf_1(\zeta,\al)\sim \frac12 \al^{-1}$, as $\al\to 0$.
\end{lem}

{\em Bounds on $a(x)$.} By definition of $a(x)$, for any real $x$ and $\al\in [0,1]$,
\begin{align}
\al \frac{\gf}{\gf\vee\phi}(x)\le a(x)  \le 1 \wedge \frac{\al}{1-\al} \frac{\gf}{\phi}(x).
\label{boundsa}
\end{align}
The following bound in terms of $\ta(\al)$, see \cite{js04} p. 1623, that again extends to the case of a density \eqref{dens-hv} with $\delta<1$, is useful for large $x$: for $\al\le \al_0$, so that $\tilde\tau(\al)=\tau(\al)$,
\begin{equation}\label{postwb}
 1-a(x) \le \1_{|x|\le \ta(\al)}
 +   e^{-\frac12 (|x|-\ta(\al))^2}\1_{|x| > \ta(\al)}.
\end{equation}

\subsection{Risk bounds for fixed $\al$} \label{sec: fixed_al_risk}

Let us consider the posterior $d_q$--loss for a fixed value of the tuning parameter $\al$, that is
\[  \int  d_q(\te , \te_0) d\Pi_{\al}(\te\given X) 
= \sum_{i=1}^n \int |\te_i - \te_{0,i}|^q d\Pi_\al(\te_i \given X).\]
To study $\int |\te_i - \te_{0,i}|^q d\Pi_\al(\te_i \given X)$, let  $\pi_\al(\cdot\given x)$ be the posterior given $\alpha$ evaluated at $X=x$,
\[ r_q(\al,\mu,x)=\int |u-\mu|^q d\pi_\al(u\given x)= (1-a(x))|\mu|^q + 
a(x)\int |u-\mu|^q \ga_x(u)du. \]

\begin{lem} \label{lemmomga}
Denoting $\ga_x(u) = \ga(u)\phi(x-u)/\gf(x)$, for any $q\in(0,2]$ and $\mu,x\in\RR$,
\begin{equation} \label{eq: UB_kappa_x} 
 \int |u-\mu|^q \ga_x(u)du \le  C\left[ |x-\mu|^{q} + 1\right]. 
\end{equation} 
For $\al$ small enough, for any $q>2\delta$ and any $\mu,x\in\RR$, the following risk bounds hold
\begin{align*}
 E_0 r_q(\al,0,x) & \leqa \al \ta(\al)^{q-\delta}/(q-\delta) + \ta(\al)^{q-1} \phi(\ta(\al))+ \phi(\ta(\al))/\ta(\al), \\
 E_\mu r_q(\al,\mu,x) & \leqa (1+\ta(\al)^q). 
\end{align*}
\end{lem}

%%%%%%%%%%%%%%%%%%%%%%%%%%%%%%%%%%%%%%%

%%%%%%%%%%   CREDIBLE SETS PROOFS %%%%%%%%%%%%%%

%%%%%%%%%%%%%%%%%%%%%%%%%%%%%%%%%%%%%%%

\section{Proofs for credible sets} \label{sec:proofs:cs}

The credible set $\cC_{q,\al}$ in \eqref{def: credible_mmle_q_fixed} is centered around the posterior median $\hat\te_\al$. We shall use below a few basic properties of this estimator, which were established in \cite{js04} (they extend without difficulty to the case of heavier tails than Cauchy, as no moments conditions on $\ga$ are needed for their proof): the fact that $\hat\te_\al$ is a shrinkage rule with the bounded shrinkage property, from Lemma 2 of \cite{js04}, and the bounds of the risk  $E_{\te_0} d_q(\hatha,\te_0)$ for fixed $\alpha$ in Lemmas 5 and 6 of \cite{js04}, recalled in the Supplement \cite{castillo:szabo:2018:supplement}, see \eqref{pmed1}--\eqref{pmed2} there.

\subsection{Proof of Propositions \ref{thm: coverage:nonadapt}
 and \ref{thm: radius:nonadapt} 
}\label{sec: coverage:nonadapt}

For any given $\te_1,\te_2$ in the credible set $\cC_{q,\al}$ from \eqref{def: credible_mmle_q_fixed}, by Lemma \ref{lempq},
\begin{align*}
 d_q(\te_1,\te_2) & \le (2^{q-1}\vee1) \Big(d_q( \te_1, \hatha) +d_q(\te_2, \hatha )\Big)
  \le  2(2^{q-1}\vee1)M v_{q,\al}(X).
\end{align*} 
Using the risk bounds from Lemma \ref{lemmomga} with $\mu=\te_{0,i}$ for each index $i$ between $1$ and $n$ leads to, distinguishing between the signal case ($\te_{0,i}=0$) and non-zero signal case ($\te_{0,i}\neq 0$),
\[ E_{\te_0}\int d_q(\te,\te_0)  d\Pi_\al(\te\given X) \le C[(n-s) \al\ta(\al)^{q-\delta} 
+ s( 1+\ta(\al)^q )]. \]
Also, combining Lemmas 5 and 6 in  \cite{js04} with fixed non-random threshold equal to $t(\al)$, one gets
\[ E_{\te_0} d_q(\hatha,\te_0) \le C[ (n-s)t(\al)^{q-1}\phi(t(\alpha)) + s(1+t(\al)^q) ]. \]
From the last displays one deduces that, for any $\te_1,\te_2$ in $\cC_{q,\al}$,
\begin{align*}
 E_{\te_0} d_q(\te_1,\te_2)  & \le C M \left[ (n-s)\Big\{\al\ta(\al)^{q-\delta}+t(\al)^{q-1}\phi(t(\al))\Big\} + s(1+\ta(\al)^q + t(\al)^q)\right].
\end{align*}
%Using the inequality $\al^{-1}=\bef(\zeta(\al))\le 2+\bef(t(\al))$ stated in Section \ref{sec: gen_prop_thres} combined with \eqref{eqbeta} and Lemma \ref{lem-gf} on $g$,
Let us recall the inequality $\al^{-1}=\bef(\zeta(\al))\le 2+\bef(t(\al))$ stated in Section \ref{sec: gen_prop_thres} and  $\bef(\cdot)=(g/\phi)(\cdot)-1$ by \eqref{eqbeta}. As $t(\al)\to +\infty$ when $\alpha\to 0$, we have that for small $\alpha$ it holds $1+(g/\phi)(t(\al))\leqa (g/\phi)(t(\al))$. Lemma \ref{lem-gf} gives that $g(x)$ has tails  $x^{-1-\delta}$ as $x\to\infty$, which implies, using again that $t(\al)$ goes to infinity as $\alpha$ goes to $0$, that for $\alpha$ small enough,
\begin{equation}\label{eq: UB:phi:t}
\phi\big(t(\alpha)\big)\lesssim \alpha \gf\big(t(\alpha)\big)\lesssim \alpha t(\alpha)^{-1-\delta}.
\end{equation}
Using Lemma \ref{lemzetaall}, the above inequality on the diameter  becomes the following, hereby proving the second part of Proposition \ref{thm: radius:nonadapt}: for any $\te_1,\te_2$ in $\cC_{q,\al}$ and small enough $\al$, 
\[   E_{\te_0} d_q(\te_1,\te_2)   \le C M \left[ n \al\log^{(q-\delta)/2}(1/\al) + s\log^{q/2}(1/\al)\right]. \]

\noindent{\em Confidence.} 
Next one considers the coverage probability 
\[ P_{\te_0}[ \te_0\in\cC_{q,\al}] 
= P_{\te_0}[ d_q(\te_0,\hatha)\le M v_{q,\al}(X) ].
\]
Let $\mu_n=\mu_n(X)=d_q(\te_0,\hatha)$ and 
 $S_0:=\{i:\ \te_{0,i}\neq 0\}$, then
\begin{align*}
\mu_n & = \sum_{i\in S_0} 
|\te_{0,i} - \hat\te_{\al,i}  
|^q 
+ \sum_{i\notin S_0} |\hat\te_{\al,i}|^q \1_{|\veps_i|>t(\al)} =: \mu_1+\mu_2.
\end{align*}
With $v_{q,\al}=v_{q,\al}(X)$, $\omega_q(x)=\int |u|^q \ga_x(u) du$, and $\kappa_{\al,q,i}(x)=\int |u-\hat\te_{\al,i}|^q\ga_x(u)du$,
\begin{align}
v_{q,\al} = & \sum_{i=1}^n  \int |\te_i-\hat\te_{\al,i}|^q d\Pi(\te_i\given X_i) \nonumber\\
= & \sum_{i: |X_i|\le t(\al)} \int |\te_i|^q d\Pi(\te_i\given X_i) 
+  \sum_{i: |X_i|>t(\al)} \int |\te_i-\hat\te_{\al,i}|^q d\Pi(\te_i\given X_i) \nonumber \\
= &  \sum_{i\in S_0: |X_i|\le t(\al)} a(X_i) \omega_q(X_i)
+  \sum_{i\in S_0: |X_i|>t(\al)} \left[a(X_i)\kappa_{\al,q,i}(X_i) + (1-a(X_i))|\hat\te_{\al,i}|^q\right]\nonumber\\
 & + \sum_{i\notin S_0: |\veps_i|\le t(\al)} a(\veps_i) \omega_q(\veps_i)
+  \sum_{i\notin S_0: |\veps_i|>t(\al)} \left[ a(\veps_i)\kappa_{\al,q,i}(\veps_i) + (1-a(\veps_i))
|\hat\te_{\al,i}|^q\right] \nonumber\\
& =: v_1+v_2+v_3+v_4.\label{eq: variance}
\end{align}

\noindent{\em Step 1, lower bound on the variance.} 
The variance at fixed $\al$ is bounded from below by, using first  Lemma \ref{lem-lbv} to bound $\omega_q$ from below, and next using $(1-\al)\phi(\veps_i)+\al\gf(\veps_i)\le 2(1-\al)\phi(\veps_i)$ when $|\veps_i|\le \ta(\al)$ to bound $a(\veps_i)$ from below, 
\begin{align}
  v_3 \ge \sum_{i\notin S_0,\, 
C_0\le |\veps_i|\le \ta(\al)} a(\veps_i) \omega_q(\veps_i)
\geq c\al \sum_{i\notin S_0} \frac{\gf}{\phi}(\veps_i) (|\veps_i|^q+1) \1_{C_0\le |\veps_i|\le \ta(\al)},\label{eq: coverage:LB:var}
\end{align}
for any $C_0>0$ and $\alpha<1/2$ (say). One deduces that, in view of the assumption $q>\delta$, for small enough choice of $C_0>0$ 
\begin{align*} 
 E[\al \sum_{i\notin S_0} \frac{\gf}{\phi}(\veps_i) (|\veps_i|^q+1) \1_{C_0\le |\veps_i|\le \ta(\al)}] & \ge  \frac{(n-s)\al}{4} \int_{C_0}^{\ta(\al)} \gf(x) x^q dx\\
&   \ge \frac{1}{2^3(q-\delta)}n\al\ta(\al)^{q-\delta}.
\end{align*} 
Furthermore in view of the inequality $(1+|u|^q)^2\le C u^{2q}$ and Lemma \ref{lem-gf} we get that
\begin{align*}
\text{Var}\big[\al \sum_{i\notin S_0} \frac{\gf}{\phi}(\veps_i) (|\veps_i|^q+1) \1_{C_0\le |\veps_i|\le \ta(\al)}\big]  
&\leq n\alpha^2 \int_{C_0}^{\ta(\al)} \frac{\gf^2}{\phi^2}(u) u^{2q} \phi(u)du \\
& \leqa n\al^2 \int_{C_0}^{\ta(\al)} u^{-2-2\delta+2q} \phi(u)^{-1}du.
\end{align*}
An integration by parts shows that, setting $d=-2-2\delta+2q$, the right hand side of the preceding display is further bounded from above by a multiple of $n\al^2\ta(\al)^{d-1} e^{\ta(\al)^2/2}$.
Using that $\ta(\al)\le t(\al)$ and $t(\al_n)^2= 2\log{n}$ (so that $e^{\ta(\al)^2/2}\le n$ for $\al>\al_n$), one gets that  the variance term is of smaller order than the square of the expectation, so 
\begin{align}
v_{q,\al}\geq  \frac{1}{2^4(q-\delta)}n\al\ta(\al)^{q-\delta},\label{eq:coverage:LB:var0}
\end{align}
with high probability. The diameter lower bound in Proposition \ref{thm: radius:nonadapt} follows, as $M v_{q,\al}\le \text{diam}_q(\cC_{q,\alpha})$, where we have used that $\delta<q$, the lower bound on $\tau(\al)$ from Lemma \ref{lemzetaall}, the inequality $n\al\ta(\al)^{q-\delta}\geqa n\al(2\log(1/\al)+C)^{(q-\delta)/2}$, and the latter is increasing in $\al$.\\

\noindent{\em Step 2, upper bound on the bias.} 
As a first step we give upper bounds for the terms $\mu_1$ and $\mu_2$. 
The posterior median is, as stated in \cite{js04} Lemma 2, a shrinkage rule with the bounded shrinkage property: there exists $b>0$ such that for  all $x\ge 0$ and $\alpha$,
\begin{equation}\label{boundedshr}
 (x-t(\al)-b)\vee 0 \le \hat\te_\al(x) \le x. 
\end{equation} 
This implies, as $\hat\te_\al(-x)=-\hat\te_\al(x)$, that by using Lemma \ref{lempq}
\begin{align}
 \mu_1 & \le (2^{q-1}\vee1)\Big(\sum_{i\in S_0}|\veps_i|^q + \sum_{i\in S_0} (t(\al)+b )^q\Big) 
  \lesssim \sum_{i\in S_0}|\veps_i|^q   + s(t(\al)+b)^q .\label{eq: UB_biasonzero:Lq}
\end{align}
Let us now use a standard chi-square bound: if $Z_i$ are $\cN(0,1)$ iid, for any integer $s\ge 1$ and $t>0$, one has $P[\sum_{i=1}^s (Z_i^2-1) \ge t] \le \exp\{ -t^2/[4(s+t)] \}$. For $t=s(\log(n/s))^{1/2}$, the bound is $\exp\{-s(\log(n/s))^{1/2}\}\le \exp\{-c\log^{1/2}n\}=o(1)$, so
\begin{align*}
 P\Big[ \sum_{i\in S_0} (\veps_i^2-1) > s\sqrt{\log(n/s)} \Big] =o(1).
 \end{align*}
Also note that  by H\"older's inequality, $\sum_{i\in S_0} |\veps_i|^q  \le (\sum_{i\in S_0} \veps_i^2)^{q/2} s^{1-q/2}$. Hence, 
\begin{equation}\label{eq: UB_biasonzero2:Lq}
\sum_{i\in S_0} |\veps_i|^q \leqa \big(s+s\sqrt{\log(n/s)}\big)^{q/2}s^{1-q/2} \leqa s \log^{q/4}(n/s),
\end{equation}
with probability tending to one. For $\mu_2$, using again that the posterior median is a shrinkage rule,
$\mu_2\le \mu_3 :=\sum_{i=1}^n  
|\veps_i|^q \1_{|\veps_i|> t(\al)}. $
 Then  in view of Lemma \ref{lem: help:adapt:3q} (with $t_1=0$, and $t_2=t(\al)\geq 1$) we have with probability tending to one that 
\[ \mu_3 \leqa t(\al)^{q-1} [ne^{-t(\al)^2/2}] + M_n t(\al)^{(q+1)/2} [ne^{-t(\al)^2/2}]^{1/2}. \]
 For the first term to dominate this expression it is enough, provided $M_n\to\infty$ slow enough,  that
 \begin{equation} \label{tec1:Lq}
  t(\al)^2 \le 2\log n - (3-q+c)\log\log n=2\log(n/(\log^{(3-q+c)/2} n)), 
 \end{equation}
 for some  $c>0$, which follows from the assumption $\alpha\gg s/n\gtrsim (\log n)^2/n$ and Lemma \ref{lemzetaall}. Therefore with probability tending to one
\begin{equation*}
 \mu_1+\mu_2 \lesssim
s \log^{\frac{q}{4}}(n/s) + s t(\al)^{q} +t(\al)^{q-1} [ne^{-t(\al)^2/2}].
\end{equation*}
Now recall the bound $v_{q,\al}\geqa n\al \ta(\al)^{q-\delta}$.
We conclude the proof of the first part of
 Proposition \ref{thm: coverage:nonadapt} (the positive coverage result) by noting that for $\alpha\geq s\log^{\delta/2}(n/s) /n$, in view of \eqref{eq: UB:phi:t} we have $\mu_1+\mu_2\lesssim s\log^{q/2}(n/s)+n\alpha t(\alpha)^{q-\delta-2}\lesssim v_{q,\al}$. Indeed, for the second term one has $q-\delta-2<0$, so $t(\al)^{q-\delta-2} \leq \ta(\al)^{q-\delta-2}$ using $t(\al)\ge \ta(\al)$. For the first term, using the lower bound in Lemma \ref{lemzetaall}, $\ta(\al)^2\geqa \log(1/\al)$ and next that
$\al \to \al \log(1/\al)^{(q-\delta)/2}$ is increasing, so that, using  $\alpha\geq s\log^{\delta/2}(n/s) /n$ again, one gets 
\begin{align*}
 v_{q,\al} \geqa n \{s\log^{\delta/2}(n/s) /n\}\ta(s\log^{\delta/2}(n/s)/n)^{q-\delta} \geqa s \log^{q/2}(n/s).
 \end{align*}
 Hence for large enough choice of $M$ in \eqref{def: credible_mmle}, one gets frequentist coverage tending to one. 

To obtain the second part of Proposition \ref{thm: coverage:nonadapt} (the  non-coverage result), we will use that by assumption $ \alpha=o(s\log^{\delta/2}(n/s) /n)$ and show below that  for some $C>0$,
\begin{align}
\inf_{\theta_0\in\tilde\Theta_{s,\alpha}} P_{\theta_0}\big(\mu_n\geq  C s\log^{q/2}(n/s)\big)\rightarrow 1,\label{eq: counter:bias}\\
v_{q,\al}=o(s\log^{q/2}(n/s)).\label{eq: counter:variance}
\end{align}
Then the result follows by combining the above statements.\\

\noindent{\em Step 3, lower bound on the bias.} To show \eqref{eq: counter:bias} note that for $\theta_0\in \tilde\Theta_{s,\alpha}$, and as $\ta(\al)\le t(\al)$,
\begin{align*}
\mu_n\geq\sum_{i\in S_0}|\hat\theta_{\alpha,i}-\theta_{0,i}|^q&\geq \sum_{\tau(\al)/8\leq |\theta_{0,i}| \leq  \tau(\al)/4}|\theta_{0,i}|^q \1_{\eps_i\in\big(-(3/4)\tau(\al),(3/4)\tau(\al)\big)}\\
&\geq  2^{-3q}\tau(\al)^{q}\sum_{i\in S_0} \1_{\eps_i\in\big(-(3/4)\tau(\al),(3/4)\tau(\al)\big)}.
\end{align*}
Then as $P\big( \eps_i\in \big(-(3/4)\tau(\al),(3/4)\tau(\al)\big)\big)\geq 3/4$ if $\alpha\leq \alpha_1$ for some sufficiently small $\alpha_1>0$, we get by Hoeffding's inequality that
\begin{align*}
P\big( \sum_{i\in S_0} \1_{\eps_i\in\big(-(3/4)\tau(\al),(3/4)\tau(\al)\big)}\leq s/2\big)\lesssim e^{-s/2}=o(1).
\end{align*}
This implies that $\mu_n\geq 2^{-3q-1}s\ta(\al)^q$ with high probability. This is at least of the order $Cs\log^{q/2}(n/s)$ using the assumption on $\alpha$.\\

\noindent{\em Step 4, upper bound on the variance.} Next we deal with \eqref{eq: counter:variance} by giving upper bounds for $v_1,v_2,v_3$ and $v_4$ in \eqref{eq: variance} for $\theta_0\in \tilde{\Theta}_{s,\alpha}$, separately. Let us start with $v_1$. Recalling that $\ta(\al)$ and $t(\al)$ differ slightly, let us split the sum defining $v_1$ over indexes $i\in S_0$ with  $|X_i|\le \ta(\al)$ and $\ta(\al)< |X_i|\le t(\al)$ respectively. In view of \eqref{boundsa} and \eqref{eq: UB_kappa_x} (with $\mu=0$), one gets
\begin{align*}
v_1&\lesssim \al \sum_{i\in S_0:\,|X_i|\leq \tau(\al)} \frac{\gf}{\phi}(X_i)(1+|X_i|^q) +\sum_{i\in S_0:\, \tau(\al)\leq |X_i|\leq t(\al)}(1+|X_i|^q)\nonumber\\
&\lesssim \al s e^{\tau(\al)^2/4}\tau(\al)^{q-\delta}+M_n\al s^{1/2}e^{3\tau(\al)^2/8}\tau(\al)^{q-\delta-1/2}\nonumber\\
&\qquad+ s \tau(\al)^{q-1}e^{-3^2\tau(\al)^2/2^5}+M_n s^{1/2}\tau(\al)^{(q+1)/2}e^{-3^2\tau(\al)^2/2^6}\nonumber=O(s),
\end{align*}
where the second inequality follows from Lemma \ref{lem: help:adapt:1q} (with $t=\tau(\al)/4$) and Lemma \ref{lem: help:adapt:3q} (with $t_2=\tau(\al)$, $t_1=\tau(\al)/4$).
Then in view of $|\hat\theta_{\al,i}|^q\leq |X_i|^q$ and $\kappa_{\al,q,i}\lesssim |\hat\theta_{\al,i}|^q+|X_i|^q+1\lesssim |X_i|^q+1$, see equation \eqref{eq: UB_kappa_x}, we get that
\begin{align*}
v_2&=\sum_{i\in S_0, |X_i|>t(\al)}a(X_i)\kappa_{\al,q,i}+(1-a(X_i))|\hat\theta_{\al,i}|^q\lesssim \sum_{i\in S_0, |X_i|>t(\al)}(1+|X_i|^q)\\
&\lesssim s t(\al)^{q-1}e^{-3^2t(\alpha)^2/2^5}+M_ns^{1/2}t(\al)^{(q+1)/2}e^{-3^2t(\alpha)^2/2^6}=O(s),
\end{align*}
where the last line follows from Lemma \ref{lem: help:adapt:3q} (with $t_2=t(\al)$, $t_1=t(\al)/4$) and from $e^{-c_1t(\alpha)^2}t(\alpha)^{c_2}=O(1)$ for $c_1,c_2>0$ and $\alpha\lesssim \alpha_1$. Next in view of Lemma \ref{lem: help:adapt:1q} (with $t=0$) and Lemma \ref{lem: help:adapt:3q} (with $t_1=0$ and $t_2=\tau(\al)$)
\begin{align*}
v_3&\lesssim \al \sum_{i\notin S_0:\,|\varepsilon_i|\leq \tau(\al)} \frac{\gf}{\phi}(\varepsilon_i)(1+|\varepsilon_i|^q) +\sum_{i\notin S_0:\, \tau(\al)\leq |\varepsilon_i|\leq t(\al)}(1+|\varepsilon_i|^q)\\
&\lesssim \al n\tau(\al)^{q-\delta}+n\tau(\al)^{q-1}e^{-\tau(\al)^2/2}+M_n \sqrt{n}\tau(\al)^{(q+1)/2}e^{-\tau(\al)^2/4}= o\big(s\log^{q/2} (n/s)\big),
\end{align*}
where the last line follows from the definition of $\tau(\al)$ (implying $e^{-\tau(\al)^2/2}\lesssim \al \gf(\tau(\al))\lesssim \al \tau(\al)^{-1-\delta}$) as well as $n\al=o(s\log^{\delta/2}(n/s))$ as assumed.   
To conclude the proof of Proposition \ref{thm: coverage:nonadapt} it is enough to note that, using similar arguments,
\begin{align*}
v_4&\lesssim \sum_{i\notin S_0, |\varepsilon_i|>t(\al)}(1+|\varepsilon_i|^q)
\lesssim n t(\al)^{q-1}e^{-\frac{t(\al)^2}{2}}+M_n n^{\frac{1}{2}}t(\al)^{\frac{q+1}{2}}e^{-t(\al)^2/4}=o(s\log^{\frac{q}{2}} (\frac{n}{s})).
\end{align*}

\subsection{Proof of Theorem \ref{thm: coverage_mmle_q}}\label{sec: coverage_mmle_q}

\noindent{\em Step 0, Concentration of $\hat{\alpha}$ under the excessive-bias condition.} A key component is to describe the behaviour of the MMLE $\hat\al$ over the set $\Theta_0^q[s;A,C_q,D_q]$. In view of \eqref{eq: subset:EB} (see also Lemma \ref{lem: effective:sparsity}, below) let us  consider the larger set $\Theta_{0}^2[s;A,C_q,D_q(\sqrt{2}A)^{2-q}]$ and denote by $\tilde{s}=\tilde{s}(\theta_0)$ the effective sparsity of $\theta_0\in \Theta_{0}^2[s;A,C_q,D_q(\sqrt{2}A)^{2-q}]$, i.e.
\begin{align}
\tilde s=\tilde s (\te_0)=\big| \bigl\{i: |\te_{0,i}|\geq A\sqrt{2\log (n/\ell)}\bigr\}\big|,\label{def: equiv:tilde:s}
\end{align}
where $\ell\geq(\log _2 n)^2$ is the smallest integer satisfying  \eqref{condition: EB} with parameters $A=A$, $C_2=C_q$ and $D_2=D_q(\sqrt{2}A)^{2-q}$.

Next we introduce weights $\tilde\al_i$, for $i=1,2$, as the solution of the respective equations
\begin{align}
d_i\tilde\al_i \tilde{m}(\tilde{\al_i})=\tilde{s}/n,\label{def:tilde:al}
\end{align}
where the constants $0<d_2<d_1$ are specified later. We note that in view of the fact that $\al\to \al\tilde{m}(\al)$ is increasing, which follows from Lemma \ref{lem-momgf},
\begin{align}
\tilde\alpha_2/\tilde\alpha_1\leq d_1/d_2,\qquad \tilde\alpha_2=o(1).\label{eq: bounds_fun_facts}
\end{align}
We now show that by appropriate choice of  constants $d_1>d_2$ the corresponding $\tilde\alpha_1<\tilde\alpha_2$ are upper and lower bounds, respectively, for $\hat\al$ under the excessive-bias condition.
\begin{lem}\label{thm: bounds_mmle}
 Under the conditions of Theorem \ref{thm: coverage_mmle_q}, for $\tilde{\alpha}_1, \tilde{\alpha}_2$ as in \eqref{def:tilde:al}, we have 
\begin{align}
\inf_{\theta_0\in \Theta_{0}^q[s;A,C_q,D_q]}P_{\theta_0}(\tilde\alpha_1\leq \hat\alpha\leq \tilde\alpha_2)\rightarrow 1.\label{assump: bounds}
\end{align} 
\end{lem}
The proof of the lemma is deferred to Section \ref{sec: bounds_mmle} of the Supplement \cite{castillo:szabo:2018:supplement}.

\noindent{\em Step 1, lower bound on the variance.}  In view of \eqref{eq: variance}--\eqref{eq: coverage:LB:var} and Lemma \ref{thm: bounds_mmle}, with $P_{\theta_0}$-probability tending to one
\begin{align*}
v_{q,\hat \al}& \geq v_3(\hat\al)\geq c\hat\al \sum_{i\notin S_0}\frac{g}{\phi}(\eps_i)(|\eps_i|^q+1)\1_{C_0\leq |\eps_i|\leq \tau(\hat\alpha)}\\
&\geq c\tilde\alpha_1\sum_{i\notin S_0}\frac{g}{\phi}(\eps_i)(|\eps_i|^q+1)\1_{C_0\leq |\eps_i|\leq \tau(\tilde\alpha_2)}.
\end{align*}
Then following from the computations above assertion \eqref{eq:coverage:LB:var0}, Lemma \ref{lem-momgf} and the inequality $\ta(\al)\le \zeta(\al)$ we have, for large enough $n$, 
\begin{align*}
v_{q,\hat \al}& \geq \frac{1}{2^4(q-\delta)}\tilde\al_1 n\ta(\tilde\al_2)^{q-\delta} 
 \geq \frac{1}{2^4(q-\delta)} \frac{\tilde{s} \ta( \tilde\al_2)^{q-\delta}}{d_1\zeta(\tilde\al_1)^{-\delta}/\delta}
 \geq \frac{\delta}{2^4d_1(q-\delta)}\tilde{s} \log^{q/2}(n/\tilde{s}).
\end{align*}

\noindent{\em Step 2, upper bound on the bias.} Let split the bias term
\[ \mu_n(\hat\al) = \sum_{i=1}^n |\te_{0,i} - \hat\te_{\hat\al,i}|^q\]
along the index sets $Q_1=\{i:\, |\theta_{0,i}|\leq 1/t(\tilde\al_2) \}$, $Q_2=\{i:\,  1/t(\tilde\al_2)\leq |\theta_{0,i}|\leq t(\tilde\al_2)/2 \}$ and $Q_3=\{i:\, |\theta_{0,i}|\geq t(\tilde\al_2)/2 \}$, i.e.
\begin{align*}
\mu_n(\hat\al)=\sum_{i\in Q_1\cup Q_2}|\theta_{0,i}-\hat\te_{\hat\al,i}|^q+\sum_{i\in Q_3}|\theta_{0,i}-\hat\te_{\hat\al,i}|^q
=:\mu_1(\hat\al)+\mu_2(\hat\al).
\end{align*}
Using that $|\hat\theta_{\hat\al,i}|^q\leq \1_{|X_i|> t(\hat\al)}|X_i|^q$, Lemma \ref{lempq}, the monotone decreasing property of the functions $t\mapsto \1_{|X_i|>t}$ and Lemma \ref{thm: bounds_mmle} we have with probability tending to one that
\begin{align}
\mu_1(\hat\al)&\leq (2^{q-1}\vee1)\Big( \sum_{i\in Q_1\cup Q_2}|\theta_{0,i}|^q+ \sum_{i\in Q_1}|X_i|^q \1_{|X_i|>t(\hat\al)}+ \sum_{i\in Q_2}|X_i|^q \1_{|X_i|>t(\hat\al)}\Big)\nonumber\\
&\leq (2^{q-1}\vee1)\Big( \sum_{i\in  Q_1\cup Q_2}|\theta_{0,i}|^q+ \sum_{i\in Q_1}|X_i|^q \1_{|X_i|>t(\tilde\al_2)}+\sum_{i\in Q_2}|X_i|^q \1_{|X_i|>t(\tilde\al_2)} \Big).\label{eq: hulp01q}
\end{align}
The first term is smaller than $C_qD_qc_0\tilde{s} \log^{q/2}(n/\tilde{s})$ following from Lemma \ref{lem: effective:sparsity} and $t(\tilde\al_2)/2\sim \{0.5\log (n/\tilde{s})\}^{1/2}<A'\{2\log (n/\tilde{s})\}^{1/2}$, for any $A'>0$ and large enough $n$.  Furthermore, by applying Lemma \ref{lem: help:adapt:3q} (with $t_1=1/t(\tilde\al_2)$, $t_2=t(\tilde\al_2)$) and  the inequality $|Q_1|<n$ the second term on the right hand side of the preceding display is bounded from above with $P_{\theta_0}$-probability tending to one by a multiple of, for arbitrary $M_n\rightarrow\infty$,
$$nt(\tilde\al_2)^{q-1}e^{-t(\tilde\al_2)^2/2}+M_n t(\tilde\al_2)^{(q+1)/2}(ne^{-t(\tilde\al_2)^2/2})^{1/2}.$$
Then by noting that in view of Lemma \ref{lemz}, we have $nt(\tilde\al_2)e^{-t(\tilde\al_2)^2/2}\gtrsim \tilde{s}\geq(\log n)^2$, 
\begin{align}
t(\tilde\al_2)^{(q+1)/2} (ne^{-t(\tilde\al_2)^2/2})^{1/2}=o(nt(\tilde\al_2)^{q-1} e^{-t(\tilde\al_2)^2/2}).\label{eq: hulp02q}
\end{align}
Finally, in view of Lemma \ref{lem: help:adapt:3q} (with $t_1=t(\tilde\al_2)/2$ and $t_2=t(\tilde\al_2)$) the third term in the right hand side of \eqref{eq: hulp01q} is bounded with probability tending to one by a multiple of
\begin{align*}
|Q_2| t(\tilde\al_2)^{q-1}e^{-t(\tilde\al_2)^2/8}+M_n t(\tilde\al_2)^{(q+1)/2}  (|Q_2|e^{-t(\tilde\al_2)^2/8})^2.
\end{align*}
 Then note that in view of Lemma \ref{lem: help:adapt:2q} (with $t=1/t(\tilde\al_2)$)  the cardinality of the set $Q_2$ is bounded from above by a multiple of $\tilde{s} t(\tilde\al_2)^{q}\log^{q/2} (n/\tilde{s})$, hence by using that the function $t\mapsto t^{c_{1}}e^{-c_2 t^2}$ tends to zero as $t$ goes to infinity for arbitrary $c_1\in\mathbb{R}$ and $c_2>0$, the preceding display is of smaller order than $\tilde{s}$. By putting together the obtained bounds, one concludes that  by choosing $M_n$ tending to infinite sufficiently slowly, e.g. $M_n=o(\log^{q/4}(n/\tilde{s}))$, one gets 
\begin{align*}
\mu_1(\hat\al)\leq  (2^{q-1}\vee1)\Big( C n t(\tilde\al_2)^{q-1}e^{-t(\tilde\al_2)^2/2}+D_qC_qc_0\tilde{s}\log^{q/2} (n/\tilde{s})\Big).
\end{align*}
It remained to deal with $\mu_2(\hat\al)$. In view of assertion  \eqref{eq: UB_biasonzero:Lq}
\begin{align}
\mu_2(\hat\al)\leq  (2^{q-1}\vee1 )\Big(\sum_{i\in Q_3}|\eps_i|^q + \sum_{i\in Q_3} (t(\hat\al)+b)^q\Big).\label{eq: hulp05q}
\end{align}
Lemma \ref{lem: help:adapt:2q} with $t=t(\tilde\al_2)/2$ implies $|Q_3| \le c_0(2^qD_qC_q+1) \tilde{s}$, so  the second term is bounded with probability tending to one using Lemma \ref{thm: bounds_mmle}  and then Lemma \ref{lemz} by
\begin{align*}
\sum_{i\in Q_3} (t(\hat\al)+b)^q\leq |Q_3|(t(\tilde\al_1)+b)^q \leq 3c_0(2^qD_qC_q+1)\tilde{s} \log^{q/2} (n/\tilde{s}),
\end{align*}
for $n$ large enough. Next we deal with the first term on the right hand side of \eqref{eq: hulp05q}. In view of \eqref{eq: UB_biasonzero2:Lq} we have with probability tending to one that
\begin{align*}
\sum_{i\in Q_3} |\veps_i|^q  \lesssim |Q_3|\log^{q/4}(n/|Q_3|)\leq c_0(2^qD_qC_q+1) \tilde{s}\log^{q/4}(n/\tilde{s}). 
\end{align*}
Now combining \eqref{eq: UB:phi:t} and Lemmas \ref{lem-momgf} and \ref{lemz}, one sees that
\begin{equation}\label{eq: UB:extra term:t}
n t(\tilde\al_2)^{q-1}e^{-t(\tilde\al_2)^2/2}\lesssim n\tilde\al_2  t(\tilde\al_2)^{q-2-\delta}  \lesssim\zeta(\tilde\al_2)^\delta \tilde{s} t(\tilde\al_2)^{q-2-\delta} =o\big(\tilde{s}\big)
\end{equation}
so putting the bounds on $\mu_1, \mu_2$ together, with probability tending to one,
\begin{align*}
\mu_n(\hat\al) \leq (2^{q-1}\vee1 )3c_0\Big(2^qD_qC_q+1+o(1) \Big)  \tilde{s} \log^{q/2}(n/\tilde{s}). 
\end{align*}
Therefore by choosing $M>3c_0(2^qD_qC_q+1)(2^{q-1}\vee1 ) 2^4d_1(q-\delta)/\delta$ we get frequentist coverage tending to one.
Note that the preceding results hold simultaneously for all $q\in(\sbl{2}\delta,2)$ if $M> \sup_{q\in (2\delta,2)}3c_0(2^qD_qC_q+1)(2^{q-1}\vee1 ) 2^4d_1(q-\delta)/\delta$.

To conclude the proof of Theorem \ref{thm: coverage_mmle_q}, it is enough to obtain the diameter bound, which is done using similar arguments as for the upper bound of Proposition \ref{thm: radius:nonadapt}, once the MMLE $\al$  is controlled using Lemma \ref{thm: bounds_mmle} as above. The detailed argument can be found in Section \ref{sec: coverage_mmle_q-supp} of the Supplement \cite{castillo:szabo:2018:supplement}. This concludes the proof.

\subsection{Technical lemmas for credible sets}

\begin{lem} \label{lempq}
For any reals $x,y$, and $q>0$,
\begin{align}
 |x+y|^q & \le (2^{q-1}\vee 1)\left[ |x|^q + |y|^q \right]. \label{manipq}
\end{align} 
\end{lem}
 \begin{proof}
First note that for $q<1$ the map $x\to |x+y|^q-|x|^q-|y|^q$ is monotone. Furthermore, since the map $x\mapsto |x|^q$ is convex for $q\geq 1$ we have by Jensen's inequality that $|(x+y)/2|^q\leq (|x|^q+|y|^q)/2$.
 \end{proof}
 The proofs of Lemmas \ref{lem: effective:sparsity} up to \ref{lem: help:adapt:2q} can be found in Section \ref{sec:proofs:technical:credible} of the Supplement \cite{castillo:szabo:2018:supplement}.
 
\begin{lem}\label{lem: effective:sparsity}
Assume that $\te_0$ satisfies the excessive-bias condition $EB(q)$ in \eqref{condition: EB_q} for some positive parameters $A,C_q,D_q$ and therefore it belongs also to $\Theta_0^{2}[s;A,C_q,D_q(\sqrt{2}A)^{2-q}]$. Let us denote by $\tilde{s}_{q}$ and $\tilde{s}$ the corresponding effective sparsity parameters. Then there exists a large enough $c_0\geq 1$ such that
$$\tilde{s}\leq \tilde{s}_{q}\leq c_0\tilde{s}.$$
Furthermore, for every $1<A'<A$,  and large enough $n$,
\begin{align*}
\sum_{|\theta_{0,i}|\leq A'\sqrt{2\log (n/\tilde{s})}}|\theta_{0,i}|^q\leq C_q D_qc_0\tilde{s} \log^{q/2}(n/\tilde{s}), \quad
\tilde{s}\leq \big|\big\{ |\theta_{0,i}|\geq A' \sqrt{2\log(n/\tilde{s})} \big\}\big|.
\end{align*}
\end{lem}

\begin{lem}[Basic bounds on $\zeta(\al_1),\tau(\al_1)$ and $t(\al_1)$ and tilde versions] \label{lemz}
Let $\alpha_1$ be defined by \eqref{alq} for $d$ a given constant, and let $\zeta_1=\zeta(\al_1)$. Then for some constants $C_1, C_2$,
\[ 2\log(n/s) + C_1  
\le \zeta(\al_1)^2 \le 2\log(n/s)  + \log(1+\log(n/s)) +C_2.\]
The same bounds hold, with possibly different constants $C_1$ and $C_2$,  for $\tau(\al_1)^2$ and $t(\al_1)^2$. 
Furthermore, the same result holds (with possibly different constants $C_1,C_2>0$) when $\alpha_1$ is replaced by $\tilde\al_1$ or $\tilde\al_2$ (defined in \eqref{def:tilde:al}) and $s$ by $\tilde{s}$.
\end{lem}

\begin{lem} \label{lem-lbv}
There exists $c>0$ such that for any $q>0$ and $\mu\in\mathbb{R}$, and any $x\in\RR$,
\[ \int |u-\mu|^q \ga_x(u) du \ge c(1+|x-\mu|^q).\]
\end{lem}

\begin{lem}\label{lem: help:adapt:1q}
Let $q\in(0,4]$ and suppose $\delta<q\wedge 2$, with $\delta$ as in \eqref{dens-hv}.
There exists a constant $C>0$ such that for a set $R_t\subseteq\{i:\, |\theta_{0,i}|\leq t\}$ and $t\geq0$ arbitrary, we have, with $P_{\theta_0}$-probability tending to one
\begin{align}\label{eq: lem:help1:Lq}
\sum_{ i\in R_t,|X_i|\le \ta(\al)} \frac{\gf}{\phi}(X_i) (|X_i|^q+1)\leq
\begin{cases}
C |R_t| \tau(\al)^{q-\delta} & \text{for $t\leq 1/\tau(\al)$}\\
 C |R_t| e^{\tau(\al)^2/4}\tau(\al)^{q-\delta}\left[1+ M_n \frac{e^{\tau(\al)^2/8}}{|R_t|^{1/2} \tau(\al)^{1/2}}\right] & \text{for $t\leq \tau(\al)/4$} 
\end{cases}
\end{align} 
where $M_n$ is an arbitrary sequence such that $M_n\rightarrow\infty$.
\end{lem}

\begin{lem}\label{lem: help:adapt:3q}
Let $q\in(0,4]$. Let $t_1\ge 0$ and $t_2\ge\max\{2,2t_1\}$. For $R_{t_1}\subseteq\{i:\, |\theta_{0,i}| \leq t_1\}$, we have
with $P_{\theta_0}$-probability tending to one that
\begin{align*} 
\sum_{i\in R_{t_1}} (|X_i|^q+1)\1_{|X_i|>t_2}\leq \tilde{c}_0  |R_{t_1}| t_2^{q-1} e^{-(t_2-t_1)^2/2}+ M_n t_2^{(q+1)/2} (|R_{t_1}|e^{-(t_2-t_1)^2/2})^{1/2},
\end{align*}
for arbitrary $M_n\rightarrow\infty$ and $\tilde{c}_0= (2^{q}\vee2) [4((2-q)\vee 1)+2]/\sqrt{2\pi}$.
\end{lem}

\begin{lem}\label{lem: help:adapt:2q}
Under the excessive-bias restriction EB(q) \eqref{condition: EB_q} the size of the set $R_t=\{i:\, t\leq |\theta_{0,i}| \}$ is bounded from above by $c_0(C_qD_qt^{-q}\log^{q/2} (n/\tilde{s})+1)\tilde{s}$, with $\tilde{s}$ as in \eqref{def: equiv:tilde:s}.
\end{lem}

\section{Proofs for the excessive-bias assumption}\label{sec: proof:eb}

\begin{proof}[Proof of Lemma \ref{lem:compa}]
First we show that for every $i=1,...,\log_2 s-1$ and $c>0$
\begin{align}
\cT[s_i,s_{i+1};c]\subset \Theta_0^2[s;\sqrt{c/2},1,c].\label{eq: inclusion1}
\end{align}
Take any $\theta\in \cT[s_i,s_{i+1};c]$ and note that $\theta_{(s_i+1)}^2\geq c\log n$ (where $\theta_{(j)}$ denotes the $j$th decreasingly ordered value of the parameter of interest). Let us denote by $I$ the largest index satisfying $\theta_{(I)}^2\geq c \log (n/s_i)$ and note that $I\in\{s_i+1,...,s_{i+1}\}$. Then one can also see that this $I$ satisfies the condition \eqref{condition: EB}, since $\big|\{j:\, \theta_{j}^2\geq  c \log (n/I)\}\big|\geq \big|\{j:\, \theta_{j}^2\geq  c \log (n/s_1)\}\big| = I$ and $\sum_{|\theta_j|\leq \sqrt{c\log (n/I)}}\theta_i^2\leq   (s_2-I)c\log(n/I)\le s_1c\log(n/s_1)$, for $I\leq n/e$, using $s_2\le 2 s_1$ and $I\ge s_1+1$. 
 To show the strict inclusion, let $\te_0$ be defined as
\begin{align*}
\theta_{0,j}^2=
\begin{cases}
n, & \text{for}\, 1\leq j\leq s_i,\\
1, & \text{for}\, s_i+1\leq j \leq 2s_i,\\
0, & \text{for}\, 2s_i+1  \leq j,
\end{cases}
\end{align*}
for any $i=1,...,\log_2 s -1$. Then  $\|\te_0-\ell_0(s_i)\|_2^2\leq \sum_{j=s_i+1}^{n}\te_{0,j}^2= s_i < c2s_i\log n$ so $\te_0\notin \cT[s_i,s_{i+1};c]$ and therefore $\te_0\notin \cT_d[s;c]$. Furthermore, by choosing $\ell=s_i$ we have that $\big|\{j:\ |\theta_{0,j}|\geq \log (n/\ell)\}\big|\geq \ell$ and $\sum_{|\theta_{0,j}|\leq \sqrt{c\log (n/\ell)}}\theta_{0,i}^2= \ell \le c\ell\log (n/\ell)$, satisfying the excessive-bias restriction with parameters given in the lemma .
\end{proof}

\begin{proof}[Proof of Theorem \ref{thm: info:theoretic:bounds}]

First consider the result under condition \eqref{condition: EB_weaker3}. We show below that
\begin{align}
\cT[s_1,s_2;c ]\subset \Theta_0^2[s_3; \sqrt{c/2},1,1]\label{inclusion:EB}
\end{align}
for every $c>0$, and $s_1<s_2\leq s_3=o(\sqrt{n}/\log{n})$, satisfying $s_2<c^{-1}s_1$. This inclusion is similar in spirit to \eqref{eq: inclusion1}, but allows that $s_2\geq 2s_1$ (for $c<1/2$), which was required in $\eqref{eq: inclusion1}$.  It is in particular true for $c=m_n^2=o(1)$ and $s_1,s_2$ as in the theorem, hence $\cT[s_1,s_2;m_n^2] \subset \Theta_0^2[s; m_n/\sqrt{2},1,1]$. Then the proof of the statement simply follows from of Theorem 4(A) of \cite{nickl:geer:2013} (the authors consider the loss $\|\cdot\|_*^2=n\times d_2$, but the same argument goes through with the loss $\|\cdot\|_*^2= d_2$ in the sequence model), where it is shown that it is not possible to construct adaptive confidence sets over the classes $\{\theta\in\ell_0[s_2]:\, \|\theta-\ell_0[s_1]\|_2^2\geq  m_n^2 s_2\log n\}$ and $\ell_0[s_1]$ (more precisely in the proof it is shown that it is not possible to construct a confidence set with size bounded by a multiple of $s_2\log n$ over  $\{\theta\in\ell_0[s_2]:\, \|\theta-\ell_0[s_1]\|_2^2\geq  m_n^2 s_2\log n\}$ and by $s_1\log n$ for $\theta=0$, which completes the proof since the zero signal  satisfies the excessive-bias assumption with $\tilde{s}=0$).

We now prove assertion \eqref{inclusion:EB}, along the lines of $\eqref{eq: inclusion1}$. We highlight here only the differences. Note that in view of the proof of Lemma \ref{lem:compa}, for every $\theta\in \cT[s_1,s_2;c ]$, we have $\sum_{|\theta_i|\leq \sqrt{c\log (n/I)}}\theta_i^2\leq (s_{2}-I)c \log (n/s_2)\leq   I\log (n/I)$ (for $I\leq n/e$), where we use $cs_2<s_1$ and $I\in\{s_1+1,...,s_2\}$.
 
Next we deal with the conditions \eqref{condition: EB_weaker1} and \eqref{condition: EB_weaker2}.  In these cases one can not directly apply the results of \cite{nickl:geer:2013}, since elements of the set $\{\theta\in\ell_0[s_2]:\, \|\theta-\ell_0[s_1]\|_2^2\geq  m_n^2 s_2\log n\}$ will not necessarily satisfy the excessive-bias assumptions \eqref{condition: EB_weaker1} and \eqref{condition: EB_weaker2} (consider for instance a signal with $n-s_2$ and $s_2$ coefficients of size zero and $m_n\sqrt{\log n}$, respectively). The proof of the nonexistence result combines then ideas from  \cite{baraud:2002} and \cite{nickl:geer:2013}, and adapts them to the present setting.

We argue by contradiction. Let us assume that under \eqref{condition: EB_weaker1} or \eqref{condition: EB_weaker2} one can construct confidence sets satisfying assertions \eqref{eq: eb:coverage} and \eqref{eq: eb:adapt}.  
We show below that this would imply the existence of a test $\psi$ such that
\begin{align}
\varlimsup_n \Big(E_{\theta_0}\psi+ \sup_{\te\in B_{0}} E_\te(1-\psi)\Big) \leq\gamma'+ 2\gamma,\label{eq: ub:test:eb}
\end{align}
where the signal $\theta_0=(\theta_{0,1},...,\theta_{0,n})$ and the set $B_0$ respectively are defined by 
\begin{equation*}
\theta_{0,i}=
\begin{cases}
 A  \sqrt{2\log(n/s_1)},&\text{for $i\leq s_1$}\\
0,&\text{else}, 
\end{cases}
\end{equation*}
\[ 
B_{0}=\Big\{\theta\in\ell_0[s_2]:\ \theta_1=...=\theta_{s_1}= A  \sqrt{2\log(n/s_1)},\ \left|\left\{i:\, \theta_i^2= c^2 2\log(n/s_2)\right\}\right|=s_2-s_1\Big\},\]
with $c^2<2\eps/(1+2\eps)<1$, 
i.e. the parameters in $B_{0}$ have signal strength $ A  \sqrt{2\log(n/s_1)}$ in the first $s_1$ coordinates, amongst the rest of the coefficients $s_2-s_1$ is of squared size $2c^2 \log(n/s_2)$, while the rest of the coefficients are zero. Assuming further that $2c^2<1$, note that $\theta_0$ and any $\theta\in B_{0}$ satisfy both of the conditions  \eqref{condition: EB_weaker1} (with $\ell=s_1$ and $C_2=1$) and \eqref{condition: EB_weaker2} (with $\ell=s_2$ and $D_2=1$). We also show below that for any $\gamma_1\in (0,1)$, $\eps_1\in(0,1-\gamma_1)$,
\begin{align}
 \varliminf_n \inf_{\phi_{\gamma_1}} \sup_{\te \in B_{0}} E_\te(1-\phi_{\gamma_1}) \ge \eps_1,\label{eq: lb:test:eb}
\end{align}
where the infimum is taken over every test of level $\gamma_1$. This leads to a contradiction with $\eqref{eq: ub:test:eb}$ (noting that $\gamma,\gamma'<1/3$).

First we verify assertion \eqref{eq: ub:test:eb}, by constructing a test $\psi$ satisfying
\begin{align}
\varlimsup_n E_{\theta_0} \psi & \le \gamma'+\gamma,  \label{m_err1_eb} \\
\varlimsup_n \sup_{\te\in B_{0}} E_\te(1-\psi) & \le \gamma. \label{m_err2_eb}
\end{align}  
Let us consider the test $\psi=\1\{\cC_n(X) \cap B_{0} \neq \emptyset\}$, and using \eqref{eq: eb:coverage},
\[ \varlimsup_n \sup_{\te\in B_{0}} E_\te(1-\psi) 
=\varlimsup_n \sup_{\te\in B_{0}}
P_\te( \cC_n(X)\cap B_{0} = \emptyset ) 
\le \varlimsup_n\sup_{\te\in B_{0}}
P_\te( \te\notin \cC_n(X) ) \le \gamma,
\]
while one also has, using the diameter bound \eqref{eq: eb:adapt} and coverage \eqref{eq: eb:coverage},\begin{align*}
\varlimsup_n E_{\theta_0}\psi & = 
\varlimsup_n P_{\theta_0}({\theta_0}\in \cC_n(X), \cC_n(X)\cap B_{0} \neq \emptyset)+
\varlimsup_n P_{\theta_0}({\theta_0}\notin \cC_n(X), \cC_n(X)\cap B_{0} \neq \emptyset) \\
& \le \varlimsup_n P_{\theta_0}(\text{diam}(\cC_n(X))\ge L s_1 \log(n/s_1)) + \varlimsup_n P_{\theta_0}({\theta_0}\notin \cC_n(X))\le \gamma'+\gamma,
\end{align*}
where in the first inequality we have used that 
$$\inf_{\theta\in B_0}\|\theta_0-\theta\|_2^2\geq 2c^2(s_2-s_1)\log(n/s_2)\gg Ls_1 \log(n/s_1).$$

Hence it remained to verify assertion \eqref{eq: lb:test:eb}. The minimax risk over $B_0$ in \eqref{eq: lb:test:eb} is bounded from below by any Bayes risk for a prior distribution on $B_0$.  Let us define a specific prior $\Pi$ on the set $B_0$ as follows. Let the first $s_1$ coordinates be fixed to the value $A  \sqrt{2\log(n/s_1)}$, and next let $S$ be sampled uniformly at random over subsets of cardinality $s_2-s_1$ among the remaining coordinates $\{s_1+1,...,n\}$. Let $\{\epsilon_j\}$ be i.i.d. Rademacher and, given $S$, set
\[ \te_j=c\sqrt{2\log (n/s_2)} \epsilon_j \qquad \text{if } j\in S,\]
and $\te_k=0$ otherwise. For convenience let us introduce the notation $\lambda=c\sqrt{2\log (n/s_2)}$.  The corresponding marginal likelihood ratio $L_\Pi(Y)=\int (dP_\te/dP_{\theta_0})(Y)d\Pi(\te)$ is
\begin{align*}
L_\Pi(Y)=\frac{1}{{n-s_1\choose s_2-s_1}}\sum_{S\in \mathcal{S}(s_2-s_1,n-s_1)}E_{\epsilon|S}\big[ \exp\big( -(s_2-s_1)\lambda^2/2+ \lambda \sum_{j\in S}\epsilon_j Y_j \big) \big],
\end{align*}
where $ \mathcal{S}(s_2-s_1,n-s_1)$ denotes the subsets of size $s_2-s_1$ of a set of size $n-s_1$ of $\{s_1+1,...,n\}$, and $E_{\epsilon|S}$ denotes the expected value corresponding to the iid Rademacher random variables $\epsilon=\{\epsilon_j:\, j\in S\}$. Let us introduce the notation $K(Y_{s_1+1})=E_{\epsilon|S}[\exp( - \lambda^2/2  +\lambda \epsilon_{j} Y_{j})]$, for some $j\in S$ and
\[a = E_{\theta_0}[K(Y_{s_1+1})^2], \quad b =E_{\theta_0}[K(Y_{s_1+1})].\]
Note that $a,b$ do not depend on $\te_0$, since $\te_{0,s_1+1}=0$.  Then
\begin{align*}
E_{\theta_0}[L_\Pi(Y)^2] &= \frac{1}{{n-s_1\choose s_2-s_1}^2}\sum_{S,S'\in  \mathcal{S}(s_2-s_1,n-s_1)} E_{\theta_0} \Big[\prod_{j\in S}E_{\epsilon|S} \exp(- \lambda^2/2 +\lambda\epsilon_j Y_j )\times\\
&\qquad\qquad \prod_{j\in S'} E_{\epsilon|S'} \exp(- \lambda^2/2  +\lambda\epsilon_j Y_j ) \Big]\\
&=\frac{1}{{n-s_1\choose s_2-s_1}^2}\sum_{S,S'\in  \mathcal{S}(s_2-s_1,n-s_1)} E_{\theta_0} \Big[\prod_{j\in S\cap S'} E_{\epsilon|S\cap S'} \exp(- \lambda^2  +2\lambda\epsilon_j Y_j )\times\\
 &\qquad\qquad\prod_{j\in S\Delta S'} E_{\epsilon|S\Delta S'} \exp(- \lambda^2/2 +\lambda\epsilon_j Y_j )  \Big]\\
&=\frac{1}{{n-s_1\choose s_2-s_1}^2}\sum_{S,S'\in  \mathcal{S}(s_2-s_1,n-s_1)} a^{|S\cap S'|}b^{|S\Delta S'|}
=b^{2(s_2-s_1)}\sum_{j=1}^{s_2-s_1} (a/b^2)^j p_{j,s_2-s_1,n-s_1},
\end{align*}
where  $A\Delta B = \{A\backslash B\}\cup \{B\backslash A\}$ and  $p_{j,s_2-s_1,n-s_1}={n-s_1\choose s_2-s_1}^{-2}|\{(S,S')\in \mathcal{S}(s_2-s_1,n-s_1)^2:\, |S\cap S'|=j \} |$. One can notice that the random variable $X$ satisfying $P(X=j)=p_{j,s_2-s_1,n-s_1}$ has a hypergeometric distribution with parameters $n-s_1,s_2-s_1$ and $(s_2-s_1)/(n-s_1)$ and the right hand side of the preceding display can be written as $b^{2(s_2-s_1)} E[(a/b^2)^X]$. Then in view of (27) of \cite{baraud:2002} (with $\cosh(\lambda^2)$ replaced by $a/b^2$ and $k$ by $s_2-s_1$), one obtains 
\begin{align} 
E_{\theta_0}[L_\Pi(Y)^2] &\leq b^{2(s_2-s_1)} \exp\left((s_2-s_1)\log\Big[1 +\frac{s_2-s_1}{n-s_1}(\frac{a}{b^2}-1)\Big]\right)\nonumber\\
&\leq b^{2(s_2-s_1)} \exp\left(\frac{(s_2-s_1)^2}{n-s_1}(\frac{a}{b^2}-1)\right).\label{eq: ub:likelihood:ratio}
\end{align}
Note that in view of $E_{\theta_0}\cosh(\lambda Y_{s_1+1})=e^{\lambda^2/2}$ for any $\lambda\in\mathbb{R}$ we have
\begin{align*}
b&= e^{-\lambda^2/2}E_{\theta_0}\cosh(\lambda Y_{s_1+1})=1.
\end{align*}
Furthermore, in view of $E_{\theta_0}\cosh^2(\lambda Y_{s_1+1})=e^{\lambda^2}\cosh(\lambda^2)$
\begin{align*}
a=E_{\theta_0}\big(e^{-\lambda^2/2}\cosh(\lambda Y_{s_1+1})\big)^2=\cosh(\lambda^2)
= (1/2+o(1))(n/s_2)^{2c^2}.
\end{align*}
By noting that $s_1=o(s_2)$ and $s_2\leq n^{1/2-\eps}$ and substituting the preceding two displays into \eqref{eq: ub:likelihood:ratio} one gets, 
using $(s_2-s_1)/(n-s_1)\le s_2/n$ and  $1+o(1)\le 2$ for large enough 
$n$,
\begin{align} 
E_{\theta_0}[L_\Pi(Y)^2] &\leq \exp\Big( \frac{s_2^2}{n}(\frac{n^{2c^2}}{2s_2^{2c^2}}-1)(1+o(1)) \Big)\nonumber\\
&\leq 1+s_2^{2-2c^2}n^{2c^2-1}\leq 1+n^{c^2(1+2\veps)-2\veps}=1+o(1).\label{eq: ub:loglike}
\end{align}
Finally, in view of (24)  of \cite{baraud:2002} (and the display below (24)  in \cite{baraud:2002}), we obtain that
\begin{align*}
\inf_{\phi_{\gamma_1}} \sup_{\te \in B_{0}} E_\te(1-\phi_{\gamma_1})&\geq \inf_{\phi_{\gamma_1}}\int_{B_0} E_{\theta} (1-\phi_{\gamma_1})\Pi(d\theta)\\
&\geq 1-\gamma_1- \|\int_{B_0} P_{\theta}\Pi(d\theta)-P_{\theta_0} \|_{TV}/2\\
&\geq 1-\gamma_1- (E_{\theta_0}[L_\Pi(Y)^2]-1)^{1/2}/2 \geq 1-\gamma_1-o(1),
\end{align*}
where the last display follows from assertion \eqref{eq: ub:loglike}, concluding the proof of assertion \eqref{eq: lb:test:eb} and therefore the proof of the theorem.
\end{proof}

%
%\bibliographystyle{apalike}
%\bibliography{spacs}

\pagebreak

\section{Supplementary material}

This supplement is structured as follows. In Section \ref{supp-lb}, we formally prove the impossibility of existence of adaptive credible sets in the sparse sequence model without further constraints on sparse vectors. The result is close in spirit to the one stated in \cite{nickl:geer:2013}, but the model is different here; we provide a proof  following the testing lower bound in \cite{baraud:2002}.  In Section \ref{supp-thresh}, we give the proofs for the preliminary lemmas on thresholds, moments and risk, while Section \ref{supp-th1} contains the proof of Theorem \ref{thm-risk-dq} of \cite{castillo:szabo:2018:main}. The proof is similar in spirit to that of Theorem 3 in \cite{cm17}, with a few significant differences in terms of estimates, in particular moments in terms of the parameter $\delta$ of the slab prior density.  Sections \ref{sec: bounds_mmle}, \ref{sec: coverage_mmle_q-supp}, and \ref{sec:proofs:technical:credible} contain the proofs of technical results used in the proof of Theorem 3 in \cite{castillo:szabo:2018:main}: the concentration of the empirical Bayes weight estimate, the bound on the diameter of the credible set and a number of technical lemmas, respectively.

\subsection{Lower bounds} \label{supp-lb}
 
In the sequence model, consider the problem of building an adaptive confidence set for the $\ell^2$ norm, in the highly sparse regime where $s=o(\sqrt{n})$. Here we follow the separation approach of Nickl and van de Geer  \cite{nickl:geer:2013} and adapt their result in the sequence model setting. Set, for some sparsity parameters $s_0, s_1$,
\[\tilde{B}_0(s_1,\rho)= \tilde{B}_0(s_1, s_0,\rho) = \{\te\in B_0(s_1),\ \|\te - B_0(s_0)\|\ge \rho \}. \]
We have $\tilde{B}_0(s_1,0)=B_0(s_1)$.

A confidence set that would be adaptive with respect to regularities $s_0, s_1$ would need to verify the following conditions with $\rho_n=0$, for some $L>0$, $\al, \al'>0$, and sparsity parameters $s_0, s_1$,
\begin{align}
\varliminf_n & \inf_{\te\in B_0(s_0)\cup \tilde{B}_0(s_1,\rho_n)} 
P_\te[\te\in \cC_n(X)]   \ge 1-\al, \label{m_coverage}\\
\varlimsup_n & \sup_{\te\in B_0(s_0)} 
P_\te[|\cC_n(X)|^2 >Ls_0\log(n/s_0) ]  \le \al', \label{m_diam0}\\
\varlimsup_n & \sup_{\te\in \tilde{B}_0(s_1,\rho_n)} 
P_\te[|\cC_n(X)|^2 >Ls_1\log(n/s_1) ]  \le \al'.
\label{m_diam1}
\end{align}
\begin{thm} \label{m_cs}
Suppose $s_0=n^{a_0}$ and $s_1=n^{a_1}$ with $0<a_0<a_1<1/2$. Suppose that  for some separation sequence $\rho_n\ge 0$ and some $0<\al,\al'<1/3$, an adaptive confidence set in the sense of \eqref{m_coverage}--\eqref{m_diam1} exists. Then necessarily
\[ \varliminf_{n}\, \frac{\rho_{n}^2}{s_1\log{n/s_1}} > 0. \]
\end{thm}
In particular, it follows from Theorem \ref{m_cs} that fully adaptive confidence sets do not exist (if so, one could take $\rho_n=0$ in Theorem \ref{m_cs}) and the result quantifies, if one removes difficult parameters as in $\tilde{B}_0(s_1,\rho_n)$, how much should be removed.

\begin{proof}
Let $\rho^*=\rho_n^*=(s_1\log{(n/s_1)})^{1/2}$. Reasoning by contradiction, suppose a confidence set satisfying \eqref{m_coverage}--\eqref{m_diam1} exists for some sequence $\rho=\rho_n$ such that
\[ \varliminf \frac{\rho}{\rho^*} = 0. \]
One can always assume that the limit exists, otherwise one can argue along a subsequence. One can then find a further sequence $\bar\rho=\bar\rho_n$ such that $\bar\rho =o(\rho^*)$, $\rho\le \bar\rho$ and $s_0\log(n/s_0)=o(\bar\rho^2)$, since $s_0=o(s_1)$.

Assuming this, one first shows that this would imply the existence of a test $\psi$ such that
\begin{align}
 \varlimsup\, E_0\psi & \le \al'+\al,  \label{m_err1} \\
  \varlimsup\, \sup_{\te\in \tilde{B}_0(s_1,\bar{\rho})} E_\te(1-\psi) & \le \al. \label{m_err2}
\end{align}  
Indeed, setting $H_0 = \{\te=0\}$, $H_1=\{\te\in \tilde{B}_0(s_1,\bar\rho)\}$ and $\psi=\1\{\cC_n(X)\cap H_1 \neq \emptyset\}$, and using \eqref{m_coverage},
\[ \sup_{\te\in H_1} E_\te(1-\psi) 
= \sup_{\te\in \tilde{B}_0(s_1,\bar\rho)}
P_\te( \cC_n(X)\cap \tilde{B}_0(s_1,\bar\rho) = \emptyset ) 
\le \sup_{\te\in \tilde{B}_0(s_1,\bar\rho)}
P_\te( \te\notin \cC_n(X) ) \le \al,
\]
while one also has, using \eqref{m_diam0} (combined with $s_0\log(n/s_0)=o(\bar\rho^2)$) and coverage \eqref{m_coverage},
\begin{align*}
E_0\psi & = 
P_0(0\in \cC_n(X), \cC_n(X)\cap H_1 \neq \emptyset)+
P_0(0\notin \cC_n(X), \cC_n(X)\cap H_1 \neq \emptyset) \\
& \le P_0(|\cC_n(X)|\ge \bar \rho) + P_0(0\notin \cC_n(X))\le \al'+\al.
\end{align*}
By Theorem 1 of \cite{baraud:2002}, for any $\al_1\in (0,1)$, $\delta\in(0,1-\al_1]$, and any $0\le k\le n$,
\[ \inf_{\phi_{\al_1}} \sup_{\te \in B_0(k),\ \|\te\|=\rho} E_\te(1-\phi_{\al_1}) \ge \delta,\]
where the infimum is taken over all tests of level $\al_1$, whenever $\rho\le r_{n,k}$, where, if $\eta=2(1-\al_1-\delta)$, $\cL(\eta)=\log(1+\eta^2)$,
\[ r_{n,k}^2 = k \log\left(1+\cL(\eta)\frac{n}{k^2} \right).\]
In particular here we take $k=s_1=n^{a_1}$ and the result holds for $\rho^2\le r_*^2:=  c s_1\log(n/s_1)$ for a suitably small $c=c(a_1)$ (assuming $\al_1+\delta<1/2$ - or far away from $1$ - say). To prove the above testing rate, Baraud \cite{baraud:2002} bounds from below the Bayesian error for the prior $\Pi=\Pi_\rho$ as follows. One first draws $R_1,\ldots,R_n$ independent Rademacher variables and, independently, one selects a subset $S\subset\{1,\ldots,n\}$ uniformly at random among subsets of size $k$. Given $\{R_k\}, S$, one sets  $\te_i=\rho R_i/\sqrt{k}$ if $i\in S$, and $\te_i=0$ otherwise. Baraud then proves, for $\rho\le r_{n,k}$ (that is, when $k=s_1$, for $\rho\le r^*$),
\[ \inf_{\phi_{\al_1}} \int E_\theta(1-\phi_{\al_1}) d\Pi(\theta) \ge 1-\al_1-\eta/2=\delta.
\]
One now notes that $\Pi$ satisfies, for $k=s_1$ and $\rho=2\bar{\rho}=o(r^*)$ (as $\bar \rho=o(\rho^*)$  by assumption), for $n$ large enough, 
\[ \Pi[\tilde B_0(s_1,\bar\rho) ]=\Pi[(s_1-s_0)\rho^2/s_1 \sbl{\ge} \bar\rho^2] =1.\]

Now applying the result by \cite{baraud:2002} above for $k=s_1$ and $\rho=\sbl{2}\bar\rho$ as before leads to, for any test $\psi$ of asymptotic level $\al+\al'$ in the sense of \eqref{m_err1} (so, for $n\ge N_\psi$ large enough, $E_0\psi\le \al+\al'$ so $\psi$ is of level $\al+\al'$ for large $n$), for $n\ge N_\psi$ and $\alpha_1=\alpha+\alpha'$,
\[ \sup_{\te\in \tilde{B}_0(s_1,\bar{\rho})} E_\te(1-\psi) 
\ge \int E_\theta(1-\psi) d\Pi(\theta) \ge \inf_{\phi_{\al_1}} \int E_\theta(1-\phi_{\al_1}) d\Pi(\theta)\ge \delta.
\] 
Taking $\delta= (1+1/10)\al$ (and assuming $(1+1/10)\al\le 1-(\al+\al')$ so that $\delta\le 1-\al_1$)   gives a contradiction with respect to  \eqref{m_err2}.
\end{proof}

\subsection{Proofs for useful thresholds and fixed $\alpha$ risk bounds} \label{supp-thresh}

\begin{proof}[Proof of Lemma \ref{lem-gf}]
The derivative $(\log \ga)'(u)$ is bounded in absolute value for $|u| \ge 1$ by a universal constant (e.g. $2$ if, say, $\delta\le 1$), and $\ga$ is unimodal and symmetric, so the fact that $\gf\asymp\ga$ and the monotonicity of $\gf/\phi$ follow from the proof of Lemma 1 in \cite{js04}, which does not use any moment bound on $\gamma$, so immediately extends to the density \eqref{dens-hv}. The second estimate is immediate using the first one and the fact that $\int_y^\infty u^{-1-\delta}du=y^{-\delta}/\delta$. 
\end{proof}

\begin{proof}[Proof of Lemma \ref{lemzetaall}]
We write $\zeta=\zeta(\al)$ within the proof to simplify the notation.

By definition of $\zeta$, one has $(g/\phi)(\zeta)=\al^{-1}-1\ge \al^{-1}/2$, for $\alpha<1/2$.  Deduce, as $g$ is bounded on $\RR$, that $\phi(\zeta)\le C\al$. Hence taking logarithms $\zeta\ge (2\log\{1/(C\al)\})^{1/2}=:\zeta^*$. As $g$ is decreasing on $\RR^+$, deduce $g(\zeta)\le g(\zeta^*)$. Using the relation above between $g,\phi$ and $\al$, deduce $\phi(\zeta)\le \al g(\zeta^*)$. So, 
\[ \zeta^2 \ge -2\log\{\sqrt{2\pi} \al g(\zeta^*)\}.\]
As $\zeta^*$ goes to infinity as $\al\to 0$, the term $-2\log{g(\zeta^*)}$ becomes larger than $-\log(2\pi)$, which leads to the announced lower bound by taking $\al$ small enough.

From Lemma \ref{lem-gf}, tails of $g$ and $\ga$ are the same, so $g(x)\geqa x^{-1-\delta}$ using \eqref{dens-hv}. 
From this one gets $(\phi/g)(\zeta)\le C\zeta^{1+\delta}\exp\{-\zeta^2/2\}\le C\exp\{-\zeta^2/4\}$ for large enough $\zeta>0$ (by the obtained lower bound on $\zeta$, it suffices that $\alpha$ is small enough). 
Deduce that $\zeta^2/4\le \log\{C(g/\phi)(\zeta)\}$, which gives 
\[ \zeta\le 2\sqrt{  \log\{C(g/\phi)(\zeta)\}} 
= 2\sqrt{\log{C}+\log(1+\al^{-1})}\le 2\sqrt{\log(1/\al)}+C',\]
where we use $\sqrt{a+b}\le \sqrt{a}+\sqrt{b}$ and $\log(1+x)\le 1+\log{x}$ for large $x$.  
Reinserting this back into the second inequality of this paragraph, which can be rewritten as 
$\zeta^2/2\le (1+\delta)\log{\zeta}+\log\{C(\al^{-1}+1)\}$, leads to the result for $\zeta(\al)$.

To prove that the same statement holds for $\tau(\al)$ and $t(\al)$ note that following from the definition of $\tau(\al)$ and $\zeta(\al)$ we have $\tau(\al)\le \zeta(\al)$, as well as, for $\al\le 1/4$,  that $\zeta(2\al)\le \tau(\al)$.
Finally, from page 1622 of \cite{js04} (independent of tails of $\ga$), we have that $\zeta(\al)^2-c\leq t(\al)^2\leq\zeta(\al)^2$.
\end{proof}  

\begin{proof}[Proof of Lemma \ref{lem-momgf}] 
The monotonicity properties follow from the explicit expression of $B(x,\al)$ as well as the fact that $B(\cdot)$ is increasing. 
Symmetry of $\ga$ and $\phi$ imply that $B=(g/\phi)-1$ is symmetric, from which symmetry of $m_1(\ta,\al)=E[ B(Z+\ta,\al) ] = \int B(x,\al)\phi(x-\ta)dx$ (in $\tau$)  follows. Next by dominated convergence, 
\[ \frac{\partial}{\partial\ta} m_1(\ta,\al)
= \int_{-\infty}^\infty \frac{ B'(x+\ta) }{ (1+\al B(x+\ta))^2 } \phi(x)dx =
 \int_0^\infty \frac{ B'(x) }{ (1+\al B(x))^2 } (\phi(x-\ta)-\phi(x+\ta))dx, \]
which is nonnegative for $\ta\ge 0$, as $B'(x)=2\int_0^\infty \text{sinh}(xt) te^{-t^2/2}\ga(t)dt\ge 0$ for $x\ge 0$.

Using that $\int \bef \phi =0$, one can rewrite, as in \cite{js04}, $\tmf$ as 
\[ \tmf(\al) = 2 \int_0^\infty \frac{\al \bef(z)^2}{1+\al\bef(z)} \phi(z)dz \]
from which it follows that, separating into $z\le \zeta$ anf $z>\zeta$,
\[ \tmf(\al) \asymp \int_0^\zeta \al \bef(z)^2\phi(z)dz + \int_\zeta^\infty \bef(z)\phi(z)dz. \]
The first term is dealt with using the estimate of Corollary 1 of \cite{js04}, valid for any $\ga$ log-Lipshitz on $\RR$ which is the case here, which leads to, for $\zeta$ large enough, or equivalently $\al$ small enough, 
\[\al\int_0^\zeta  \bef(z)^2\phi(z)dz \le \al \frac{C}{\zeta} \frac{\gf(\zeta)^2}{\phi(\zeta)}
\leqa  \frac{\gf(\zeta)}{\zeta}.\]
For the second term, noting that $\bef\phi \sim \gf$ and using Lemma \ref{lem-gf} for the tail bound on (minus) the primitive of $\gf$, one gets that this term is asymptotic to $\gf(\zeta)\zeta/\delta \asymp \zeta^{-\delta}/\delta$ and is always of larger order than the first term for $\delta \le 2$. This proves the claim on $\tmf$.

Now turning to $\mf_1$ and $\mf_2$, note that the global bounds directly follow from the fact that $|\bef(x,\al)|\leq C \vee \al^{-1}$. The intermediate bounds are derived as in \cite{js04}, since the proofs  involve only the log-Lipschitz property of $\ga$.

For small signals $|\mu|\le 1/\zeta$ and the first moment, one proceeds as in \cite{js04} by (Taylor-) expanding the function $\mu\to m_1(\mu,\al)$ at the order $2$ around 
$\mu=0$. The first derivative in $\mu$ is $0$, since the function is symmetric. The following bound on the second derivative is as in \cite{js04}: on $[-\zeta,\zeta]$ one bounds $\phi''(u)$ by $C(1+u^2)\phi(u)$ and uses $\phi(z-\mu)\le C\phi(z)$ thanks to the fact that $|\mu|\le \zeta^{-1}$,
\begin{align*}
 \left| \frac{\partial^2}{\partial \mu^2} m_1(\mu,\al) \right| 
& \le \int_{-\infty}^{\infty} | \bef(z,\al) \phi''(z-\mu)| dz  \\
& \le \int_{-\zeta}^{\zeta} | \bef(z) | (1+z^2)\phi(z) dz 
+ \frac2{\al}\int_{|z|>\zeta} \phi''(z-\mu)dz \\
& \leqa \int_{-\zeta}^{\zeta}  \gf(z) (1+z^2) dz + \frac1{\al}\int_{|z|>\zeta} \phi''(z-\mu)dz= (i) + (ii).
\end{align*}
The term (ii) is, as in \cite{js04}, eq. (96), bounded by a constant times $\zeta\bef(\zeta)\phi(\zeta)\le \zeta \gf(\zeta)\le C \zeta^{-\delta}$. The integral defining (i) can be separated in $|z|\le 2$, for which it is bounded by a constant, and $|z|>2$, part on which one integrates by part to obtain
\[ \int_2^\zeta \gf(z) z^2dz \le C\int_2^\zeta z^{1-\delta} dz = \frac{\zeta^{2-\delta}}{2-\delta}. \] 
One concludes that the term $(i)$ dominates in the expression of the second derivative, and the bound for $m_1$ follows by a Taylor expansion. 

The bound for the second moment and small signal is obtained by separating again $|z|\le \zeta$ and 
$|z|>\zeta$ to obtain  
\begin{align*}
m_2(\mu,\al) & \le C\int_{-\zeta}^\zeta \bef(z)^2\phi(z-\mu)dz + \frac{1}{\al^2} \int_{|z|>\zeta} \phi(z-\mu)dz\\
& \le C\int_0^\zeta \bef(z)^2\phi(z)dz + \frac{2}{\al^2} \frac{\phi(\zeta-|\mu|)}{\zeta-|\mu|}\\
& \le C\frac{\gf(\zeta)}{\zeta}\frac{1}{\al} + C\frac{\gf(\zeta)}{\zeta}\frac{1}{\al}.
\end{align*}
By using  $\gf(\zeta)\asymp \delta \tmf(\al)/\zeta$ which follows from the estimate on $\tmf$, the bound on $m_2$ for small $\mu$ follows. Finally, the proof of the equivalent for $\mf_1(\zeta,\al)$ is the same as in \cite{js04}.
\end{proof} 

%\subsection{Proof of Lemmas \ref{lemmomga} and \ref{lem: effective:sparsity} }\label{sec: proof:lem56}
\begin{proof}[Proof of Lemma \ref{lemmomga}]
{\em Bound on moments of $\ga_x$.} Recall $\ga_x(u) = \ga(u)\phi(x-u)/\gf(x)$.  
By definition of $\gf$, for any $x$ we have $\int \ga_x(u) du = 1$ and
$(\log \gf)'(x) = \int (u-x) \ga_x(u) du$, which gives
\[ \int u \ga_x(u) du = x +(\gf'/\gf)(x).  \]
Noting that $\int (u-x)^2 \ga_x(u)du=\gf''(x)/\gf(x)+1$
and decomposing $u^2 = (u-x)^2+2x(u-x)+x^2$,
\begin{align*}
\int u^2 \ga_x(u) du = (\gf''/\gf)(x) + 1 + 2x (\gf'/\gf)(x)+x^2.
\end{align*}
From the definition \eqref{dens-hv} of 
 $\gamma$ if follows $|\ga'|\le c_1\ga$ and $|\ga''| \le c_2 \ga$ for some $c_1, c_2>0$. This leads to 
\[ |\gf'(x)| \leq \int |\ga'(x-u)|\phi(u)du   \le c_1\int \ga(x-u)\phi(u)du=c_1\gf(x) \]
and similarly $|\gf''|\le c_2 \gf$, so that $\int u^2 \ga_x(u) du \le x^2+C$ (the term containing $g'$ is negative).
Similarly, for any $\mu$, 
$\int (u-\mu)^2\ga_x(u)du = (x-\mu)^2 
+  (\gf''/\gf)(x) + 1+ 2(x-\mu) (\gf'/\gf)(x)$. 
This leads to the bound
\[ \int (u-\mu)^2\ga_x(u)du \le (x-\mu)^2 + c_2 +2c_1|x-\mu|\le 2(x-\mu)^2+C. \]
Using H\"older's inequality and manipulating the $q$th powers via \eqref{manipq} leads to
\begin{equation*}%\label{eq: UB_kappa_x} 
\int |u-\mu|^q \ga_x(u)du \le \Big[\int (u-\mu)^2 \ga_x(u)du\Big]^{q/2}\le 
C\left[ |x-\mu|^{q} + 1\right].
\end{equation*}
{\em Bounds on the risk for zero signal.} By definition of $r_q$ together with \eqref{boundsa} and \eqref{eq: UB_kappa_x},
\begin{align*}
 E_0 r_q(\al,0,x) & \leqa \int_{-\ta(\al)}^{\ta(\al)} \frac{\al}{1-\al} \gf(x)(1+|x|^q)dx 
 + \int_{|x| > \ta(\al)}(1+|x|^q)\phi(x) dx 
 \end{align*} 
so that, for $\al$ small enough, using Lemma \ref{lemphi} below as well as $\delta<q$,
\[ E_0 r_q(\al,0,x)
  \leqa \al \ta(\al)^{q-\delta}/(q-\delta) + \ta(\al)^{q-1} \phi(\ta(\al))+ \phi(\ta(\al))/\ta(\al).\]
By the definition of $\ta(\al)$ above, for $\al\le 1/2$ we have $\phi(\ta(\al))\lesssim \al \gf(\ta(\al)) \leqa \al \ta(\al)^{-1-\delta}$ by Lemma \ref{lem-gf}. Deduce, since by assumption $\delta<q$,
\[ E_0 r_q(\al,0,x) \leqa \al \ta(\al)^{q-\delta}/(q-\delta).  \] 
{\em Bounds on the risk for nonzero signal.} Using the preceding bounds,
\[ r_q(\al,\mu,x) 
\le (1-a(x))|\mu|^q + C(1+|x-\mu|^q). 
 \]
 When taking the expectation with respect to $E_\mu$, the second term is bounded by a constant, as $X-\mu\sim\cN(0,1)$ under $E_\mu$. To bound the expectation of the first term, we distinguish two cases, and set 
$T:=\ta(\al)$ for simplicity of notation. If $|\mu|\le 4T$, using $a(x)\ge 0$,
\[ E_\mu[1-a(x)]|\mu|^q \le CT^q.\]
If $|\mu|>4T$, one writes, using the bound \eqref{postwb} on $1-a(x)$, 
\[ E_\mu[1-a(x)]|\mu|^q \le E_\mu[\1_{|x|\le T}]|\mu|^q 
+ E_\mu\left[e^{-(|x|-T)^2/2}\1_{|x|> T}\right]|\mu|^q. \]
The expectation of the first indicator is bounded using that if $Z\sim\cN(0,1)$,
\[ E [\1_{|Z+\mu|\le T}] \le P[|Z|\ge |\mu| - T] \le P[|Z|\ge |\mu|/2]. \]
This implies, with $\psi(u)=\int_u^\infty \phi(u)du\le \phi(u)/u$ for $u>0$, 
\[ E_\mu[ \1_{|x|\le T}] |\mu|^q \leqa |\mu|^{q-1}\phi(|\mu|/2)\le C.\]
Second, if $A=\{ x,\ |x-\mu|\le |\mu|/2 \}$ and $A^c$ denotes its complement,
\[ E_\mu[e^{-\frac12 (|x|-T)^2}]\le \int_{A^c}\frac{1}{\sqrt{2\pi}} e^{-\frac12(x-\mu)^2} dx
+ \int_A \frac{1}{\sqrt{2\pi}} e^{-\frac12 (|x|-T)^2} dx.\]
The first term in the sum is bounded above by $2\psi(|\mu|/2)$. The second term, as $A\subset\{x,\ |x|\ge |\mu|/2\}$, is bounded above by $2\psi(|\mu|/4)$. This implies, in the case $|\mu|>4T$, that $E_\mu[1-a(x)]|\mu|^q  \le C+4\psi(|\mu|/4)|\mu|^q\le C'.$ One concludes that for any $\al\in(0,\al_0]$ and any real $\mu$, 
\[ E_\mu r_q(\al,\mu,x) \le C(1+\ta(\al)^q). \qedhere\]
\end{proof}

\begin{lem} \label{lemphi}
For $M>0$, $q>0$ and $q-1\le M^2/2$, we have with $\bar{c}=(2-q)\vee 1$ that
\[ \int_M^\infty x^q \phi(x)dx \le \bar{c}(M^{q-1}\vee M^{q-3})\phi(M). \]
\end{lem}
\begin{proof}
 The result follows by integration by parts.
 \end{proof}

\subsection{Proof of Theorem \ref{thm-risk-dq}} \label{supp-th1}

The proof is overall quite similar to that of Theorem 3 in \cite{cm17}. The key differences lie in, first, the behaviour of the posterior integrated risk for fixed $\al$ studied in Section 4.2 of the present paper, which turns out to be sensitive to the choice of the tails of the slab distribution in terms of the contribution to the risk of coordinates that are zero: $\te_{0,i}=0$. From the proof of Lemma 5 it can be seen that there is a transition at $\delta=q$ and that values of $\delta$ such that $\delta>q$ lead to a suboptimal convergence rate of the empirical Bayes posterior distribution. Second, the moment estimates (Lemma 4 in Section 4.1 of the main document) are different, due again to the tails of the slab prior. It can be seen that this in particular implies that the MMLE $\hat\al$ is located slightly differently depending on the tails of the slab.

\subsubsection{Risk bounds for random $\al$}

Let us recall that the considered plug-in posterior can be written as
\[ \Pi_{\hat\al}[\cdot\given X] \sim \otimes_{i=1}^n\big( (1-a_{\hat\al}(X_i))
\delta_0 + a_{\hat\al}(X_i) \ga_{X_i}(\cdot)\big),\]
where $\hat \al$ is the MMLE as defined in \eqref{defhal}. Define, for $q\in(0,2]$,
\[ r_q(\hat\al,\mu,x) = \int |u-\mu|^q d\Pi_{\hat \al}(u\given x). \]
The following two Lemmas on the risk $r_q$ are proved in Section \ref{apptech} below.

\begin{lem}[no signal or small signal]\label{lemns}
Let $\al$ be a fixed non-random element of $(0,1)$. Let $\hat\al$ be a random element of $[0,1]$ that may depend on $x\sim \cN(0,1)$ and on other data. Then there exists $C_1>0$ such that
\[ E r_q(\hat\al,0,x) \le C_1\left[\al\ta(\al)^{q-\delta}/(q-\delta) +  P(\hat\al>\al)^{1/2}\right]. \]
For any real $\mu$, it also holds that if $x\sim \cN(\mu,1)$, for some $C>0$,
\[ E r_q(\hat\al,\mu,x) \le |\mu|^q + C.\] 
\end{lem}

\begin{lem}[signal] \label{lemsig}
Let $\al$ be a fixed non-random element of $(0,1)$. Let $\hat\al$ be a random element of $[0,1]$ that may depend on $x\sim \cN(\mu,1)$ and on other data and such that $\tilde\tau(\hat\al)^2\le d\log{n}$ with probability $1$ for some $d>0$. Then there exists $C>0$ such that for all real $\mu$,
\[ E r_q(\hat\al,\mu,x) \le C\left[1 + \tilde\ta(\al)^q +  \{1+ (d\log{n})^{q/2}\}P(\hat\al<\al)^{1/2}\right]. \]
\end{lem}

{\em Bounds for posterior median and mean.} The posterior median $\hat\te_\al=\hat\mu$ for fixed $\al$ satisfies the following bounds, as shown in \cite{js04}, Lemmas 5 and 6, recalling that their proofs extend without difficulty to the case of tails possibly heavier than Cauchy, as is the case for the density $\ga$ in \eqref{dens-hv} we consider here. 
The first bound is typically useful for small signals $\mu$, the second for large signals. We have, for any $q\in(0,2]$,
\begin{align}
E_\mu[|\hat \mu - \mu|^q ] & \leq c_q\left[|\mu|^q + t^{q-1}\phi(t)
+(t^q+b^q+1)P[\hat t < t]^{1/2} \right] \label{pmed1}\\
E_\mu[|\hat \mu - \mu|^q ]& \leq 8\left[ t^q+b^q+1 +
\sqrt{d\log n}P[\hat t > t]^{1/2} \right] \label{pmed2}, 
\end{align}
where $\hat t$ is a random threshold that may depend on $X$ and other data and where for the second inequality one assumes that $\hat t\le \sqrt{d\log n}$ with probability $1$.

\subsubsection{Concentration bounds of the MMLE $\hat\al$}

We derive two concentration bounds on $\hat \zeta=\zeta(\hat\al)$. As $\al\to \zeta(\al)$ is one-to-one, those provide, equivalently, bounds on $\hat\al$. 

{\em Undersmoothing bound on $\hat\zeta$.} One defines $\al_1$ as the quantity that solves the equation, for a small constant $d\in(0,2)$, 
\begin{equation} \label{alq}
   d\al_1 \tmf (\al_1) =s/n,
 \end{equation}   
and $\zeta_1= \bef^{-1}(\al_1^{-1})$. Such a solution exists and is unique, as $\al\to \al \tilde{m}(\al)$ is increasing in $\al$ for $\al$ small enough, and equals $0$ at $0$. 
Also,  provided $c_1$ in \eqref{techsn} is small enough, it can be made smaller than any given arbitrary constant. 

Let us further check that $\al_1$ belongs to the interval $[\al_n,1]$ over which the likelihood is maximised. Indeed, using $\log(n/s)\le \log{n}-2\log\log{n}$, the bound on $\zeta$ from Lemma \ref{lemz} gives that $\zeta_1^2\le 2\log{n}-\frac{3}{2}\log\log{n}$, so that $t(\al_1)\le \zeta(\al_1)=\zeta_1\le \sqrt{2\log{n}}=t(\al_n)$,

\begin{lem} \label{lemzetaunder}
Let $\alpha_1$ be defined by \eqref{alq} for $d$ a given small enough constant and let $\zeta_1$ be given by $B(\zeta_1)=\al_1^{-1}$. Suppose \eqref{techsn} holds. Then for some constant $C>0$, 
\[ \sup_{\te\in\ell_0[s] } P_\te[\hat\zeta < \zeta_1] \le \exp(-Cs). \]
\end{lem}
\begin{proof}[Proof of  Lemma \ref{lemzetaunder}]
Recall the definition of the score function $S(\al)=\sum_{i=1}^n \bef(X_i,\al)$ with $B(X_i,\al)$ as in \eqref{eqbeta}. The map $\al\to S(\al)$ is strictly decreasing with probability one, as the function $\al\to B(X_i,\al)$ is, as soon as $B(X_i)\neq 0$, which happens with probability one. 
Then one notices  that $\{\hat\zeta<\zeta_1\}=\{\hat \al>\al_1\}=\{S(\al_1)>0\}$ 
as well as 
$\{\hat\zeta>\zeta_1\}=\{S(\al_1)<0\}$: the sign of $S$ at any particular $w$ determines on which side of $\hat \al$ the given $\al$ lies. So,
\[ P_{\te}[\hat\zeta < \zeta_1] = P_{\te}[S(\al_1)>0].  \]
For the rest of the proof we denote $\al=\al_1$ and $\zeta=\zeta_1$. The score function is a sum of independent variables. By Bernstein's inequality,  
if $W_i$ are centered independent variables with $|W_i|\le M$ and $\sum_{i=1}^n\text{Var}(W_i)\le V$, then for any $A>0$,
\[ P\left[ \sum_{i=1}^n W_i >A \right] \le \exp\{-\frac12 A^2/(V+\frac13 MA) \}.\]
Set $W_i=\bef(X_i,\al)-\mf_1(\te_{0,i},\al)$ and $A = -\sum_{i=1}^n \mf_1(\te_{0,i},\al)$. Then one can take $M=c_3/\al$, using Lemma \ref{lem-momgf}. 
One can bound $-A$ from above as follows
\begin{align*}
 -A & \le -\sum_{i\notin S_0} \tilde \mf(\al) + \sum_{i\in S_0} \frac{c}{\al}  \le -(n-s) \tilde \mf(\al) + cs/\al \\
 & \le -n \tilde \mf(\al)/2 + cdn\tilde{\mf}(\al) \le -n \tilde \mf(\al)/4,
\end{align*} 
where we have used the definition of $\al=\al_1$. Also, using again Lemma \ref{lem-momgf},
\begin{align}
 V(\al) & \le \sum_{i\notin S_0} \mf_2(0,\al)  
 + \sum_{i\in S_0} \mf_2(\te_{0,i},\al)
 \le  c_4 (n-s) \frac{\tilde \mf(\al)}{\zeta^\kappa \al} 
 + cs/\al^2 \nonumber\\
 & \le \frac{C}{\al}\left[ (n-s) \tilde{\mf}(\al)\zeta^{-\kappa} + cs/\al\right] 
  \le C\al^{-1}\left[ n\tilde{\mf}(\al)\zeta^{-\kappa}/2 + c d n\tilde{\mf}(\al)\right] \nonumber\\
 & \le C' d n \tilde{\mf}(\al)/\al,\label{eq: Lem2: V}
\end{align} 
where one uses that $\zeta^{-1}$ is bounded. This leads to 
\[ \frac{V+\frac13 MA}{A^2} \le \frac{C'd}{n \al \tilde \mf(\al)} + \frac{4c_3}{3n \al \tilde \mf(\al)}\le \frac{c_5^{-1}}{n \al \tilde \mf(\al)}.\]
So,  
$P\left[ S(\al) >0 \right] \le \exp\{ - c_5 n\al\tilde{\mf}(\al) \}$. This concludes the proof using \eqref{alq}.  \qedhere
\end{proof}

{\em Oversmoothing bound on $\hat\zeta$.} 
Let us define, for $\ta\in[0,1]$ and $\mu\in\RR^{n}$, the quantity 
\[ \tilde \pi(\ta;\mu) =\big|\{i:\, |\mu_i|\ge \ta\}\big|/n.\]
Let $\pi_1$ be the proportion, for $\zeta_1$ the pseudothreshold corresponding to 
\eqref{alq},
\[  \pi_1 = \tilde\pi (\zeta_1; \te_0). \]
Let $\al(\ta,\pi)$ be defined by, for a certain $\al_0<1$ and $0\le \pi\le 1$,
\[ \al(\ta,\pi) = \sup\{\al\le \al_0:\ \pi \mf_1(\ta,\al)\ge 2\tmf(\al) \}. \]
and define $\zeta_{\ta,\pi}=\bef^{-1}(1/\al(\ta,\pi))$ as the corresponding pseudo-threshold. Set, for $\zeta_1,\pi_1$ as above,
\begin{equation} \label{al2}
\al_2:=\al(\zeta_1,\pi_1),\qquad \zeta_2=\bef^{-1}(\al_2^{-1}).
\end{equation}
\begin{lem}[\cite{js04}, Lemma 11] \label{lemov}
Suppose \eqref{techsn} holds. Then for  $\zeta_2$  defined by \eqref{al2} and for $C>0$ a constant only depending on $c_1$ in \eqref{techsn},
\begin{equation*}
 P_{\te_0}[\hat\zeta > \zeta_2] 
\le \exp\{-Cn\zeta_2\phi(\zeta_2)\}. 
\end{equation*}
\end{lem}  
  
%%%%%%%%%%%%%%%%%%%%%%%%%%%%%%%%%%%%%%%%

%%%%%%%%%  PROOFS FOR MMLE ALPHA_HAT  %%%%%%%%%%%%

%%%%%%%%%%%%%%%%%%%%%%%%%%%%%%%%%%%%%%%%

\begin{proof}[Proof of Lemma \ref{lemov}]
 In the proof we shall use that $\zeta_2\ge \zeta_1$ and that $\pi_1\zeta_2^2\le C(s/n) \log(n/s)$ as proved in \eqref{zd}--\eqref{zd2} below, whose proof is independent of the present Lemma. 
First note that $\pi_1 \mf_1(\zeta_1,\alpha_2) \ge 2\tilde{\mf}(\alpha_2)$, as $\alpha \to \mf_1(\ta,\alpha)/\tilde{\mf}(\alpha)$ is continuous. 
 So, for  $0<A\le 1$, 
\[ \mf_1(\zeta_1,\alpha_2) \ge \frac{2}{\pi_1^{1-A}} \frac{1}{(\pi_1 \zeta_2^2)^{A}}{\zeta_2^{2A}}\tilde{\mf}(\alpha_2). \]
By  Lemma \ref{lem-momgf}, for $\delta/2\le A\le 1$ (recall that $\delta<2$), 
we have that $\zeta_2^{2A} \tilde{\mf}(\alpha_2)$ is bounded away from $0$, as $\zeta_2\ge \zeta_1$ and $\zeta_1$ goes to infinity with $n/s$. As $\pi_1\le s/n$, this shows that $\mf_1(\zeta_1,\alpha_2)$ goes to infinity with $n/s$. So one can use the reasoning below (88) in \cite{js04} to see that  $E_{\zeta_1}|\bef(X,\alpha_2)|\le C\mf_1(\zeta_1,\alpha_2)$ as well as $\mf_2(\zeta_1,\alpha_2)\le C\alpha_2^{-1}\mf_1(\zeta_1,\alpha_2)$. From there on, one can follow the proof of Lemma 11 in \cite{js04} to obtain the desired deviation bound.
\end{proof}

\subsubsection{Other tools}\label{sec:other:tools}

\begin{lem} \label{lem-lb3}
Let $\bar{\Phi}(t)=\int_t^\infty \phi(u)du$. For $\pi_1, \zeta_1$ as above, a solution $0<\al\le \al_1$ to the equation 
 \begin{equation} \label{wt}
 \bar{\Phi}(\zeta(\al)-\zeta_1) = 8\pi_1^{-1}\al\tilde{m}(\al)
\end{equation}
exists. Let $\al_3$ be the largest such solution. Then for $c_1$ in \eqref{techsn} small enough and $n$ large enough,
\begin{equation} \label{m1}
m_1(\zeta_1,\al_3) \ge \frac14 B(\zeta_3)\bar\Phi(\al_3-\zeta_1).
\end{equation}
\end{lem}   
 
\subsubsection{End of the proof of Theorem \ref{thm-risk-dq}}
 
\begin{proof}[End of the proof of Theorem \ref{thm-risk-dq}]
With $v_{q,\al}(X) = \int  d_q(\te , \hatha) d\Pi_{\al}(\te\given X)$, one splits
\[ v_{q,\al}(X)
=\Big[\sum_{i:\ \te_{0,i}=0} + \sum_{i:\ 0<|\te_{0,i}|\le \zeta_1} +
 \sum_{i:\ |\te_{0,i}| > \zeta_1}\Big] E_{\te_0}\int |\te_i-\te_{0,i}|^q d\Pi_{\hal}(\te_i\given X),\] 
where $\zeta_1$ is defined in \eqref{alq}. We next use the random $\hat{\alpha}$ bounds obtained in Lemma \ref{lemns}
 \begin{align*}
&\lefteqn{\Big[\sum_{i:\ \te_{0,i}=0} + \sum_{i:\ 0<|\te_{0,i}|\le \zeta_1} \Big] E_{\te_0}\int |\te_i-\te_{0,i}|^q d\Pi_{\hal}(\te_i\given X)}\\
& \le C_1 \sum_{i:\ \te_{0,i}=0} \left[\al_1\ta(\al_1)^{q-\delta}/(q-\delta)+ P_{\te_0}(\hal>\al_1)^{1/2}\right]
+ \sum_{i:\ 0<|\te_{0,i}|\le \zeta_1} (|\te_{0,i}|^q + C)\\
& \le C_1\left[(n-s)\al_1\ta(\al_1)^{q-\delta}/(q-\delta)
+(n-s)e^{-C\log^{2}n} \right] + (\zeta_1^q+C)s,
\end{align*}
where for the last inequality we use the first Bernstein bound. From \eqref{alq} one gets 
\[ n\al_1\leqa s  \zeta_1^{-1}g(\zeta_1)^{-1}
\leqa  s\delta \zeta_1^{\delta}.\] 
Note that by assumption $2\delta\le q$, so that $\delta/(q-\delta)\le 1$. Now using Lemma \ref{lemz} and the fact that $t(\al_1)\le \zeta_1$, one obtains that the contribution to the risk of the indices $i$ with $|\te_{0,i}|\le \zeta_1$ is bounded by a constant times $s \log^{q/2}(n/s)$.

It remains to bound the part of the risk for indexes $i$ with $|\te_{0,i}| > \zeta_1$. To do so, we shall use the second Bernstein bound involving the pseudo-threshold $\zeta_2$ in Lemma \ref{lemov}. The following estimates are useful below, with $\eta_n=s/n$,
\begin{align} 
\zeta_1^2 & < \zeta_2^2     \label{zd} \\
\pi_1 \zeta_2^q & \le C\eta_n \log^{q/2}(1/\eta_n).  \label{zd2}
\end{align}
These  are established in a similar way as in \cite{js04} (see also \cite{cm17} for a similar proof and further comments), but with the updated definition of $\al_1, \zeta_1$ from \eqref{alq}, so we include the proof below for completeness. Using the random $\hat\al$ bound for large signals combined with the second Bernstein bound on $\hat\zeta$, one obtains
\begin{align*} 
\sum_{i:\ |\te_{0,i}|>\zeta_1}
 E_{\te_0}\int |\te_i-\te_{0,i}|^q d\Pi_{\hal}(\te_i\given X)  \le  C_2 n \pi_1
\left[ 1+ \zeta_2^q + (1+d\log n)P_{\te_0}(\hal<\al_2)^{1/2}\right]\\
=C_2 n \pi_1
\left[ 1+ \zeta_2^q + (1+d\log n)P_{\te_0}(\hat\zeta>\zeta_2)^{1/2}\right].
\end{align*}
Let us verify that the term in brackets in the last display is bounded above by $C(1+\zeta_2^q)$. If $\zeta_2^{q}>\log{n}$, this is immediate by bounding $P_{\te_0}(\hat\zeta>\zeta_2)$ by $1$. If $\zeta_2^{2}\le \log{n}$,  Lemma \ref{lemov} implies $P_{\te_0}(\hat\zeta>\zeta_2)\le  \exp(-Cn\phi(\zeta_2))\le \exp(-C\sqrt{n})$, so this is also the case.  
Conclude that the last display is bounded above by $Cn\pi_1(1+\zeta_2^q)\le C'n\pi_1\zeta_2^q$. Using \eqref{zd2}, this term is itself bounded by $Cs\log^{q/2}(n/s)$, which  concludes the proof of the theorem, 
given \eqref{zd}--\eqref{zd2}.

We now check that \eqref{zd}--\eqref{zd2} hold.  We first compare $\al_1$ and $\al_2$. For large $n$, as $\al_1\to 0$, the global bound on $m_1$ from Lemma \ref{lem-momgf} is $1/\al$, so that, with $\eta_n=s/n$,
\[ \frac{\mf_1(\zeta_1,\al_1)}{\tilde{\mf}(\al_1)}
\le \frac1{\al_1}\left(\frac{\eta_n}{d\al_1}\right)^{-1}
\le \frac{d}{\eta_n}\le \frac{d}{\pi_1}, \]
using the rough bound $\pi_1\le \eta_n$. As $d < 2$, by the definition of $\alpha_2$ in \eqref{al2}, this means $\al_2< \al_1$ that is $\zeta_1<\zeta_2$, noting that $\al\to m_1(\zeta_1,\al)/\tilde{m}(\alpha)$ is decreasing at least on $(0,\al_2]$, as product of two positive decreasing maps using Lemma \ref{lem-momgf}. 

To prove  \eqref{zd2}, one compares $\zeta_2$ first to a certain $\zeta_3=\zeta(\al_3)$ defined by $\al_3$ solution of 
\[ \bar\Phi(\zeta(\al_3)-\zeta_1) = \frac8{\pi_1}\al_3\tilde{\mf}(\al_3), \]
which exists by Lemma \ref{lem-lb3}. Using the inequality in this lemma  one gets
\[  \frac{\mf_1(\zeta_1,\al_3)}{\tilde{\mf}(\al_3)} \ge
\frac{\frac14 \bef(\zeta_3)\bar\Phi(\zeta_3-\zeta_1)}{\tilde{\mf}(\al_3)}
= \frac1{4\al_3}\frac{8\al_3 \tilde{\mf}(\al_3)}{\pi_1\tilde{\mf}(\al_3)} =\frac{2}{\pi_1}.  \]
This shows that $\al_3\le \al_2$ that is $\zeta_2\le \zeta_3$. Following \cite{js04}, one distinguishes two cases to further bound $\zeta_3$. 

If $\zeta_3>\zeta_1+1$, using $\zeta_2^q\le \zeta_3^q$,
\begin{align*}
 \pi_1\zeta_2^q & \le \zeta_3^q \frac{8\al_3\tilde{\mf}(\al_3)}{\bar{\Phi}(\zeta_3-\zeta_1)} \leqa \zeta_3^{q+1} \frac{\gf(\zeta_3)}{\bef(\zeta_3)} \frac{\zeta_3-\zeta_1}{\phi(\zeta_3-\zeta_1)}\\
 & \le C \zeta_3^{q+2} \frac{\phi(\zeta_3)}{\phi(\zeta_3-\zeta_1)}=C\zeta_3^{q+2}\phi(\zeta_1)e^{-(\zeta_3-\zeta_1)\zeta_1}\\
 & \le C (\zeta_1+1)^{q+2} e^{-\zeta_1}\phi(\zeta_1),
 \end{align*}
 where for the last inequality we have used that $x\to x^{q+2} e^{-(x-\zeta_1)\zeta_1}$ is decreasing for $x\ge \zeta_1+1$. Lemma \ref{lemz} now implies that 
 $\phi(\zeta_1)\leqa \eta_n$. As $\zeta_1$ goes to $\infty$ with $n/s$, one gets $\pi_1\zeta_2^q\leqa \eta_n \zeta_1^{q}$.

If $\zeta_1\le \zeta_3\le \zeta_1+1$, let $\zeta_4=\zeta(\al_4)$ with $\al_4$  solution in $\al$ of
\[ \bar\Phi(1) = 8\al\tilde{\mf}(\al) \pi_1^{-1}.\]
By the definition of $\zeta_3$, since $\bar\Phi(1) \le \bar\Phi(\zeta_3-\zeta_1)$, we have  $\al_4\le \al_3$. Using the bound on $\tilde m$ from Lemma \ref{lem-momgf} as before, one gets
\[ \bar\Phi(1) \leqa \frac{\zeta_4\gf(\zeta_4)}{\bef(\zeta_4)}\pi_1^{-1} 
\leqa \zeta_4\phi(\zeta_4)\pi_1^{-1}. \]
Taking logarithms this leads to $\zeta_4^2\le C +\log(\zeta_4)+ 2\log(\pi_1^{-1})$ and further to 
\[ \zeta_4^2 \le C' + \log\log(\pi_1^{-1}) + 2\log(\pi_1^{-1}).  \]
In particular, the same upper bound holds for $\zeta_2^2$. The function $x\to x(\log(1/x)+\log\log(1/x))^{q/2}$ is increasing, so one gets, using that $\pi_1\le \eta_n$,
\[ \pi_1\zeta_2^q \le C\eta_n\log^{q/2}(1/\eta_n). \qedhere\]
\end{proof}

\subsubsection{Proof of technical Lemmas} \label{apptech}

\begin{proof}[Proof of Lemma \ref{lemns}]
By the explicit expression of $r_q(\hat\al,0,x)$ and \eqref{eq: UB_kappa_x}, we have
\[  r_q(\hat\al,0,x) \le C a_{\hat\al}(x) [1+|x|^q]. \]
Using \eqref{boundsa} and keeping $a_{\hat\al}(x)$ only if  $\hat \al\le \al$ leads to
\begin{align*}
r_q(\hat\al,0,x)
& \le \left[ \frac{\hat \al}{1-\hat \al} \frac{\gf}{\phi}(x) \wedge 1\right] (1+|x|^q)\1_{\hat\al\le \al} + C(1+|x|^q)\1_{\hat\al >\al} \\
& \le \left[ \frac{ \al}{1- \al} \frac{\gf}{\phi}(x) \wedge 1\right] (1+|x|^q)\1_{\hat\al\le \al} + C(1+|x|^q)\1_{\hat\al >\al}.
\end{align*}
For the first term in the last display, one bounds the indicator from above by $1$ and one recognises the term for fixed $\al$ that we have previously bounded   by $\al\ta(\al)^{q-\delta}/(q-\delta)$. 
The first part of the lemma follows by noting that $E_0[(1+|x|^q)\1_{\hat\al >\al}]$ is bounded from above by 
$(2+2E_0[x^{2q}])^{1/2} P(\hat\al>\al)^{1/2}\le C_1 P(\hat\al>\al)^{1/2}$ by Cauchy-Schwarz inequality. The second part of the lemma follows from the fact that
$r_q(\al,\mu,x)  \le (1-a(x))|\mu|^q + a(x)(|x-\mu|^q+C)\le |\mu|^q +|x-\mu|^q+C$ for any $\al$. 
\end{proof}

\begin{proof}[Proof of Lemma \ref{lemsig}]
Using the expression of $r_q(\al,\mu,x)$ and the bound \eqref{postwb} on $1-a(x)$ gives
 \[ r_q(\hal,\mu,x)  \le |\mu|^q \left[\1_{|x|\le \ta(\hal)} 
 +   e^{-\frac12 (|x|-\ta(\hal))^2}\1_{|x| > \ta(\hal)} \right] + C(1+|x-\mu|^q).
\]
It is enough to bound the first term on the right hand side in the last display, as the last two are bounded by a constant under $E_\mu$. Let us distinguish the two cases $\hal\ge \al$ and $\hal<\al$.  

In the case $\hal\ge \al$, as $\ta(\al)$ is a decreasing function of $\alpha$, 
\begin{align*}
& \lefteqn{\left[\1_{|x|\le \ta(\hal)} 
 +   e^{-\frac12 (|x|-\ta(\hal))^2}\1_{|x| > \ta(\hal)} \right] \1_{\hal\ge \al}}\\
 & \le \left[\1_{|x|\le \ta(\al)} + \1_{\ta(\hal)\le |x|\le \ta(\al)} + 
 e^{-\frac12 (|x|-\ta(\hal))^2}\1_{|x| > \ta(\al)} \right] \1_{\hal\ge \al}\\
 & \le 2 \1_{|x|\le \ta(\al)} + e^{-\frac12 (|x|-\ta(\al))^2} \1_{|x| > \ta(\al)},
\end{align*}
where we have used $e^{-\frac12 v^2}\le 1$ for any $v$ and that $e^{-\frac12 (u-c)^2}\le e^{-\frac12 (u-d)^2}$ if $u>d\ge c$. As a consequence, one can proceed as for the fixed $\al$ bound obtained previously to get
\[ E \left[ r_q(\hal,\mu,x)1_{\hal\ge \al}\right] \le C\left[1 + \ta(\al)^q \right].\]

In the case $\hal < \al$, setting $b_n=\sqrt{d\log n}$ and noting that $\ta(\hal)\le b_n$ with probability $1$ by assumption, proceeding as above, with $b_n$ now replacing $\ta(\al)$, one can bound 
\begin{align*}
& \lefteqn{\1_{|x|\le \ta(\hal)} 
 +   e^{-\frac12 (|x|-\ta(\hal))^2}\1_{|x| > \ta(\hal)}}\\
 & \le 2 \1_{|x|\le b_n} + e^{-\frac12 (|x|-b_n)^2} \1_{|x| > b_n}.
\end{align*}
From this one deduces that 
\begin{align*}
&\lefteqn{ E \left( |\mu|^q\left[\1_{|x|\le \ta(\hal)} 
 +   e^{-\frac12 (|x|-\ta(\hal))^2}\1_{|x| > \ta(\hal)} \right] \1_{\hal<\al}\right) }\\
& \le C \left( E_\mu\left[|\mu|^{2q}\1_{|x|\le b_n} + |\mu|^{2q} e^{-(|x|-b_n)^2}\right] \right)^{1/2} P(\hal <\al)^{1/2}.
\end{align*} 
Using similar bounds as in the fixed $\alpha$ case, one obtains
\[ E_\mu\left[|\mu|^{2q}\1_{|x|\le b_n} + |\mu|^{2q}e^{-(|x|-b_n)^2} \right] \le C(1+b_n^{2q}). \]
Taking the square root and gathering the different bounds obtained concludes the proof.
\end{proof}

\begin{proof}[Proof of Lemma \ref{lem-lb3}] 
First we check the existence of a solution. Set $\zeta_\al=\zeta(\al)$ and $R_\al:= \bar\Phi(\zeta_\al-\zeta_1)/(\al \tmf(\al))$. For $\al\to 0$ we have $\zeta_\al-\zeta_1\to\infty$ so by using $\bar\Phi(u)\asymp \phi(u)/u$ as $u\to\infty$ one gets, treating terms depending on $\zeta_1$ as constants (for fixed $n$ and as $\al\to 0$) and using $\phi(\zeta_\al)\asymp \al \gf(\zeta_\al)$, 
\[ \bar\Phi(\zeta_\al-\zeta_1) \asymp \frac{\phi(\zeta_{\al}-\zeta_1)}{\zeta_\al-\zeta_1} \asymp \frac{\al \gf(\zeta_\al) e^{\zeta_\al\zeta_1}}{\zeta_\al}.\] 
As $\tmf(\al)\asymp \zeta_\al \gf(\zeta_\al)$, one gets $R_\al\asymp e^{\zeta_\al\zeta_1}/\zeta_\al^{\sbl{2}}\to \infty$ as $\al\to 0$. On the other hand, with $\pi_1\le s/n$  and $\al_1\tmf(\al_1)=ds/n$,
\[ R_{\al_1} = \frac{1}{2\al_1\tmf(\al_1)} = \frac{d n}{2s} 
\le \frac{8}{\pi_1} \frac{d}{16}, \]
so that $R_{\al_1}< 8/\pi_1$ as $d<2$. This shows that the equation at stake has at least one solution for $\al$ in the interval $(0,\al_1)$. 

By definition of $\mf_1(\mu,\al)$, for any $\mu$ and $\al$, and $\zeta=\zeta(\al)$,
\begin{align*}
m_1(\mu,\al) & = \int_{-\zeta}^{\zeta} \frac{\bef(x)}{1+\al\bef(x)} \phi(x-\mu)dx\ +\ \int_{|x|>\zeta} \frac{\bef(x)}{1+\al\bef(x)} \phi(x-\mu)dx\\
 & = \qquad \qquad (A) \qquad \qquad \qquad\ +\ \qquad \qquad \qquad  (B).
 \end{align*}
By definition of $\zeta$, the denominator in (B) is bounded from above by $2\al\bef(x)$ so
\[ (B) \ge \frac1{2\al} \int_{|x|>\zeta} \phi(x-\mu)dx \geq \frac12\bef(\zeta)\bar\Phi(\zeta-\mu).\]
One splits the integral (A) in two parts corresponding to $\bef(x)\ge 0$ and $\bef(x)<0$. Let $c$ be the real number such that $\gf/\phi(c)=1$. By construction the part of the integral (A) with $c\le |x|\le \zeta$ is nonnegative, so, for $\al\le |\bef(0)|^{-1}/2$,
\begin{align*}
(A) & \ge \int_{-c}^{c} \frac{\bef(x)}{1+\al\bef(x)} \phi(x-\mu)dx\\
& \ge  -\int_{-c}^c \frac{|\bef(0)|}{1-\al|\bef(0)|} \phi(x-\mu)dx  \\
& \ge -2|\bef(0)| \int_{-c}^c \phi(x-\mu)dx,
\end{align*}
where one uses the monotonicity of $y\to y/(1+\al y)$. For $\mu\ge c$, the  integral $\int_{-c}^c \phi(x-\mu)dx$ is bounded above by $2\int_0^c \phi(x-\mu)dx\le 2c\phi(\mu-c)$. To establish \eqref{m1}, it thus suffices to show that
\[ (i):=4|\bef(0)| c\phi(\zeta_1-c) \le \frac14B(\zeta_3)\bar\Phi(\zeta_3-\zeta_1)=:(ii). \]
The right hand-side equals $2\tmf(\al _3)/\pi_1$ by definition of $\zeta_3$. By Lemma \ref{lem-momgf} and the definition of $\ga$, one has  $\tmf(\al _3)\asymp \zeta_3\gf(\zeta_3)\asymp \zeta_3^{-\delta}$. It is enough to show that $\pi_1^{-1}\zeta_3^{-\delta}$ is larger  than $C\phi(\zeta_1-c)$, for suitably large $C>0$. 

Let us distinguish two cases. In the case $\zeta_3\le 2\zeta_1$, the previous claim is obtained, since $\zeta_1$ goes to infinity with $n/s$ by Lemma \ref{lemz} and $\phi(\zeta_1-c)=o(\zeta_1^{-1})$. 
In the case $\zeta_3> 2\zeta_1$, we obtain an upper bound on $\zeta_3$ by rewriting the equation defining it. For $t\ge 1$, one has $\bar\Phi(t)\ge C\phi(t)/t$. Since $\zeta_3-\zeta_1> \zeta_1$ in the present case, it follows from the equation defining $\zeta_3$ that
\[ C\frac{\phi(\zeta_3-\zeta_1)}{\zeta_3-\zeta_1} \le 8\al_3\tmf(\al_3)/\pi_1. \]
This can be rewritten using $\phi(\zeta_3-\zeta_1)=\sqrt{2\pi}\phi(\zeta_3)\phi(\zeta_1)e^{\zeta_1\zeta_3}$, as well as $\phi(\zeta_3)=\gf(\zeta_3)\al_3/(1+\al_3)\geqa \al_3 \gf(\zeta_3)$ and $\tmf(\al_3)\asymp\zeta_3\gf(\zeta_3)$. This leads to 
\[\frac{e^{\zeta_1\zeta_3}}{\zeta_3^2}\le \frac{C}{\pi_1}e^{\zeta_1^2/2}.\]
By using $e^x/x^2\ge Ce^{x/2}$ for $x\ge 1$ one obtains $\zeta_1^2e^{\zeta_1\zeta_3/2}\le  e^{\zeta_1^2/2}C/\pi_1$, that is, using $\zeta_1^2\ge 1$ and Lemma \ref{lempq},
\[ \zeta_3^\delta  \le C\zeta_1^\delta + C\zeta_1^{-\delta}\log^\delta(C/\pi_1) \]
so that, using that $u\to u\log^{\delta}(1/u)$ is bounded on $(0,1)$, and $\pi_1\le 1 \le \zeta_1$,
\[ \pi_1\zeta_3^\delta \le C(\zeta_1^\delta+1).\] 
 So the previous claim is also obtained in this case, as $\phi(\zeta_1-c)$ is small compared to $(C'\zeta_1)^{-\delta}$ for large $\zeta_1$.

\end{proof}

\subsection{Proof of Lemma \ref{thm: bounds_mmle}} \label{sec: bounds_mmle}

As a first step we prove the upper bound for $\hat\al$. Let us introduce the variables $W_i(\al)=B(X_i,\al)-m_1(\theta_{0,i},\al)$, $A(\al)=-\sum_{i=1}^n m_1(\theta_{0,i},\al)$ and $ V(\al)=\sum_{i=1}^n \text{Var}_{\te_0} W_i(\al)$. Note that $W_i(\al)\leq M(\al)$ for $M(\al)=1/\al$, using Lemma \ref{lem-momgf}. 
We show below that
\begin{align}
A(\tilde{\al}_2)\geq n\tilde{m}({\tilde\al_2})/2;\quad V(\tilde\al_2) \lesssim n \tilde{m}(\tilde\al_2)/\tilde\al_2\label{eq: mmle:help1}
\\
-A(\tilde\al_1)\geq n\tilde{m}({\tilde\al_1});\quad V(\tilde\al_1)\lesssim n\tilde{m}(\tilde\al_1)/\tilde{\al}_1 \label{eq: mmle:help2},
\end{align}
Then since the sign of the score function $\sum_{i=1}^n B(X_i,\alpha)$ determines which side of $\alpha$ the MMLE $\hat\alpha$ lies, Bernstein's inequality leads to
\begin{align*}
P_{\theta_0}(\hat\alpha>\tilde\al_2)&=P_{\theta_0}\big( \sum_{i=1}^n W_i(\tilde\alpha_2) >A(\tilde\al_2)\big)
\leq  \exp\Big\{ \frac{1}{2}\frac{A(\tilde\al_2)^2}{V(\tilde\al_2)+M(\tilde\al_2) A(\tilde\al_2)/3} \Big\}\\
P_{\theta_0}(\hat\al<\tilde\al_1 )&= P_{\theta_0}\Big(-\sum_{i=1}^n W_i(\tilde\alpha_1)  >-A(\tilde\al_1) \Big)\leq \exp\Big\{-\frac{1}{2}\frac{A(\tilde\al_1)^2}{V(\tilde\al_1)-M(\tilde\al_1)A(\tilde\al_1)/3}\Big \},
\end{align*}
which combined with \eqref{eq: mmle:help1}--\eqref{eq: mmle:help2} provides
$P_{\theta_0}(\hat\al>\tilde\al_2 )\leq \exp \{-c n \tilde\al_2\tilde{m}(\tilde\al_2)\}=o(1)$ as well as 
$P_{\theta_0}(\hat\al<\tilde\al_1 ) \leq \exp \{-c n \tilde\al_1\tilde{m}(\tilde\al_1)\}=o(1)$, for some $c>0$.

Let us now prove \eqref{eq: mmle:help1}. Consider the partition of indexes $S_1=\{i:\, |\te_{0,i}|\leq 1/\zeta(\tilde\al_2)\}$, $S_2=\{i:\, 1/\zeta(\tilde\al_2)<|\te_{0,i}|\leq \zeta(\tilde\al_2)/2\}$ and $S_3=(S_1\cup S_2)^c$. By applying twice Lemma \ref{lem: help:adapt:2q} (with $q=2$ and first with $t= 1/\zeta(\tilde\al_2)$ then with $t=\zeta(\tilde{\al}_2)/2$) and using that EB(q) implies EB(2) up to a change in the constants,  the cardinality of the sets $S_2$ and $S_3$ are bounded from above by  $c_0(C_2D_2+1)\tilde{s}\log (n/\tilde{s}) \zeta(\tilde{\al}_2)^2$ and $c_0(4C_2D_2+1)\tilde{s}\log (n/\tilde{s})/ \zeta(\tilde{\al}_2)^2$, respectively, with $C_2,D_2$ defined below \eqref{def: equiv:tilde:s}. In view of  Lemma \ref{lemz} we have $\log (n/\tilde{s}) \sim \zeta(\tilde\al_2)^2$, hence following from the estimates for small and intermediate signals of $m_1(\te_{0,i},\al)$ in Lemma \ref{lem-momgf}, the definition \eqref{def:tilde:al} and the excessive-bias condition with $q=2$,
\begin{align} 
-A(\tilde{\al}_2) &\leq -\sum_{i\in S_1} \tilde{m}({\tilde\al_2})+C \zeta({\tilde\al_2})^{2-\delta} \sum_{i\in S_1}\theta_{0,i}^2+C\sum_{i\in S_2}e^{-\zeta(\tilde\al_2)^2/8}/\tilde\al_2+
 \sum_{i\in S_3}\frac{1}{{\tilde\al_2\wedge c}}\nonumber\\
&\leq -(n-s) \tilde{m}({\tilde\al_2})+ CD_q(\sqrt{2}A)^{2-q} \log (n/\tilde{s}) \zeta({\tilde\al_2})^{2-\delta}  \tilde{s}\nonumber\\
&\qquad+Cc_0(C_2D_2+1) \zeta(\tilde\al_2)^2 \log(n/\tilde{s})\tilde{s} e^{-\zeta(\tilde\al_2)^2/8}/\tilde\al_2
 +c_0(4D_2C_2+1+o(1)) \tilde{s}/{\tilde\al_2}\nonumber\\
&\leq -(1+o(1)) n\tilde{m}({\tilde\al_2})   +d_2c_0(4C_2D_2+1+o(1))n\tilde{m}({\tilde\al_2}) \leq -n\tilde{m}({\tilde\al_2})/2, \nonumber
\end{align}
for small enough choice of the parameter $d_2$ ($d_2<c_0^{-1}(4C_2D_2+1)^{-1}/2$ is sufficiently small), where in the third inequality we used that the second term is a $o( n\tilde{m}(\tilde\al_2) )$, since  it is equal up to a constant multiplier to $(\log^{2-\delta/2}(n/\tilde s) \tilde \al_2)d_2n\tilde{m}(\tilde\al_2)$  and then we substituted $\tilde\al_2=d_2^{-1}\tilde{m}(\al_2)^{-1}\tilde{s}/n=d_2^{-1}\zeta(\tilde\al_2)^{-\delta}(\tilde{s}/n)$. Furthermore, again in view of the bounds on $m_2$ in Lemma \ref{lem-momgf} and \eqref{def:tilde:al},
\begin{align}
 V(\tilde\al_2) & \le \sum_{i\in S_1}  m_2(\theta_{0,i},\tilde\al_2)  
 + \sum_{i\in S_2} m_2(\te_{0,i},\tilde\al_2)
+   \sum_{i\in S_3} m_2(\te_{0,i},\tilde\al_2)\nonumber\\
 & \lesssim  \delta (n-\tilde{s}) \frac{\tilde m(\tilde\al_2)}{\zeta(\tilde\al_2)^2 \tilde\al_2} 
 +\zeta(\tilde\al_2)\log (n/\tilde{s}) \tilde{s} e^{-\zeta(\tilde\al_2)^2/8}/\tilde\al_2^2 + \tilde{s}/\tilde\al_2^2 \nonumber\\
 & \lesssim \frac{1}{\tilde\al_2}\left[ (n-\tilde{s}) \tilde{m}(\tilde\al_2)\zeta(\tilde\al_2)^{-2} +\tilde{s}/\tilde\al_2\right] \lesssim n \tilde{m}(\tilde\al_2)/\tilde\al_2. \nonumber 
\end{align}    
Finally we prove \eqref{eq: mmle:help2}. Note that in view of Lemma \ref{lemz} we have $\zeta(\tilde\al_1)\sim \sqrt{2\log (n/\tilde{s})}$, hence $\zeta(\tilde\al_1)\leq A\sqrt{2\log(n/\tilde{s})}$ (for any $A>1
$ and large enough $n$). Therefore the cardinality of the set $S_4=\{i\in S_0: |\theta_{0,i}|\geq \zeta(\tilde\al_1)\}$  is at least $\tilde{s}$, see Lemma \ref{lem: effective:sparsity}. It follows from Lemma \ref{lem-momgf} that for any $\al$ the map $\mu\to m_1(\mu,\al)$ takes its minimum at $\mu=0$. So  $m_1(\theta_{0,i},\tilde{\al}_1)\ge -\tilde{m}(\tilde{\al}_1)$ 
for every $\theta_{0,i}\in\mathbb{R}$ and $m_1(\theta_{0,i},\tilde{\al}_1)\geq m_{1}(\zeta(\tilde\al_1),\tilde\al_1)$ for every $i\in S_4$. So in view of  Lemma \ref{lem-momgf}, we have $m_1(\theta_{0,i},\tilde{\al}_1)\ge 1/(4\tilde\al_1)$ for large enough $n$, 
\begin{align*}
-A(\tilde\al_1) &\geq -\sum_{i\notin S_4} \tilde{m}({\tilde\al_1})+ \frac{1}{4}\sum_{i\in S_4}\tilde\al_1^{-1} \ge -n \tilde{m}({\tilde\al_1})+(d_1/4)n\tilde{m}({\tilde\al_1})
\ge n\tilde{m}({\tilde\al_1}),
\end{align*}
for large enough choice of the parameter $d_1$ ($d_1>8$ is sufficiently large), and similarly to the bound on $V(\tilde\al_2)$, we have $V(\tilde\al_1)\lesssim n\tilde{m}(\tilde\al_1)/\tilde{\al}_1$.

\subsection{Proof of Theorem \ref{thm: coverage_mmle_q} - diameter bound}\label{sec: coverage_mmle_q-supp}

In view of \eqref{eq: variance}, one can decompose
\begin{align*}
v(\hat\al) & =   \sum_{|X_i|\le t(\hat\al)} a_{\hat\al}(X_i) \omega_q(X_i)
+  \sum_{|X_i|>t(\hat\al)} \big(a_{\hat\al}(X_i)\kappa_{\hat\al,q,i}(X_i) + (1-a_{\hat\al}(X_i))|\hat\te_{\hat\al,i}|^q\big) \nonumber\\
& =: v_1(\hat\al)+v_2(\hat\al).
\end{align*}
We deal with the terms $v_1, v_2$ separately, starting with $v_1(\hat\al)$. Let us partition the set of indexes into $R_1=\{i:\, |\theta_{0,i}|\leq 1/\tau(\tilde\al_2) \}$, $R_2=\{i:\,  1/\tau(\tilde\al_2)< |\theta_{0,i}|\leq \tau(\tilde\al_2)/4 \}$ and $R_3=\{i:\, |\theta_{0,i}|> \tau(\tilde\al_2)/4\}$. Note that $|R_3|\lesssim \tilde{s}$ in view of Lemma \ref{lem: help:adapt:2q} (with $t=\tau(\tilde\al_2)/4$) and $\tau(\tilde\al_2)^2\sim 2\log(n/\tilde{s})$, see Lemma \ref{lemz}.  
Then following \eqref{eq: UB_kappa_x} (for $\mu=0$)
and $t(\tilde\al_1)^2\sim 2\log(n/\tilde{s})$, see Lemma \ref{lemz},
\begin{align*}
\sum_{i\in R_3, |X_i|\leq t(\hat\al)}a_{\hat\al}(X_i)\omega_q(X_i)\lesssim \sum_{i\in R_3, |X_i|\leq t(\tilde\al_1)}(1+|X_i|^q)
\lesssim |R_3|t(\tilde\al_1)^q\lesssim \tilde{s}\log^{q/2}(n/\tilde{s}).
\end{align*}
Next considering the sum defining $v_1(\hat \al)$ restricted to indices that belong to $R_1$,
\begin{align}
\sum_{i\in R_1, |X_i|\leq t(\hat\al)}a_{\hat\al}(X_i)\omega_q(X_i)\leq
 \Big\{\sum_{i\in R_1, |X_i|\leq \tau(\tilde\al_1)}  +\sum_{i\in R_1,\tau(\tilde\al_1)< |X_i|\leq t(\tilde\al_1)} \Big\} a_{\hat\al}(X_i)\omega_q(X_i).\label{eq: hulp03:Lq}
\end{align}
In view of Lemma \ref{thm: bounds_mmle} and assertion \eqref{boundsa}, with probability tending to one, 
\begin{align}
a_{\hat\al}(X_i)\leq 1\wedge \frac{\tilde\al_2}{1-\tilde\al_2}\frac{\gf}{\phi}(X_i). \label{eq: UB:hyper:posterior}
\end{align}
Then by assertion \eqref{eq: UB_kappa_x} used to bound $\omega_q$, Lemma \ref{lem: help:adapt:1q} (with $t=1/\tau(\tilde\al_2)$), Lemma \ref{lem: help:adapt:3q} (with $t_1=1/\tau(\tilde\al_2)$ and $t_2=\tau(\tilde{\al}_1)$), and Lemma \ref{lemz} the following upper bounds, up to constant multipliers, hold for the two terms on the right hand side of \eqref{eq: hulp03:Lq}, for $M_n\to \infty$ slowly,
\begin{align}
\sum_{i\in R_1,  |X_i|\leq\tau(\tilde{\al}_1)}\tilde{\al}_2\frac{\gf}{\phi}(X_i)(1+|X_i|^q)&\leq n\tilde\al_2\tau(\tilde{\al}_1)^{q-\delta}\lesssim \tilde{s} \log^{q/2} (n/\tilde{s}),\nonumber\\
\sum_{i\in R_1,  \tau(\tilde{\al}_1)<|X_i|\leq t(\tilde{\al}_1)}(1+|X_i|^q)&\lesssim n\tau(\tilde\al_1)^{q-1}e^{-\tau(\tilde\al_1)^2/2}+M_n\tau(\tilde\al_1)^{(q+1)/2}(ne^{-\tau(\tilde\al_1)^2/2})^{1/2},\nonumber\\
&\lesssim n\tau(\tilde\al_1)^{q-1}e^{-\tau(\tilde\al_1)^2/2},\label{eq: hulp0004}
\end{align}
with probability tending to one, where the last inequality follows from the argument above \eqref{eq: hulp02q}, with $\tau(\tilde\al_1)$ instead of $t(\tilde\al_2)$, while the first inequality from 
\begin{align*}
\tilde{\al}_2n\tau(\tilde{\al}_2)^{q-\delta}&=d_2\tilde{s} \tau(\tilde{\al}_2)^{q-\delta}/\tilde{m}(\tilde\al_2)\asymp
\tilde{s}\tau(\tilde\al_2)^{q-\delta}\zeta(\tilde\al_2)^\delta\sim \tilde{s} \log^{q/2}(n/\tilde{s}).  
\end{align*}
Finally for $i\in R_2$,  
proceeding as for $i\in R_1$, with probability tending to one
\begin{align}
\sum_{i\in R_2, |X_i|\leq t(\hat\al)}a_{\hat\al}(X_i)\omega_q(X_i)
&\lesssim \tilde\al_2\sum_{i\in R_2, |X_i|\leq \tau(\tilde\al_1)} \frac{\gf}{\phi}(X_i)(|X_i|^q+1)
+\sum_{i\in R_2, \tau(\tilde\al_1)\leq |X_i|\leq t(\tilde\al_1)} (|X_i|^q+1).
\label{eq: UB_var_adapt_term3:Lq}
\end{align}
We deal with the two terms on the right hand side separately, starting with the first one. Note that following from the inequality \eqref{eq: bounds_fun_facts} and the definition of $\tau(\tilde\al_1)$ we have that $\tilde\al_2/(1-\tilde\al_2)\lesssim  \tilde\al_1/(1-\tilde\al_1)=(\phi/\gf)(\tau(\tilde\al_1)).$ Also note that in view of  Lemma \ref{lem: help:adapt:2q} (with $t=1/\tau(\tilde\al_2)$) we have $|R_2|\lesssim \tau(\tilde\al_2)^q \log^{q/2}(n/\tilde{s}) \tilde{s}$. Then in view of Lemma \ref{lem: help:adapt:1q} (with $t=\tau(\tilde\al_2)/4$) and Lemma \ref{lemz} the first term in \eqref{eq: UB_var_adapt_term3:Lq} is bounded from above with overwhelming probability by a constant times 
\begin{align*}
 |R_2|& e^{-\tau(\tilde\al_1)^2/4}\tau(\tilde\al_1)^{q-\delta}\gf^{-1}(\tau(\tilde\al_1) )+  M_n |R_2|^{1/2} e^{-\tau(\tilde\al_1)^2/8}\tau(\tilde\al_1)^{q-\delta-1/2}\gf^{-1}(\tau(\tilde\al_1) )\\
& \lesssim \tilde{s} \log^{q/2}(n/\tilde{s})  e^{-\ta(\tilde\al_1)^2/8}\tau(\tilde\al_1)^{2q+1}=o(\tilde{s} \log^{q/2}(n/\tilde{s})),
\end{align*}
for $M_n\to \infty$ slowly, as $e^{-t^2/8}t^{2q+1}\to 0$ as $t\to\infty$. Next, in view of Lemma \ref{lem: help:adapt:3q} (with $t_1=\tau(\tilde\al_2)/4$ and $t_2=\tau(\tilde\al_1)$) 
the second term on the right hand side of \eqref{eq: UB_var_adapt_term3:Lq} is bounded from above with probability tending to one by a multiple of, for $M_n\to \infty$ slowly,
\begin{align*}
|R_2| e^{-\tau(\tilde\al_1)^2/4}&\tau(\tilde\al_1)^{q-1}+M_n \tau(\tilde\al_1)^{(q+1)/2}(|R_2|e^{-\tau(\tilde\al_1)^2/4})^{1/2}\\
&\lesssim \tilde{s}\log^{q/2} (n/\tilde{s})\tau(\tilde\al_1)^{2q+1}e^{-\tau(\tilde\al_1)^2/8}
=o(\tilde{s} \log^{q/2}(n/\tilde{s})).
\end{align*}
We conclude that with probability tending to one the left hand side of \eqref{eq: UB_var_adapt_term3:Lq} is a 
$o(\tilde{s} \log^{q/2}(n/\tilde{s}))$.

It remained to deal with the term $v_2(\hat\al)$. We apply the partition $R_1,R_2,R_3$ as before. First note that in view of Lemma \ref{lem: help:adapt:3q} (with $t_1=1/t(\tilde\al_2)$ and $t_2= t(\tilde{\al}_2)$), the inequalities $|\hat\theta_{\hat\al,i}|^q\leq |X_i|^q$ and $\kappa_{\hat\al,q,i}(X_i)\lesssim |X_i|^q+|\hat\theta_{\hat\al,i}|^q+1$ we have, for $M_n\to \infty$ slowly enough, 
\begin{align*}
&\sum_{i\in R_1, |X_i|\geq t(\hat\al)}a_{\hat\al}(X_i)\kappa_{\hat\al,q,i}(X_i) + (1-a_{\hat\al}(X_i))|\hat\te_{\hat\al,i}|^q \lesssim \sum_{i\in R_1,  |X_i|\ge  t(\tilde{\al}_2)}(|X_i|^q+1)\\
&\qquad\qquad\lesssim n t(\tilde\al_2)^{q-1}e^{-t(\tilde\al_2)^2/2}+M_n t(\tilde\al_2)^{(q+1)/2} (ne^{-t(\tilde\al_2)^2/2})^{1/2}\lesssim n t(\tilde\al_2)^{q-1}e^{-t(\tilde\al_2)^2/2},
\end{align*}
with probability tending to one, where the last inequality follows from \eqref{eq: hulp02q}. Next we deal with the index set $R_3$, satisfying $|R_3|\lesssim \tilde{s}$. Then note that in view of \eqref{eq: UB_kappa_x} and \eqref{boundedshr}, 
\begin{align*}
\sum_{i\in R_3,t(\hat\al)\leq |X_i|}\hat\al  \kappa_{\hat\al,q,i}(X_i)&\leq \sum_{i\in R_3} [|X_i-\hat\te_{\hat\al,i}|^q+c_3]\lesssim
|R_3|(t(\hat\al)^q+1 )\lesssim \tilde{s} t(\tilde\al_1)^q,
\end{align*}
which is less than a constant times $\tilde{s} \log^{q/2}(n/\tilde{s})$, with probability tending to one, using  Lemma \ref{lemz}.
 Furthermore, using Lemmas \ref{thm: bounds_mmle} and \ref{lemz} again, one gets the same bound for 
\begin{align*} 
\sum_{i\in R_3,t(\hat\al)\leq |X_i|\leq 2 t(\hat\al)}(1-a_{\hat\al}(X_i))|\hat\te_{\hat\al,i}|^q&\leq \sum_{i\in R_3,t(\hat\al)\leq |X_i|\leq 2 t(\hat\al)} |X_i|^q
\lesssim t(\hat\al)^q |R_3|\leq t(\tilde\al_1)^q \tilde{s},
\end{align*}
again with probability tending to one. Besides following from $\phi(x)x^{q+1+\delta}\lesssim \phi(t )t^{q+1+\delta}$ for $x\ge t\ge 1$, assertion \eqref{eq: bounds_fun_facts}, the inequality $1-a_\alpha(X_i)\leq \al^{-1}\phi(X_i)/\gf(X_i)$ and Lemmas \ref{lem-momgf} and \ref{lemzetaall}, we have 
\begin{align*}
\sum_{i\in R_3, |X_i|\geq 2 t(\hat\al)}(1-a_{\hat\al}(X_i))|\hat\te_{\hat\al,i}|^q
\leq \sum_{i\in R_3, |X_i|\geq 2 t(\hat\al)} \frac{\phi(X_i)|X_i|^q}{\hat\al \gf(X_i)} 
\lesssim  \sum_{i\in R_3, |X_i|\geq 2 t(\tilde\al_2)} \frac{\phi(X_i) |X_i|^{q+1+\delta}}{\tilde{\al}_1}\\
\lesssim \tilde{\al}_1^{-1} t(\tilde\al_2)^{q+1+\delta} e^{-2t(\tilde\al_2)^2}|R_3|
\lesssim  (\tilde{\al}_2/\tilde{\al}_1)t(\tilde\al_2)^{q+1+\delta} e^{-3t(\tilde\al_2)^2/2}\tilde{s}=o(\tilde{s}).
\end{align*}
Hence it remained to deal with the set $R_2$. Recall that $|R_2|\lesssim \tau(\tilde\al_2)^{q/2} \log^q(n/\tilde{s}) \tilde{s}$. 
With $|\hat\theta_{\hat\al,i}|^q\leq |X_i|^q$,  $\kappa_{\hat\al,q,i}(X_i)\lesssim |X_i|^q+1$, and $t(\tilde\al_2)\leq t(\hat\al_n)$ by \eqref{eq: UB_kappa_x}, \eqref{boundedshr}, and Lemma \ref{thm: bounds_mmle},
\begin{align*}
\sum_{i\in R_2, |X_i|\geq t(\hat\al)}a_{\hat\al}(X_i)\kappa_{\hat\al,q,i}(X_i) + (1-a_{\hat\al}(X_i))|\hat\te_{\hat\al,i}|^q 
&\lesssim  \sum_{i\in R_2, |X_i|\geq t(\tilde\al_2)} (|X_i|^q+1).
\end{align*}
Then by Lemma \ref{lem: help:adapt:3q} (with $t_1=\tau(\tilde\al_2)/4$ and $t_2=t(\tilde\al_2)$), Lemma \ref{lemz} and the inequality $\tau(\al)\leq t(\al)$ the right hand side of the preceding display is further bounded from above by constant times
\begin{align*}
|R_2| t(\tilde\al_2)^{q-1}e^{-t(\tilde\al_2)^2/4} +M_n t(\tilde\al_2)^{(q+1)/2} (|R_2|e^{-t(\tilde\al_2)^2/4})^{1/2} =o(\tilde{s} \log^{q/2}(n/\tilde{s})),
\end{align*}
 with probability tending to one.
So with probability tending to one $v_2(\hat\al)\lesssim \tilde{s}\log^{q/2} (n/\tilde{s})$, which together with  \eqref{eq: UB:extra term:t} (and the same argument for $\tau(\tilde\al_1)$) results in the upper bound
\begin{align*}
v(\hat\al)\lesssim  \tilde{s}\log^{q/2} (n/\tilde{s})+n t(\tilde\al_2)^{q-1}e^{-t(\tilde\al_2)^2/2}+n\tau(\tilde\al_1)^{q-1}e^{-\tau(\tilde\al_1)^2/2}\lesssim  \tilde{s}\log^{q/2} (n/\tilde{s}).
\end{align*}
Noting that in view of Lemma \ref{lem: effective:sparsity} we have $\tilde{s}\leq \tilde{s}_{q}$, this concludes the proof of Theorem \ref{thm: coverage_mmle_q}.

\subsection{Proof of technical lemmas for credible sets }\label{sec:proofs:technical:credible}

\begin{proof}[Proof of Lemma \ref{lem: effective:sparsity}] 
First of all note that if \eqref{condition: EB_q} is satisfied for some integer $\ell$ then \eqref{condition: EB} is also satisfied for the same integer with the parameters given in the lemma. Therefore it immediately holds that $\tilde{s}\leq \tilde{s}_{q}$. For the other direction let us assume that there are two integers $0<\ell<\ell'\leq s$ satisfying condition \eqref{condition: EB}. Denoting $\bar{s}(\ell):=\big|\{i:\, |\theta_{0,i}|\geq A\sqrt{2\log(n/\ell)}\}\big|$, 
\begin{align}
\ell'/C_q\leq \bar{s}(\ell')&= \bar{s}(\ell)+\big|\big\{i:\, A\sqrt{2\log(n/\ell)}> |\theta_{0,i}|\geq A\sqrt{2\log(n/\ell')}\big\}\big|\nonumber\\
& \leq \bar{s}(\ell)+ \ell \frac{D_q(\sqrt{2}A)^{2-q} \log (n/\ell)}{2A^2\log (n/\ell')}.\label{eq:help:eff:sparsity}
\end{align}
Let us distinguish two cases. For $\ell'\leq 2C_q D_q(\sqrt{2}A)^{2-q}A^{-2} \ell$  we have that $\log(n/\ell')\geq\log(n/\ell)-\log(2C_q D_q(\sqrt{2}A)^{2-q}A^{-2})$, which is larger than $\log(n/\ell)/2$, for every $\ell<s\leq c_1 n$ for sufficiently small $c_1>0$. Therefore the right hand side of \eqref{eq:help:eff:sparsity} is bounded from above by
$$\bar{s}(\ell)+\ell D_q(\sqrt{2}A)^{2-q}/A^2\leq (1+C_qD_q(\sqrt{2}A)^{2-q}/A^2)\bar{s}(\ell),$$
hence $\bar{s}(\ell')\lesssim \bar{s}(\ell)$ follows.
Next consider the case $\ell'>2C_q D_q(\sqrt{2}A)^{2-q} A^{-2} \ell$. Note that the function $f(x)=x\log(n/x)$ is monotone increasing for $x\leq n/e$, hence
\begin{align*}
\ell'&\geq 2C_qD_q(\sqrt{2}A)^{2-q}A^{-2}\ell \frac{\log \big(A^2n/(2C_qD_q(\sqrt{2}A)^{2-q}\ell) \big)}{\log(n/\ell')}\\
&= 2C_qD_q(\sqrt{2}A)^{2-q}A^{-2}\ell \Big(\frac{\log(n/\ell)-\log(2C_qD_q(\sqrt{2}A)^{2-q}A^{-2}) }{\log(n/\ell')}\Big)\\
&\geq C_qD_q(\sqrt{2}A)^{2-q}A^{-2}\ell \frac{\log(n/\ell) }{\log(n/\ell')}.
\end{align*}
Therefore the right hand side of \eqref{eq:help:eff:sparsity} is bounded from above by $$\bar{s}(\ell)+\ell' /(2C_q)\leq \bar{s}(\ell)+\bar{s}(\ell')/2, $$
hence $\bar{s}(\ell')\leq 2\bar{s}(\ell)$, concluding the proof of the first statement.

In view of the inequality $c_0c_1n\geq c_0s\geq c_0\tilde{s}\geq \tilde{s}_{q}\geq \ell /C_q$, for all $\ell$ satisfying \eqref{condition: EB_q}, we have for any $A>A'>1$ that $A\{2\log (n/\ell)\}^{1/2}\geq A'\{2\log(n/\tilde{s})\}^{1/2}$ for sufficiently small $c_1>0$. Then
\begin{align*}
\sum_{|\theta_{0,i}|\leq A'\sqrt{2\log(n/\tilde{s})}} |\theta_{0,i}|^q\leq \sum_{|\theta_{0,i}|\leq A\sqrt{2\log(n/\ell)}} |\theta_{0,i}|^q\leq D_q \ell \log^{q/2}(n/\ell)\leq D_qC_qc_0 \tilde{s}\log^{q/2}(n/\tilde{s}).
\end{align*}
Finally, for the third statement note that
$$\big|\big\{i:\, |\theta_{0,i}| \geq A'\sqrt{2\log(n/\tilde{s})}\big\}\big|\geq \big|\big\{i:\, |\theta_{0,i}| \geq A\sqrt{2\log (n/\ell)}\big\}\big|=\tilde{s}_{q}\geq \tilde{s}. \qedhere $$
\end{proof}

\begin{proof}[Proof of Lemma \ref{lemz}]
From the definition of $\al_1$, setting $\zeta_1=\zeta(\al_1)$, 
one has $\al_1^{-1}\asymp (n/s)\zeta_1 g(\zeta_1)\asymp (n/s)\zeta_1^{-\delta}$ using Lemma \ref{lem-momgf}. Inserting this in the  estimates of Lemma \ref{lemzetaall} gives the first part of the result. The result for the tilda versions follows by using the definitions of $\tilde{\al}_i$ as $d_i n\tilde{m}(\tilde\al_i)\tilde\al_i=\tilde s$ and using the same argument. 

\end{proof}

\begin{proof}[Proof of Lemma \ref{lem-lbv}]
We distinguish two regimes: large and small $x-\mu$. By symmetry one can always assume $x-\mu \ge0$. Set $h_x(u):=\phi(x-u)\gamma(u)$. 
 Suppose $x-\mu\ge 10$ and set $A=5$. By definition, and using that $x-\mu-A\ge (x-\mu)/2$,  
\begin{align*}
\int |u-\mu|^q \ga_x(u) du& \ge \int_{x-A}^{x+A} |u-\mu|^q \frac{h_x(u)}{\gf(x)} du 
  \ge 2^{-q}(x-\mu)^q \int_{x-A}^{x+A} \frac{\ga(u)}{\gf(x)}\phi(x-u)du \\
 &\ge 2^{1-q}(x-\mu)^qA \phi(A)\frac{\ga(x+A)}{\gf(x)}.
\end{align*}
As $(\log \ga)'$ is bounded on $\RR$, one has $\ga(x+A)\ge \ga(x)e^{-cA}$ for some $c>0$. Also,  by Lemma \ref{lem-gf}, $\gf(x)\asymp \ga(x)$, so for any real $x$, we have $\gf(x)\le C\ga(x)$ for some $C>0$. From this and the previous bound, one deduces that $\int |u-\mu|^q \ga_x(u) du \geqa (x-\mu)^q$ for $x-\mu\ge 10$. 

Suppose $0\le x-\mu\le 10$. By writing  $\phi(x-u)/\phi(\mu-u)=
\exp((\mu-x)(\mu+x-2u)/2)$, 
 \begin{align*}
\int |u-\mu|^q \ga_x(u) du& \ge \gf(x)^{-1} \int_{\mu+1}^{\mu+2} |u-\mu|^q \frac{\phi(x-u)}{\phi(\mu-u)}\ga(u)\phi(\mu-u) du \\
 & \ge  e^{-40}  \int_{\mu+1}^{\mu+2} \frac{\ga(u)}{\gf(x)}\phi(\mu-u) du. 
\end{align*}
As noted above, $\gf(x)\asymp \ga(x)$, which implies that $\ga(u)/\gf(x)=(\ga(u)/\ga(x))(\ga(x)/\gf(x))\asymp 1$, since $\ga(u)/\ga(x)\asymp1$ for $0\le x-\mu\le 10$, so the previous display is bounded from below by a constant times $\int_1^2\phi(-u) du$, so that for some constant $d>0$, we have $\int |u-\mu|^q \ga_x(u) du \geqa d$ when  $0\le x-\mu\le 10$. 
Putting the  two previous bounds together gives the result.
\end{proof}

\begin{proof}[Proof of Lemma \ref{lem: help:adapt:1q}]
Let us denote the left  hand side of $\eqref{eq: lem:help1:Lq}$ by $G$. Then \begin{align*}
E_{\theta_0}G&\leq\sum_{i\in R_t} \int_{-\ta(\al)}^{\ta(\al)} \frac{\gf(x)}{\phi(x)}(|x|^q+1)\phi(x-\theta_{0,i})dx\\
&\leq 2|R_t|  \int_{0}^{\ta(\al)} e^{tx} (|x|^q+1) \gf(x)dx\leq 2C_{q,\delta}|R_t| e^{t\tau(\al)} \ta(\al)^{q-\delta},
\end{align*}
using $g(x)\propto (1+|x|)^{-1-\delta}$ with the definition of $\ga$ in \eqref{dens-hv}, Lemma \ref{lem-gf} and $\delta<q$.
  Next the variance of $G$ is bounded by, using again $\delta<q$,
\begin{align*}
\text{Var}_{\theta_0}G&\leq \sum_{i\in R_t} 4\int_{0}^{\ta(\al)} \frac{\gf(x)^2}{\phi(x)^2}(1+x^{2q})\phi(x-\theta_{0,i})dx\\
&\leq 4|R_t| e^{t\tau(\al)} \int_{0}^{\ta(\al)} \phi(x)^{-1}(1+ x^{2q}) \gf(x)^2dx
\leq C_{q,\delta}|R_t| e^{t\tau(\al)+\tau(\al)^2/2} \ta(\al)^{2q-1-2\delta}.
\end{align*} 
Then the statements of the lemma follow by Chebyshev's inequality.
\end{proof}

\begin{proof}[Proof of Lemma \ref{lem: help:adapt:3q}]
First note that in view of Lemma \ref{lempq}
\begin{align*}
\sum_{i\in R_{t_1}} (|X_i|^q+1)\1_{|X_i|>t_2}&\leq 2 \sum_{i\in R_{t_1}} |X_i|^q \1_{|X_i|>t_2}\\
&\leq 2(2^{q-1}\vee1) \sum_{i\in R_{t_1}}(t_1^q+  |\eps_i|^q) \1_{|\eps_i|>t_2-t_1}.
\end{align*}
Then in view of Lemma \ref{lemphi} and the inequality $\int_t^\infty\phi(x)dx\leq \phi(t)/t$, the expected value of the left hand side of the preceding display is bounded from above by a multiple of 
\begin{align*}
E_{\theta_0}\sum_{i\in R_{t_1}} (t_1^q+  |\eps_i|^q)  \1_{|\eps_i|>t_2-t_1}&=2\sum_{i\in R_{t_1}}\int_{t_2-t_1}^{\infty}(t_1^q+x^q) \phi(x)dx\\
&\leq 2(\bar{c}/\sqrt{2\pi}) |R_{t_1}|(t_2-t_1)^{q-1}e^{-(t_2-t_1)^2/2}+t_2^q|R_{t_1}|\int_{t_2-t_1}^{\infty} \phi(x)dx\\
&\leq \frac{4\bar{c}+2}{\sqrt{2\pi}} |R_{t_1}|t_2^{q-1} e^{-(t_2-t_1)^2/2},
\end{align*}
while the variance is bounded by
\begin{align*}
\text{Var}_{\theta_0}\sum_{i\in R_{t_1}}  (t_1^q+|\eps_i|^q) \1_{|\eps_i|>t_2-t_1}&\leq 2 \sum_{i\in R_{t_1}}\int_{t_2-t_1}^{\infty}(t_1^q+x^q)^2 \phi(x)dx\\
&\lesssim |R_{t_1}|t_2^{2q-1} e^{-(t_2-t_1)^2/2}.
\end{align*}
Therefore by Chebyshev's inequality we have with overwhelming probability that
\begin{align*} 
\sum_{i\in R_{t_1}} (|X_i|^q+1)\1_{|X_i|>t_2}&\leq \tilde{c}_0 |R_{t_1}|t_2^{q-1} e^{-(t_2-t_1)^2/2}+ M_n t_2^{q-1/2} (|R|e^{-(t_2-t_1)^2/2})^{1/2},
\end{align*}
with arbitrary $M_n\rightarrow\infty$ and $\tilde{c}_0= (2^{q}\vee2) (4 \bar{c}+2)/\sqrt{2\pi}$ where $\bar{c}=(2-q)\vee 1$. 
\end{proof}

\begin{proof}[Proof of Lemma \ref{lem: help:adapt:2q}]
From the monotone increasing property of $\ell\mapsto \ell\log(n/\ell)$ (for $\ell\leq s=o(n)$) and $\ell\leq C_q\tilde{s}_{q}\leq C_qc_0 \tilde{s}$ (see Lemma \ref{lem: effective:sparsity} with $C_q,c_0\geq 1$) follows
\begin{align*}
|R_t|&\leq t^{-q}\sum_{i\in R_t: |\theta_{0,i}|\leq A\sqrt{2\log (n/\ell)}}|\theta_{0,i}|^q + |\{i:\,|\theta_{0,i}|\geq A\sqrt{2\log (n/\ell)} \}|\\
& \leq D_qC_q c_0  \tilde{s} t^{-q}\log^{q/2} (n/\{C_q c_0 \tilde{s}\})+\tilde{s}_{q}\leq c_0(C_qD_qt^{-q}\log^{q/2} (n/\{C_q c_0 \tilde{s}\})+1)\tilde{s}\\
&\leq D_qC_q c_0  \tilde{s} t^{-q}\log^{q/2} (n/\tilde{s})+\tilde{s}_{q}\leq c_0(C_qD_qt^{-q}\log^{q/2} (n/\tilde{s})+1)\tilde{s}. \qquad \qedhere
\end{align*}
\end{proof}

\subsection{Proof of Remarks \ref{rem: cred:alternative} and \ref{rem: equiv:adapt}}\label{sec: cred:alternative}

First we deal with assertion \eqref{eq: equiv:nonadapt}. We follow below similar lines of reasoning as in the proof of Lemma 4.1 of \cite{vsvuq17}. 

Let us introduce the shorthand notation $W=d_q(\theta,\hat\theta_\alpha)$ and $sd(W\given X,\al)$ for the standard deviation of a variable of distribution $\cL(W\given X,\al)$. Note that $v_{q,\alpha}(X)=E(W|X,\alpha)$. Then in view of Chebyshev's inequality
\begin{align*}
&\Pi_\alpha \big(\theta:\, W>  E(W|X,\alpha)+ \beta^{-1/2} sd(W|X,\alpha) |X\big)<\beta,
\end{align*}
which implies $r_\beta\leq E(W|X,\alpha)+ \beta^{-1/2} sd(W|X,\alpha)$. Similary, by applying again Chebyshev's inequality, we get that
\begin{align*}
\Pi_\alpha& \big(\theta:\, W<  E(W|X,\alpha)- (1-\beta)^{-1/2} sd(W|X,\alpha) |X\big)= \\
&\Pi_\alpha \big(\theta:\, -(W- E(W|X,\alpha) )>  (1-\beta)^{-1/2} sd(W|X,\alpha) |X\big)< 1-\beta,
\end{align*}
resulting in the lower bound $E(W|X,\alpha)- (1-\beta)^{-1/2} sd(W|X,\alpha) \leq r_{\beta}$.

Hence it is sufficient to show that 
\begin{align}
sd(W|X,\alpha)=o\big(E(W|X,\alpha)\big),\label{eq: std:dev:UB}
\end{align}
for $(\log_2 n)^{\delta/2}/n\ll \alpha\leq \alpha_1$, for some sufficiently small $\alpha_1>0$. Recalling that for independent random variables the variance of the sum is equal to the sum of the variances we get that
\begin{align*}
var(W|X,\al)\leq \sum_{i=1}^n \int |\theta_i-\hat\theta_{\alpha,i}|^{2q}d\Pi(\theta_i|X_i)=v_{2q,\al}.
\end{align*}
Then, similarly to \eqref{eq: variance},
\begin{align*}
v_{2q,\al}= &  \sum_{i\in S_0: |X_i|\le t(\al)} a(X_i) \omega_{2q}(X_i)
+  \sum_{i\in S_0: |X_i|>t(\al)} \left[a(X_i)\kappa_{\al,2q,i}(X_i) + (1-a(X_i))|\hat\te_{\al,i}|^{2q}\right]\nonumber\\
 & + \sum_{i\notin S_0: |\veps_i|\le t(\al)} a(\veps_i) \omega_{2q}(\veps_i)
+  \sum_{i\notin S_0: |\veps_i|>t(\al)} \left[ a(\veps_i)\kappa_{\al,2q,i}(\veps_i) + (1-a(\veps_i))
|\hat\te_{\al,i}|^{2q}\right] \nonumber\\
& =: u_1+u_2+u_3+u_4.
\end{align*}
Since in view of Lemma \ref{lem: momga2} the upper bounds for $v_3$ and $v_4$ in Step 4 of Section \ref{sec: coverage:nonadapt} go through with $q\in(0,4]$, we get from \eqref{eq:coverage:LB:var0} that
$$u_3+u_4\lesssim \alpha n \tau(\alpha)^{2q-\delta}\lesssim \tau(\al)^q v_{q,\alpha}.$$
Next, following from Lemma \ref{lem: momga2}, Lemma \ref{lem-lbv}, and the inequality $(1+|x|^{2q})\leq (1+|x|^q)^2$,
\begin{align*}
u_1&\leq  \sum_{i\in S_0: |X_i|\le t(\al)} a(X_i) \omega_{2q}(X_i)\lesssim \sum_{i\in S_0: |X_i|\le t(\al)} a(X_i) (1+|X_i|^{2q})\\
&\lesssim \sum_{i\in S_0: |X_i|\le t(\al)} a(X_i) \omega_{q}(X_i) (1+|X_i|^{q})\lesssim v_1 t(\al)^q< v_{q,\al}t(\al)^q.
\end{align*}
Finally, it remained to deal with the term $u_2$. First note that in view of Lemma \ref{lem: momga2} (with $\mu=X_i, x=X_i$),  Lemma \ref{lem-lbv}, and assertion \eqref{boundedshr} (which implies $|X_i-\hat\te_{\al,i}|\le t(\al)$), $\kappa_{\al,2q,i}(X_i) \lesssim |\hat\te_{\al,i}-X_i|^{2q}+1\lesssim (|\hat\te_{\al,i}-X_i|^{q}+1)^2\lesssim \kappa_{\al,q,i}(X_i)t(\alpha)^{q}$, which together with $|\hat\te_{\al,i}|\leq |X_i|$ implies that
\begin{align*}
& \sum_{i\in S_0:t(\al)\leq |X_i|\leq 2t(\al)} \left[a(X_i)\kappa_{\al,2q,i}(X_i) + (1-a(X_i))|\hat\te_{\al,i}|^{2q}\right]\\
&\qquad\qquad\lesssim  \sum_{i\in S_0:t(\al)\leq |X_i|\leq 2t(\al)} \left[a(X_i)\kappa_{\al,q,i}(X_i)t(\alpha)^{q} + (1-a(X_i))|\hat\te_{\al,i}|^{q}|X_i|^q\right]\\
&\qquad\qquad \leq v_2 t(\al)^q< v_{q,\al}t(\al)^q.
\end{align*}
Furthermore, in view of the definition of $a(X_i)$, for $0<\alpha\leq \alpha_1$ (with some sufficiently small  $\alpha_1>0$) and $|X_i|\geq 2 t(\al)$,
\begin{align*}
(1-a(X_i))|X_i|^{2q}\leq \frac{\phi(X_i)}{\alpha g(X_i)}|X_i|^{2q}\lesssim \alpha^3 t(\alpha)^{2q+1+\delta}.
\end{align*}
Then similarly as above
\begin{align*}
& \sum_{i\in S_0: |X_i|>2t(\al)} \left[a(X_i)\kappa_{\al,2q,i}(X_i) + (1-a(X_i))|\hat\te_{\al,i}|^{2q}\right]\\
&\qquad\qquad\lesssim  \sum_{i\in S_0:|X_i|>2t(\al)}\big(a(X_i)\kappa_{\al,q,i}(X_i)t(\alpha)^{q} + \alpha^3 t(\alpha)^{2q+1+\delta} \big)\\
&\qquad\qquad \lesssim v_2t(\al)^q +o(n\alpha)\leq v_{q,\alpha}t(\al)^q.
\end{align*}
Combining the above inequalities and noting that $v_{q,\alpha}\gtrsim n\alpha \tau(\al)^{q-\delta}\gg t(\al)^q$, for $\al\gg (\log_2 n)^{\delta/2}/n$ we get that 
$$v_{2q,\alpha}\lesssim t(\alpha)^q v_{q,\alpha}=o(v_{q,\alpha}^2),$$
finishing the proof of assertion \eqref{eq: equiv:nonadapt}.

Next we prove Remark \ref{rem: equiv:adapt}. Note that $\tilde{\al}_1=d_1 \tilde{s}/(n \tilde{m}(\tilde{\al}_1))\asymp \tilde{s}\zeta(\tilde\al_1)^\delta/n\gg (\log n)^{\delta/2}/n$ as $\tilde{s}\rightarrow\infty$ under the excessive bias assumption, see \eqref{condition: EB_q}. Hence in view of assertion \eqref{assump: bounds} combined with $\tilde{\al}_2=o(1)$ (see \eqref{eq: bounds_fun_facts}) 
we get that $(\log n)^{\delta/2}/n\ll \hat{\al}=o(1)$ with probability tending to one. The proof concludes by noting that the two types of credible sets are equivalent for $(\log_2 n)^{\delta/2}/n\ll \alpha\leq \alpha_1$, as proven above.

\begin{lem}\label{lem: momga2}
For any $q\in(0,2]$ and $\mu,x\in\mathbb{R}$
\begin{align*}
\int |\mu-u|^{2q}\gamma_x(u)du\leq C\big[ |x-\mu|^{2q}+1\big].
\end{align*}
\end{lem}
\begin{proof}
The proof follows the same line of reasoning as Lemma \ref{lemmomga}, nevertheless for completeness we provide the details below.
 Recall $\ga_x(u) = \ga(u)\phi(x-u)/\gf(x)$ and by elementaty computations we get that
\begin{align*}
\int (u-x) \ga_x(u) du&=g'(x)/g(x),\\
\int (u-x)^2 \ga_x(u) du&=\gf^{(2)}(x)/\gf(x)+1,\\
\int (u-x)^3 \ga_x(u) du&=-\gf^{(3)}(x)/\gf(x)+3g'(x)/g(x),\\
\int (u-x)^4 \ga_x(u) du&=\gf^{(4)}(x)/\gf(x)+6g^{(2)}(x)/g(x)-3.
\end{align*}
Then by using the equation $u^4=(u-x)^4+4x(u-x)^3+6x^2(u-x)^2+4x^3(u-x)+x^4$ we get that
\begin{align*}
\int u^4\gamma_x(u)du&=\frac{g^{(4)}}{g}(x)-6\frac{g^{(2)}}{g}(x)+3+4x\Big(-\frac{g^{(3)}}{g}(x)+3 \frac{g'}{g}(x)\Big)\\
&\qquad\qquad+ 6x^2\Big(\frac{g^{(2)}}{g}(x)+1 \Big)+4x^3\frac{g'}{g}(x)+x^4.
\end{align*}
Noting that by the definition of $\gamma(u)$, see \eqref{dens-hv}, $|g'|\leq c_1 g$, $|g^{(2)}|\leq c_2 g$, $|g^{(3)}|\leq c_3 g$, $|g^{(4)}|\leq c_4 g$, for somce $c_1,c_2,c_3,c_4>0$ we get by similar computations as above that there exist positive constants $C_1,C_2,C_3,C_4,C_5,C_6$ such that
\begin{align*}
\int(u-\mu)^4\gamma_x(u)du&\leq (x-\mu)^4+C_1 |x-\mu|^3+C_2 (x-\mu)^2+C_3 |x-\mu|+C_4\\
&\leq C_5(x-\mu)^4+C_6.
\end{align*} 
Using H\"older's inequality and manipulating the $q$th powers via \eqref{manipq} leads to
\begin{equation*}%\label{eq: UB_kappa_x} 
\int |u-\mu|^{2q} \ga_x(u)du \le \Big[\int (u-\mu)^4 \ga_x(u)du\Big]^{q/2}\le 
C\left[ |x-\mu|^{2q} + 1\right].
\end{equation*}
\end{proof}

\subsection{Notes on posterior median and mean} \label{supp-pointest}

{\em Minimaxity of the posterior median, $0<q\le 2$.} Although we do not use it in the proof of the main results of the paper (only in the discussion previous to Theorems \ref{thm: coverage_mmle} and \ref{thm: coverage_mmle_q} to mention global minimaxity for the diameter of the considered credible sets), one can check that the posterior median is rate-minimax under the assumptions of Theorem \ref{thm-risk-dq}. We now briefly sketch the argument. For sub-Cauchy tails of the slab density (including Cauchy tails and any density of the form \eqref{dens-hv} with $\delta\ge 1$), this is a direct consequence of Theorem 1 in \cite{js04}, as the conditions of that statement are satisfied for these slabs. For heavier slab tails, one can still follow the steps in the proof of Theorem 1 in \cite{js04}. The only changes concern the slightly different moment estimates for $\tilde m$ and $m_1$ (which in turn give a slightly different behaviour of $\hat\al$, depending on the slab tails). The risk bounds from Section 6 of \cite{js04} do not depend on slab tails and thus remain unchanged. As the corresponding risk bounds are always equal or below the risk bounds for the full posterior in Lemma \ref{lemmomga}, they eventually lead to the posterior median converging at minimax rate, in a similar way as for the proof of Theorem \ref{thm-risk-dq} for the full posterior. \\

{\em Suboptimality of the posterior mean for $q<1$.} Johnstone and Silverman \cite{js04}, Section 10, showed that the empirical Bayes posterior mean has a suboptimal convergence rate when $q<1$ and when $\ga$ is the Laplace slab. 
We next show that their argument extends to slabs $\ga$ as in \eqref{dens-hv}.

The argument in \cite{js04}, Section 10, only depends on the behaviour of $g(\sqrt{2\log{n}})$ and the fact that the function $\tilde\mu_1$ they introduce on p. 1623 is strictly increasing, verifies  $\tilde\mu_1(x)\ge x-\Lambda$ as well as $\tilde\mu_1'(0)>0$. The monotonicity of $\tilde\mu_1$ and the first inequality directly follow from the fact that $(\log g)'=g'/g$ is bounded on $\RR$ (this is true for the slabs we consider, see the proof of Lemma 5), while 
\[ \tilde\mu_1'(0)=\frac{\int u^2 \phi(x-u)\ga(u)du}{\int  \phi(x-u)\ga(u)du}>0. \]
So, reproducing their argument with the slab $\ga$ being one of those considered in the present paper, one gets, for $\te_0=0$, and any $0< q\le 2$, as on the first equation of p. 1648 of \cite{js04}, if $\bar\te(\hat\al)$ is the posterior mean at stake corresponding to a slab $\ga$,
\[ Ed_q(\bar\te(\hat\al),\te)\ge \sum_{i=1}^n Cn^{-q}g(\sqrt{2\log{n}})^{-q}=
Cn^{1-q}g(\sqrt{2\log{n}})^{-q}. \]
From this we conclude that for any $q<1$ the $d_q$-risk is at least $n^d$ with $d=1-q>0$ so is above the minimax risk as soon as $s=o(n^d)$. \\

{\em Posterior mean, the case $q=1$.} 
In \cite{js04}, it is shown that when $q=1$, the Laplace slab is suboptimal by using the argument above, as $g(\sqrt{2\log{n}})\asymp\exp(\sqrt{2\log{n}})$ gives a suboptimal rate for very small sparsities $s$, as the minimax rate is $s\log^{1/2}(n/s)$ when $q=1$. In fact, the posterior mean for $q=1$ is also suboptimal for any slab density as in \eqref{dens-hv} and very small sparsities $s$ (e.g $s=\log\log{n}$), by the same argument, see the paragraphs above. As a byproduct, this shows in passing that some slight suboptimality (by $\log{n}$ factors) {\em must} occur for convergence of the full posterior as in Theorem \ref{thm-risk-dq} if $q=1$ and $s$ is of smaller order than, say, $\log{n}$ (which is excluded by condition \eqref{techsn}). Indeed, by the convexity argument given below the statement of Theorem \ref{thm-risk-dq}, if the full posterior was minimax in this regime, it would automatically imply the same rate on the posterior mean. To handle very small $s$, i.e. $s=o(\log^2{n})$ without loss of a logarithmic factor, one could use the modification proposed by Johnstone and Silverman \cite{js04} in their Section 4 (Theorem 2 in \cite{js04} and surrounding paragraphs) to handle the very sparse case. 
%In particular, the modification of the estimator used by Johnstone and Silverman \cite{js04} in their Section 4 (Theorem 2 in \cite{js04} and surrounding paragraphs) to handle sparsities $s_n=o(\log^2{n})$ {\em must} be used if one wants to get the whole possible spectrum of sparsities $s$ such that $s_n\to\infty$ and $s_n=o(n)$. This is interesting in itself, as the necessity of the modification in \cite{js04} was not proved there. \\

Remarkably,  although the posterior mean is suboptimal if $q< 1$, the full posterior distribution itself does converge at optimal minimax rate for any $q\in(2\delta,2]$ (and $\delta$ can be taken as small as desired provided the slab tails are heavy enough) as follows from Theorem \ref{thm-risk-dq} of the paper. This phenomenon was first observed for the hierarchical spike-and-slab by \cite{cv12}, see the paragraphs below Theorem 2.2 and Remark 2.3 in that paper.

\bibliographystyle{apalike}
\bibliography{spacs}

\end{document}